\documentclass[12pt,reqno]{amsart}

\usepackage{enumerate}
\usepackage{enumitem}
\usepackage{amsxtra}
\usepackage{amsmath}
\usepackage{amscd}
\usepackage{amssymb}
\usepackage{amsfonts}
\usepackage[all]{xy}
\usepackage{amsthm} 
\usepackage{color}
\usepackage{upgreek}
\usepackage{latexsym}
\usepackage[english]{babel} 
\usepackage{hyperref}
\usepackage{nicefrac}
\usepackage{euscript}
\usepackage[foot]{amsaddr}
\usepackage{amscd}
\usepackage[sort]{cite}           % nicer citations
\usepackage{fancyhdr}          % nicer headers
\usepackage{amsfonts}
\usepackage{color}
\usepackage{upgreek}
\usepackage{graphicx}
\usepackage{tikz}
\usepackage{hyperref}
\usepackage{stmaryrd}
\usepackage{nicefrac}
\usepackage{wasysym}

\xyoption{all}

\theoremstyle{plain}
\newtheorem{theorem}{Theorem}[section]
\newtheorem{lemma}[theorem]{Lemma}
\newtheorem{lemma-definition}[theorem]{Lemma-Definition}
\newtheorem{definition-lemma}[theorem]{Definition-Lemma}

\newtheorem{proposition}[theorem]{Proposition}

\newtheorem{corollary}[theorem]{Corollary}

\newtheorem*{theorem-A}{Theorem A}
\newtheorem*{theorem-B}{Theorem B}
\newtheorem*{theorem-C}{Theorem C}
\newtheorem*{theorem-D}{Theorem D}
\newtheorem*{conjecture-A}{Conjecture A}
\newtheorem*{conjecture-B}{Conjecture B}
\newtheorem*{conjecture-C}{Conjecture C}
\newtheorem*{conjecture-D}{Conjecture D}

\theoremstyle{definition}

\theoremstyle{remark}
\newtheorem{remark}[theorem]{Remark}

\numberwithin{equation}{section}

\def\K{\mathrm{K}}

\def\R{\mathrm{R}}

\def\T{\mathrm{T}}

\def\X{\mathrm{X}}
\def\Y{\mathrm{Y}}

\def\bbC{\mathbb{C}}
\def\bbD{\mathbb{D}}

\def\bbG{\mathbb{G}}

\def\bbN{\mathbb{N}}

\def\bbQ{\mathbb{Q}}

\def\bbZ{\mathbb{Z}}

\def\frakg{\mathfrak{G}}

\def\frakL{\mathfrak{L}}
\def\frakM{\mathfrak{M}}

\def\frakP{\mathfrak{P}}

\def\frakS{\mathfrak{S}}

\def\frakX{\mathfrak{X}}
\def\frakY{\mathfrak{Y}}
\def\frakZ{\mathfrak{Z}}

\def\calA{\mathcal{A}}

\def\calC{\mathcal{C}}

\def\calE{\mathcal{E}}

\def\calL{\mathcal{L}}

\def\calO{\mathcal{O}}

\def\calS{\mathcal{S}}

\def\calU{\mathcal{U}}
\def\calV{\mathcal{V}}
\def\calW{\mathcal{W}}
\def\calX{\mathcal{X}}

\def\frakg{\mathfrak{g}}
\def\frakh{\mathfrak{h}}

\def\bfa{\mathbf{a}}
\def\bfb{\mathbf{b}}
\def\bfc{\mathbf{c}}
\def\bfd{\mathbf{d}}

\def\bfi{\mathbf{i}}
\def\bfj{\mathbf{j}}
\def\bfk{\mathbf{k}}
\def\bfl{\mathbf{l}}
\def\bfm{\mathbf{m}}
\def\bfn{\mathbf{n}}

\def\bfp{\mathbf{p}}
\def\bfq{\mathbf{q}}
\def\bfr{\mathbf{r}}
\def\bfs{\mathbf{s}}
\def\bft{\mathbf{t}}

\def\bfv{\mathbf{v}}
\def\bfw{\mathbf{w}}

\def\bfY{\mathbf{Y}}

\def\bfI{\mathbf{I}}

\def\bfN{\mathbf{N}}

\def\bfU{\mathbf{U}}

\def\bfY{\mathbf{Y}}

\def\mod{\operatorname{mod}\nolimits}

\def\top{{\operatorname{top}\nolimits}}

\def\top{{\operatorname{top}\nolimits}}
\def\top{{\operatorname{top}\nolimits}}

\def\rk{{\operatorname{rk}\nolimits}}
\def\gr{{\operatorname{gr}\nolimits}}
\def\op{{{\operatorname{op}\nolimits}}}

\def\-{{\operatorname{-}\!}}

\def\Ker{\operatorname{Ker}\nolimits}
\def\Hom{\operatorname{Hom}\nolimits}

\def\RHom{\operatorname{RHom}\nolimits}

\def\End{\operatorname{End}\nolimits}

\def\Res{\operatorname{Res}\nolimits}

\def\id{\operatorname{id}\nolimits}

\def\Spec{\operatorname{Spec}\nolimits}

\def\Tr{\operatorname{Tr}\nolimits}

\def\Sym{{{\operatorname{Sym}\nolimits}}}

\def\codim{{\operatorname{codim}}}
\def\ev{\operatorname{ev}\nolimits}

\def\leqsl{{\operatorname\leqslant}}
\def\geqsl{{\operatorname\geqslant}}
\def\preceq{{\operatorname\preccurlyeq\nolimits}}

\def\stab{{\bfs\bft\bfa\bfb}}
\def\Exp{\operatorname{Exp}}
\def\cc{{\operatorname{c}}}

\def\Tr{\operatorname{Tr}\nolimits}
\def\wt{\operatorname{wt}\nolimits}

\def\ad{\operatorname{ad}\nolimits}
\def\eu{\operatorname{eu}\nolimits}

\def\ch{\operatorname{ch}\nolimits}

\def\la{{\langle}}
\def\ra{{\rangle}}

\def\res{\mathrm{res}}

\numberwithin{itemcounter}{subsection}
\numberwithin{equation}{section}
\usepackage[toc,page]{appendix}

\usepackage{etoolbox}
\appto\appendix{\addtocontents{toc}{\protect\setcounter{tocdepth}{1}}}
\title[COHA's and Yangians]
{Cohomological Hall algebras of quivers and Yangians}

\author{O. Schiffmann$^1$} 
\address{{\scriptsize{$^1$~Laboratoire de Math\'ematiques d'Orsay, Universit\'e Paris Saclay, 91405 Orsay, France, UMR8628 (CNRS) and Simion Stoilow Institute of Mathematics,Bucharest, Romania, olivier.schiffmann@universite-paris-saclay.fr}}}
\author{E. Vasserot$^2$} 
\address{\scriptsize{$^2$~Universit\'e Paris Cit\'e, 75013 Paris, France, UMR7586 (CNRS), 
		Institut Universitaire de France (IUF), vasserot@imj-prg.fr
}}

\begin{document}
	\maketitle

	\begin{abstract} We construct an isomorphism between the preprojective cohomological Hall algebra of an 
arbitrary quiver and a positive half of the corresponding Maulik-Okounkov Yangian, 
which intertwines the respective actions on the cohomology of the Nakajima quiver varieties. 
We use this to prove a conjecture of Okounkov relating the character of the Maulik-Okounkov 
Lie algebra to Kac polynomials.
		
	\end{abstract}

	%\tableofcontents

	\section{Introduction and notation}
	
	\medskip
	
	\subsection{Introduction}
	
According to Nakajima, to each quiver $Q$ and pair of dimension vectors $(v,w)$ of $Q$ 
one associate a quasi-projective symplectic resolution
	$$\frakM(v,w) \to \Spec  H^0(\frakM(v,w),\calO_\frakM(v,w)).$$
The variety $\frakM(v,w)$ which is a moduli space of framed representations of the double of $Q$, enjoys many remarkable geometric properties and plays an important role in 
geometric representation theory as well as in mathematical physics and algebraic geometry, see, e.g., \cite{N98}, \cite{N00}, \cite{HLRV}. 
In particular, when the quiver $Q$ carries no edge loops and thus can be regarded as an orientation of the 
generalized Dynkin 
diagram of a Kac-Moody algebra $\frakg$, Nakajima 
constructed an action of $\frakg$ on the space 
	$$\bigoplus_vH_{\top}(\frakL(v,w),\bbQ),$$
where $\frakL(v,w)=\pi^{-1}(0)$ is the central fiber, called the nilpotent quiver variety, which is 
Lagrangian.
This module is identified with the irreducible highest weight module of highest weight 
$\sum_i w_i \omega_i$ where $\omega_i$ is the $i$th fundamental weight of $\mathfrak{g}$. In the same spirit, 
when the quiver $Q$ is of finite type, Nakajima constructed an action of the quantum affine algebra of $\mathfrak{g}$ on the equivariant K-theory space 
	$$\bigoplus_{v} K^{G_w\times\bbC^\times}(\frakL(v,w))$$
see \S\ref{sec:quiver variety} for the notation. 
This module is called a universal standard module.
It is a geometric analog of the global Weyl modules. 
A cohomological version of this construction, 
due to Varagnolo \cite{V}, yields an action of the Yangian of $\mathfrak{g}$ on the equivariant Borel-Moore homology space
	$$F_w^\vee=\bigoplus_{v} H^{G_w\times\bbC^\times}_\bullet(\frakL(v,w),\bbQ).$$
There are variants of these representations in which 
one replaces $\frakL(v,w)$ with $\frakM(v,w)$, producing dual modules. Nakajima's and Varagnolo's actions do 
extend to the case of arbitrary quivers but the precise nature 
of the algebra which acts or the structure of the resulting modules are not well understood. 
	
	\medskip
	
There are two main approaches to the problem of constructing and understanding symmetry algebras acting on 
the cohomology of Nakajima quiver varieties for general quivers.
One approach is due to Maulik-Okounkov \cite{MO19} who construct, using the notion of stable envelope, a 
quantum $R$-matrix acting on tensor products of spaces 
	$$F_w=\bigoplus_{v} H^{G_w\times T}_\bullet(\frakM(v,w),\bbQ).$$
Here the quiver $Q$ is arbitrary, $T$ is a torus acting by rescaling the edges, and $w$ is any dimension vector of 
$Q$. Let $R=R_T$ denote the equivariant cohomology ring $H^\bullet_T$, and $K$ its fraction field.
Through the RTT formalism, Maulik and Okounkov define an associative $\bbZ I\times \bbZ$-graded 
$R$-algebra $\bfY_R$ acting on the same spaces $F_w$, called the Maulik-Okounkov Yangian. 
Taking a quasi-classical limit, they also define a classical R-matrix and a 
	$\bbZ I\times \bbZ$-graded Lie algebra $\frakg_R$, called the Maulik-Okounkov Lie algebra, with a triangular decomposition $\frakg_R=\frakg_R^-\oplus\frakg_R^0\oplus\frakg_R^+$.
If the quiver $Q$ is of finite type, 
then $\frakg_R$ is the semisimple Lie algebra associated with $Q$, 
	and $\bfY_R$ is the Yangian of the same type. 
	In general, the $R$-algebra $\bfY_R$ is a deformation of the enveloping algebra of the current algebra 
	$\frakg_R[u]$.
As demonstrated by Maulik and Okounkov in \cite{MO19} and in many subsequent works, see, e.g., 
\cite{OkounkovLectures}, \cite{OkounkovICM} and the references therein, operators coming from this Yangian 
action provide a very powerful tool in the study of the enumerative geometry of $\frakM(v,w)$, such as the 
computation of the quantum cohomology ring, quantum differential equation, the formulation of 3d mirror 
symmetry \ldots On the other hand, from the construction it is hard to determine the precise algebraic nature of 
the Maulik-Okounkov Yangian, which is the algebra generated by all these operators.
	
	\medskip
	
In a different approach, we define $R$-algebras 
$\Y_R$ and $ \Y^\vee_R$ in \cite{SV18}, 
called the preprojective cohomological Hall algebras (COHA) of the quiver $Q$\footnote{One should not confuse these algebras with the cohomological Hall algebras defined  by Kontsevich and Soibelman in the context of 3CY categories. See however, the appendix of \cite{DavisonAppendix} for a relation between the two constructions.}, which act on the same spaces $F_w$ and $ F^\vee_w$, see also \cite {YangZhao}.
This construction is based on the geometry of the cotangent to the moduli stacks of unframed representations of 
the quiver $Q$, satisfying some nilpotency condition in the case of $\Y^\vee_R$, 
and the action on the cohomology of quiver varieties is through Hecke correspondences. There has been much 
progress in understanding the structure of the $R$-algebras $\Y_R, \Y^\vee_R$ in the past few years, see, e.g., 
\cite{SV18}, \cite{DM20}, \cite{DHSM2} and the references therein. In particular, their characters 
are given by Kac polynomials and both algebras are quantum deformations of the 
enveloping algebra of a Lie algebra of polynomial loops into a Borcherds algebra.
	
	\medskip
	
The aim of our paper is to prove that the two very different approaches mentioned above yield
 the same algebras of operators. Our main result provides a commutative diagram
	\begin{equation}\label{E:intro1}
		\xymatrix{{\Y}_K \ar[r]^-{\Phi} \ar[dr]_-{\rho} & \bfY^-_K \ar[d]^-{\rho} \\ & \prod_w \text{End}_K(F_w)}
	\end{equation}
where $\Phi$ is an algebra isomorphism onto one half $\bfY^-_K$ of $\bfY_K$ and $\rho$ denotes the 
product over all $w\in\bbN I$ of the representations on $F_w$. 
We also establish a version of the diagram \eqref{E:intro1} for some 
$R$-forms. Using this isomorphism at the integral level we deduce an equality of characters between $\bfY_R$ 
and $U(\frakg_R^-[u])$, thereby proving a conjecture made by Okounkov soon after the introduction of Yangians, 
see Corollary~\ref{Cor:intro} below.
	
	\medskip
	
In order to state our results more precisely, we need some notation.
Let $v$ be a dimension vector. 
Let $\Pi$ be the preprojective 
algebra of $Q$, and  $\calX_\Pi(v)$ be the moduli stack of complex representations of 
$\Pi$ of dimension $v$. 
Let $\calX_\Pi^\vee(v)$ be the closed substack consisting of nilpotent representations in Lusztig's sense.
There is a large torus $T$ acting on 
$\calX_\Pi^\vee(v)$ and $\calX_\Pi(v)$ by rescaling the arrows. 
Let $R$ be the $T$-equivariant cohomology ring of $\Spec(\bbC)$. Set
$$\Y_R^\vee=\bigoplus_v H_\bullet^{T}(\calX_\Pi^\vee(v),\bbQ)
,\quad
\Y_R=\bigoplus_v H_\bullet^{T}(\calX_\Pi(v),\bbQ).$$
As proved in \cite{SV18}, the $R$-modules $\Y_R^\vee$ and $\Y_R$ are equipped with graded 
$R$-algebra structures and they act on $F^\vee_w$ and $F_w$ respectively, for any $w \in \bbN I$. Besides the 
obvious $\bbZ I$-grading, there is an extra $\bbZ$-grading given by the cohomological degree.
All the graded pieces of $\Y^\vee_{R}$ and $\Y_{R}$ are free $R$-modules of finite rank. The graded 
character of $\Y_R$ and $\Y^\vee_R$ are respectively given by the following formulas involving the Kac 
polynomial $A_v(t)$ counting absolutely indecomposable representations of dimension vector $v$ of the 
quiver $Q$. Set
	$$P(X,q)=\sum_{l\in\bbZ}  \dim H^{T}_{-(v,v)-2l}(X,\bbQ)\,q^l$$ 
	for $X=\calX_\Pi^\vee(v)$ or $\calX_\Pi(v)$ where $(v,v)$ is as in \eqref{pairing1}.
	We have
	\begin{equation}\label{Okounkov}
		\begin{split}
			\sum_v P(\calX_\Pi^\vee(v),q)\,z^v&= 
			(1-q^{-1})^{-\rk(T)}\Exp\Big(\sum_v \frac{A_v(q^{-1})}{1-q^{-1}}z^v\Big)\\
			\sum_v P(\calX_\Pi(v),q)\,z^v&= 
			(1-q^{-1})^{-\rk(T)}\Exp\Big(\sum_v \frac{A_v(q)}{1-q^{-1}}z^v\Big)\\
		\end{split}
	\end{equation}
where $\Exp$ stands for the plethystic exponential. Motivated by the analogy with Yangians, 
we consider extensions $\overline{\Y}_R^\vee$ and $\overline{\Y}_R$ of $\Y_R^\vee$ and $\Y_R$ by 
adding a central commutative algebra isomorphic to the Macdonald ring of symmetric functions of the quiver. 
We extend the action of ${\Y}_R$ and ${\Y}^{\vee}_R$ on $F_w$ and $F_w^\vee$ to $\overline{\Y}_R$ and 
$\overline{\Y}^{\vee}_R$ by making this central algebra act by multiplication by tautological classes with respect 
to $w$.
	
	\medskip 
		
To relate COHAs and Yangians, we first construct in $\S 3$ a nilpotent analog $\bfY^\vee_R$ of the
Yangian which acts on $F_w^\vee$ for any $w$, by restricting the Lagrangian cycle defining the stable 
enveloppe to the nilpotent quiver varieties. Like $\bfY_R$, the $R$-algebra
$\bfY^\vee_R$ is equipped with a filtration 
whose associated graded is isomorphic to $U(\frakg_R^\vee[u])$ for some 
$\bbZ I\times\bbZ$ graded $R$-Lie algebra
$\frakg^\vee_R$. 
Moreover, identifying the dual of $F_w$ and $F_w^\vee$ through the perfect 
intersection pairing, we have an isomorphism
$(-)^\T: \bfY_R^\vee \stackrel{\sim}{\longrightarrow} \bfY^\op_R$ fitting in the following diagram,
see \S\ref{sec:duality} and Lemma~\ref{lem:FF},	
$$\xymatrix{\bfY^\vee_R \ar[rr]^-{(-)^\T} \ar[d]_-{\rho^\vee} && \bfY^\op_R \ar[d]^-{\rho}\\
\prod_w \text{End}_R(F_w^\vee) \ar[rr]^-{x \mapsto x^\T} && \prod_w \text{End}_R(F_w)^\op}$$
The duality descends to the Lie algebras, yielding $\bbZ I\times \bbZ$-graded $R$-module 
isomorphisms
	\begin{equation}\label{E:dualityintro}
		\frakg_R^{+,\vee} \simeq (\frakg_R^{-,\vee})^* \simeq (\frakg_R^+)^*\simeq \frakg_R^-
		\end{equation}
see Proposition~\ref{prop:grading}.
Our main theorem reads as follows. Let 
$$\ev^\vee : \bfY_R^\vee \to \prod_w F_w^\vee$$ be the product of evaluation maps on the vacuum vectors. 
	
\begin{theorem}[thm~\ref{thm:2}, \ref{thm:last}, prop.~\ref{prop:Phi}]\label{thm:mainintro}
\hfill
\begin{enumerate}[label=$\mathrm{(\alph*)}$,leftmargin=8mm,itemsep=1.2mm]
\item The Yangian $\bfY_R^\vee$ admits a triangular decomposition
$$\bfY_{R}^{-,\vee}\otimes_{R}\bfY_{R}^{0,\vee}\otimes_{R}\bfY_{R}^{+,\vee}=\bfY_{R}^\vee.$$
\item There is an $R$-algebra  isomorphism $\Phi: \Y^\vee_R \to \bfY_R^{-,\vee}$
which intertwines the representations on $F^\vee_w$ for any $w \in \bbN I$.		
\item The obvious map $\Psi:\bfY_R^{-,\vee} \to \bfY_R^{\vee}/ 
\Ker\ev^\vee$ is an isomorphism of graded $R$-modules.
\end{enumerate}
\end{theorem}
The proof proceeds is several steps. One first establishes a version of (a) over the field 
$K$, using 
arguments similar to \cite{SV18b} and a recent spherical generation result of Negut \cite{N23b}. We then use 
some characterization of the integral form in terms of the action on the modules $F_w^\vee$. The isomorphism 
of Part (b), in particular, the fact that $\Psi$ is surjective, 
comes from the cyclicity of the action of $\Y_R^\vee$ on the modules $F^\vee_w$ established in \cite{SV18}.
It is the main reason why we use $\bfY^\vee_R$ instead of $\bfY_R$.  
%Part (b) of the theorem may be interpreted as saying that $\bfY_R^{-,\vee}$ is a positive half of 
%$\bfY^\vee_R$. One can make precise sense of this, for both $\bfY_R$ and $\bfY_R^\vee$.
%We'll do it elsewhere. 
	
	\medskip
	
Using the representation theory of Maulik-Okounkov Lie algebras and passing to the classical limit,
i.e.,  to the associated graded with respect to the so-called $u$-filtration, one proves that
$$ \gr\left( \bfY^\vee_R/ \Ker \ev^\vee\right) \simeq U(\frakg^{-,\vee}_R[u]).$$
Combining this with Theorem~\ref{thm:mainintro},  \eqref{Okounkov} and \eqref{E:dualityintro} we obtain the following.
	
	\begin{corollary}[Okounkov's conjecture]\label{Cor:intro} For any $v \in \bbN I$ we have
		$$\rk \left( \frakg_{v,R}\right)=A_{Q,v}(t), \quad \rk \left( \frakg^\vee_{v,R}\right)=A_{Q,v}(t^{-1}).$$
	\end{corollary}
The formula for the graded character of $\frakg_R$ 
was conjectured by Okounkov, see \cite{O13}. 
It is a generalization of the Kac conjecture proved by Hausel in \cite{HauselKac}, which provides a 
Lie theoretic interpretation of the constant term $A_{Q,v}(0)$ of Kac polynomials.
	
\medskip
	
As we were finishing this paper, a different proof of Okounkov's conjecture appeared, by T. M. Botta and B. Davison \cite{MBD}. Their method relies on the theory of BPS Lie algebras and nonabelian stable enveloppes.

\bigskip

\subsection{Notation}
For each module $M_R$ over an integral domain $R$ and for any commutative $R$-algebra $K$
we abbreviate $M_K=M_R\otimes_RK$.
Let $K_R$ is the fraction field of $R$. We'll say that an $R$-linear morphism 
$f:M_R\to N_R$ is generically invertible over $R$ if the induced map $M_K\to N_K$ is invertible.
For any vector bundle $\calE$ and any integer $l\in\bbN$,
let $\ch_l(\calE)$ denote the degree $2l$ component of the Chern character
$\ch(\calE)$.
All varieties and schemes are defined over $\bbC$.
Let $H_G^\bullet(X,\bbQ)$ and $H^G_\bullet(X,\bbQ)$
denote the $G$-equivariant cohomology and Borel-Moore homology of the $G$-variety $X$.
We'll abbreviate
$$R_G=H^\bullet_G=H_G^\bullet(\Spec(\bbC),\bbQ)
,\quad
K_G=K_{R_G}.$$
We equip $R_G$ with the $\bbZ$-grading such that for each integer $l$ we have
$$R_{l,G}=H_G^{2l}(\Spec(\bbC),\bbQ).$$

\bigskip

\section{Cohomological Hall algebras}

\subsection{The Nakajima quiver varieties}\label{sec:quiver variety}

\subsubsection{Basic on quivers}\label{sec:quivers}
Let $Q$ be a finite quiver with sets of vertices and of arrows $Q_0$ and $Q_1$.
Let $s,t:Q_1\to Q_0$ be the source and target.
Let $\alpha^*$ be the arrow opposite to the arrow $\alpha\in Q_1$.
We'll use the auxiliary sets 
$$Q_1^*=\{\alpha^*\,;\,\alpha\in Q_1\}
,\quad
Q'_0=\{i'\,;\,i\in Q_0\}
,\quad
Q'_1=\{a_i:i\to i'\,;\,i\in Q_0\}
.$$
From $Q$ we construct new quivers as follows :
\begin{itemize}[leftmargin=3mm]
\item[-]
$Q^*$ is the opposite quiver : $(Q^*)_0=Q_0$, $(Q^*)_1=Q_1^*$,
\item[-]
$\overline Q$ is the double quiver :
$\overline Q_0=Q_0$, 
$\overline Q_1=Q_1\cup Q_1^*$,
\item[-] 
$Q_f$ is the framed quiver :
$Q_{f,0}=Q_0\sqcup Q'_0$,
$Q_{f,1}=Q_1\sqcup Q'_1$,
\item[-]
$\overline Q_f=\overline{(Q_f)}$ is the framed double quiver.
\end{itemize}
We abbreviate $I=Q_0$.
For any finite dimensional $\bbZ I$-graded vector space $V$ we write
$V=\bigoplus_{i\in I}V_i$.
Let $\delta_i\in\bbN I$ be the Dirac functions at $i$.
The dimension vector is
$v=\sum_{i\in I}v_i\delta_i$ where $v_i = \dim(V_i)$.
Given two $\bbZ I$-graded vector spaces $V,W$, the representation varieties of $Q$ and $Q_f$ are
$$\X_Q(v)=\prod_{x\in Q_1}\Hom(V_{s(x)},V_{t(x)})
 ,\quad
\X_{Q_f}(v,w)=\prod_{x\in Q_1}\Hom(V_{s(x)},V_{t(x)})\times
\prod_{i\in Q_0}\Hom(V_i,W_i).$$
A representation of $\overline Q_f$  is a tuple 
$x=(x_\alpha,x_{\alpha^*},x_a,x_{a^*})$ with
$\alpha, a$ running in $Q_1, Q'_1$ respectively.
We abbreviate $h=x_h$ for each arrow $h$, and we write $x=(\alpha\,,\,\alpha^*\,,\,a\,,\,a^*)$.
We also abbreviate
$$\overline\X=\X_{\overline Q_f}
,\quad
\X=\X_{Q_f}
,\quad
\X^*=\X_{Q^*_f}.$$
Let $G_v=\prod_{i\in I}GL(V_i)$ and
$T=(\bbG_m)^{Q_1}\times\bbG_m.$
Let $\frakg_{v,R}$ be the Lie algebra of $G_v$.
Let $S\subset T$ be any subtorus.
We'll write
$$R=R_T=\bbQ[t_\alpha,\hbar\,;\,\alpha\in Q_1]
,\quad
K=K_R
,\quad
R_{v,S}=R_{G_v\times S}.$$

\medskip

\subsubsection{The Nakajima quiver varieties}\label{sec:quiver variety}

The group $G_v\times G_w\times T$ acts on the variety $\overline\X(v,w)$ as follows :
the groups $G_v$, $G_w$ act by conjugation, and the torus $T$ acts as follows 
\begin{align}\label{action1}
(z_\alpha,z)\cdot x=(z_\alpha\alpha,zz_{\alpha^*}\alpha^*,a,za^*)
 ,\quad
z_{\alpha^*}=z_\alpha^{-1}.
\end{align}
The representation variety $\overline\X(v,w)$ is holomorphic symplectic with a Hamiltonian action of
the groups $G_v$ and $G_w$.
The moment map of the $G_v$-action is given by the formula
$$\mu_{v,w}:\overline\X(v,w)\to\text{Lie}(G_v)^*
, \quad
(x_\alpha,x_{\alpha^*},x_a,x_{a^*})\mapsto \sum_\alpha [x_\alpha,x_{\alpha^*}] + x_{a^*}x_a$$
in which we identify $\text{Lie}(G_v)$ with its dual via the trace.
We'll say that a representation in $\overline\X(v,w)$ is stable if it has no proper 
subrepresentation containing $W$.
Set
$$\overline\X(v,w)_s=\{x\in\overline\X(v,w)\,;\,x\ \text{is\ stable}\}$$
We also set
$$
\X_\Pi(v,w)_s=\overline\X(v,w)_s\cap\X_\Pi(v,w)
,\quad
\X_\Pi(v,w)=\mu_{v,w}^{-1}(0).$$
The Nakajima quiver varieties are the categorical quotients 
$$\frakM(v,w)=\X_\Pi(v,w)_s/\hspace{-0.05in}/G_v
 ,\quad
\frakM_0(v,w)=\X_\Pi(v,w)/\hspace{-0.05in}/G_v.$$
Let $\pi$ be the obvious morphism
$\frakM(v,w)\to\frakM_0(v,w).$
It is a projective map.
We define
$$\frakM'_0(v,w)=\Spec H^0(\frakM(v,w),\calO_{\frakM(v,w)}).$$
The map $\pi$ factorizes as follows, with $\pi'$ being a symplectic resolution of singularities
$$\xymatrix{\frakM(v,w)\ar[r]^-{\pi'}&\frakM'_0(v,w)\ar[r]&\frakM_0(v,w).}$$
The $G_w\times T$-variety $\frakM(v,w)$ is smooth, quasi-projective,
connected and holomorphic symplectic with Hamiltonian $G_w$-action.
Let $\frakL(v,w)$ be the nilpotent quiver variety.
It is the projective subvariety of $\frakM(v,w)$ whose set of points is 
the attracting set for the 
$\bbG_m$-action on $\frakM(v,w)$ associated with a
cocharacter $\gamma$ of the form
\begin{align*}
\gamma(z)=z^a\id_{\X(v,w)}\oplus z^{a^*}\id_{\X^*(v,w)}
,\quad
a,a^*\in\bbZ_{<0}
,\quad
z\in\bbC^\times
\end{align*}
Let $\bfv=(v_1,\dots, v_s)$ and  $\bfw=(w_1,\dots, w_s)$ be $s$-tuples
of elements of $\bbN I$. We write
$$|\bfw|=\sum_{r=1}^sw_r
,\quad
G_\bfw=\prod_{r=1}^sG_{w_r}
,\quad
T_\bfw=\prod_{r=1}^sT_{w_r}$$
where $T_w\subset G_w$ is a maximal torus for each $w$. We set
$$\frakM(w)=\bigsqcup_{v\in\bbN I}\frakM(v,w)
,\quad
\frakL(w)=\bigsqcup_{v\in\bbN I}\frakL(v,w),$$
$$\frakM(\bfv,\bfw)=\prod_{r=1}^s\frakM(v_r,w_r)
,\quad
\frakL(\bfv,\bfw)=\prod_{r=1}^s\frakL(v_r,w_r),$$
$$\frakM(\bfw)=\bigsqcup_{\bfv\in(\bbN I)^s}\frakM(\bfv,\bfw)
,\quad
\frakL(\bfw)=\bigsqcup_{\bfv\in(\bbN I)^s}\frakL(\bfv,\bfw).$$
We'll also need the tautological vector bundles on $\frakM(w)$
$$\calV=\bigoplus_{i\in I}\calV_i
,\quad
\calW=\bigoplus_{i\in I}\calW_i.$$

\subsubsection{The cohomology of quiver varieties}\label{sec:cohquiv}
The $G_w$-varieties $\frakM(v,w)$, $\frakL(v,w)$ are equivariantly formal with
even Borel-Moore homology.
Set $d_{v,w}=\frac{1}{2}\dim \frakM(v,w)$ and $d_v=d_{v,v}$.
We have
\begin{align}\label{pairing1}d_{v,w}=v\cdot w-\frac{1}{2}(v,v)_Q
,\quad
v\cdot w=\sum_{i\in I}v_iw_i
,\quad
(v,w)_Q=\sum_{i,j\in I}\bfc_{i,j}v_iw_j
\end{align}
where 
$\bfc_{i,j}=2\delta_{i,j}-\sharp\{\alpha:i\to j\,;\,\alpha\in\overline Q_1\}.$
Set
$$F_{v,w,R_{w,S}}=H_\bullet^{G_w\times S}(\frakM(v,w),\bbQ)
,\quad
F_{v,w,R_{w,S}}^\vee=H_\bullet^{G_w\times S}(\frakL(v,w),\bbQ).$$
We equip $F_{v,w,R_{w,S}}$ and $F_{v,w,R_{w,S}}^\vee$ with the $\bbZ$-gradings such that
\begin{align}\label{cohdeg}
F_{v,w,l,R_{w,S}}=H^{G_w\times S}_{2d_{v,w}-2l}(\frakM(v,w),\bbQ)
,\quad
F_{v,w,l,R_{w,S}}^\vee=H^{G_w\times S}_{2d_{v,w}-2l}(\frakL(v,w),\bbQ).
\end{align}
Taking the sum over all $v$'s, we get the following $\bbZ I\times\bbZ$-graded $R_{w,S}$-modules
$$F_{w,R_{w,S}}=\bigoplus_{v\in\bbN I}F_{v,w,R_{w,S}}
,\quad
F_{w,R_{w,S}}^\vee=\bigoplus_{v\in\bbN I}F_{v,w,R_{w,S}}^\vee
.$$
We'll call $v$ the weight and $l$ the cohomological degree.
We'll abbreviate 
\begin{itemize}[label=$\mathrm{(\alph*)}$,leftmargin=8mm,itemsep=1.2mm]
\item[-]
$\wt(x)=$ weight of $x$,
\item[-]
$\deg(x)=$ cohomological degree of $x$.
\end{itemize}
The vacuum vectors are the fundamental classes
$$\phi_w=[\frakM(0,w)]
,\quad
\phi_w^\vee=[\frakL(0,w)].$$
For any $s$-tuples $\bfv=(v_1,\dots, v_s)$  and $\bfw=(w_1,\dots, w_s)$
of elements of $\bbN I$ we write
$$R_{\bfw,S}=\bigotimes_{r=1}^sR_{w_r,S}
,\quad
K_{\bfw,S}=K_{R_{\bfw,S}},$$
$$
F_{\bfv,\bfw,R_{\bfw,S}}=\bigotimes_{r=1}^sF_{v_r,w_r,R_{w_r,S}}
,\quad
F_{\bfv,\bfw,R_{\bfw,S}}^\vee=\bigotimes_{r=1}^sF_{v_r,w_r,R_{w_r,S}}^\vee,
$$
$$
F_{\bfw,R_{\bfw,S}}=\bigoplus_{\bfv\in(\bbN I)^s}F_{\bfv,\bfw,R_{\bfw,S}}
,\quad
F_{\bfw,R_{\bfw,S}}^\vee=\bigoplus_{\bfv\in(\bbN I)^s}F_{\bfv,\bfw,R_{\bfw,S}}^\vee,
$$
$$\phi_\bfw=\bigotimes_{r=1}^s\phi_{w_r}
,\quad
\phi_\bfw^\vee=\bigotimes_{r=1}^s\phi_{w_r}^\vee.$$
We equip  $F_{\bfw,R_{\bfw,S}}$ and $F_{\bfw,R_{\bfw,S}}^\vee$
with the product $\bbZ I\times\bbZ$-grading.
We define the dual $R_{\bfw,S}$-module of $F_{\bfw,R_{\bfw,S}}$ to be
$$(F_{\bfw,R_{\bfw,S}})^*=\bigoplus_{\bfv\in(\bbN I)^s}(F_{\bfv,\bfw,R_{\bfw,S}})^*
=\bigoplus_{\bfv\in(\bbN I)^s}\Hom_{R_{\bfw,S}}(F_{\bfv,\bfw,R_{\bfw,S}}, R_{\bfw,S})$$ 
We define $(F_{\bfw,R_{\bfw,S}}^\vee)^*$ similarly.
Let $\hbar:S\to\bbG_m$ be the weight of the symplectic form on $\frakM(\bfw)$.
Unless specified otherwise we'll assume that the restriction of 
$\hbar$ to the subtorus $S\subset T$ is non trivial.

\begin{lemma}\label{lem:F}
\hfill
\begin{enumerate}[label=$\mathrm{(\alph*)}$,leftmargin=8mm,itemsep=1.2mm]
\item
$F_{\bfw,R_{\bfw,S}}$ and $F_{\bfw,R_{\bfw,S}}^\vee$ are free $\bbZ I\times\bbZ$-graded
$R_{\bfw,S}$-modules of finite rank such that
$$F_{\bfv,\bfw,R_{\bfw,S}}=F_{\bfv,\bfw,R_\bfw}\otimes_RR_S
,\quad
F_{\bfv,\bfw,R_{\bfw,S}}^\vee=F_{\bfv,\bfw,R_\bfw}^\vee\otimes_RR_S.$$
\item
The cup-product on $\frakM(\bfv,\bfw)$ yields a perfect pairing of 
 $\bbZ$-graded $R_{\bfw,S}$-modules
\begin{align}\label{PP1}
F_{\bfv,\bfw,R_{\bfw,S}}\times F_{\bfv,\bfw,R_{\bfw,S}}^\vee\to R_{\bfw,S}.
\end{align}
\item 
The K\"unneth formula holds, i.e., the 
obvious maps yield the isomorphisms
$$H_\bullet^{G_\bfw\times S}(\frakM(\bfw),\bbQ)=F_{\bfw,R_{\bfw,S}}
,\quad
H_\bullet^{G_\bfw\times S}(\frakL(\bfw),\bbQ)=F_{\bfw,R_{\bfw,S}}^\vee.$$
\item
The pushforward by the closed embedding
$\frakL(\bfw)\subset\frakM(\bfw)$ yields an $R_{\bfw,S}$-linear map
\begin{align}\label{iota1}
\iota:F_{\bfw,R_{\bfw,S}}^\vee\to F_{\bfw,R_{\bfw,S}}
\end{align} 
such that $\phi_\bfw^\vee\mapsto\phi_\bfw$.
This map is injective and generically invertible over $R_{\bfw,S}$.
\end{enumerate}
\end{lemma}

\begin{proof}
The proof of of the lemma is given in \cite{N00}. We give some details on Part (d) for completeness.
Since the restriction of $\hbar$ to $S$ is non trivial, we have
$\frakM_0(\bfw)^S=\{0\}$, hence $\frakM(\bfw)^S=\frakL(\bfw)^S$.
Thus the map $F_{w,R_S}^\vee\to F_{w,R_S}$ becomes invertible after inverting finitely many weights of 
$T_\bfw\times S$ which restrict non trivially to $S$.
We deduce that the map $\iota$ is injective and generically invertible over $R_{\bfw,S}$.
\end{proof}

From \eqref{PP1}, we deduce that as a graded $R_{\bfw,S}$-module we have
\begin{align}\label{*v1}
F_{\bfv,\bfw,R_{\bfw,S}}^\vee=(F_{\bfv,\bfw,R_{\bfw,S}})^*.
\end{align}
We define
$$A_{\bfw,R_{\bfw,S}}=(F_{\bfw,R_{\bfw,S}})^*\otimes_{R_{\bfw,S}}F_{\bfw,R_{\bfw,S}}
,\quad
A_{\bfw,R_{\bfw,S}}^\vee=(F_{\bfw,R_{\bfw,S}}^\vee)^*\otimes_{R_{\bfw,S}}F_{\bfw,R_{\bfw,S}}^\vee.$$
From \eqref{*v1}, we deduce that
$$A_{\bfw,R_{\bfw,S}}=F_{\bfw,R_{\bfw,S}}^\vee\otimes_{R_{\bfw,S}}F_{\bfw,R_{\bfw,S}}
,\quad
A_{\bfw,R_{\bfw,S}}^\vee=F_{\bfw,R_{\bfw,S}}\otimes_{R_{\bfw,S}}F_{\bfw,R_{\bfw,S}}^\vee,$$
hence, the transpose relatively to the pairing \eqref{PP1} yields the $R_{\bfw,S}$-module isomorphism
\begin{align}\label{T1}
(-)^\T:A_{\bfw,R_{\bfw,S}}\to A_{\bfw,R_{\bfw,S}}^\vee.
\end{align}
Let $\bfw\bfw'$ denote the tuple
given by glueing the tuples $\bfw$ and $\bfw'$.
We have
$$F_{\bfw,R_{\bfw,S}}\otimes_RF_{\bfw',R_{\bfw',S}}=F_{\bfw\bfw',R_{\bfw\bfw',S}}
,\quad
A_{\bfw,R_{\bfw,S}}\otimes_RA_{\bfw',R_{\bfw',S}}=A_{\bfw\bfw',R_{\bfw\bfw',S}}.$$

\begin{remark}\label{rem:RS}
We'll also need the following $R_S$-modules
$$F_{v,w,R_S}=H_\bullet^S(\frakM(v,w),\bbQ)
,\quad
F_{\bfv,\bfw,R_S}=\bigotimes_{r=1}^sF_{v_r,w_r,R_S}
,\quad
F_{\bfw,R_S}=\bigoplus_{\bfv\in(\bbN I)^s}F_{\bfv,\bfw,R_S}.
$$
We define
$F_{v,w,R_S}^\vee$, 
$F_{\bfv,\bfw,R_S}^\vee$ and $F_{\bfw,R_S}^\vee$ similarly.
They are free $R_S$-modules.
The pushforward by the closed embedding
$\frakL(\bfw)\subset\frakM(\bfw)$ yields an injective $R_S$-linear map
\begin{align}\label{iota2}
\iota:F_{\bfw,R_S}^\vee\to F_{\bfw,R_S}
\end{align}
which is generically invertible over $R_S$.
We also set
$$(F_{\bfw,R_S})^*=\bigoplus_{\bfv\in(\bbN I)^s}(F_{\bfv,\bfw,R_S})^*
=\bigoplus_{\bfv\in(\bbN I)^s}\Hom_{R_S}(F_{\bfv,\bfw,R_S}, R_S),$$ 
The cup-product on $\frakM(\bfv,\bfw)$ yields a perfect pairing of 
 $\bbZ I\times\bbZ$-graded $R_S$-modules
\begin{align}\label{PP2}
F_{\bfv,\bfw,R_S}\times F_{\bfv,\bfw,R_S}^\vee\to R_S
\end{align}
which identifies the $\bbZ I\times\bbZ$-graded $R_S$-module $F_{\bfv,\bfw,R_S}^\vee$ with the dual of $F_{\bfv,\bfw,R_S}$.
We define
\begin{align*}
A_{\bfw,R_S}&=(F_{\bfw,R_S})^*\otimes_{R_S}F_{\bfw,R_S}=F_{\bfw,R_S}^\vee\otimes_{R_S}F_{\bfw,R_S}
,\\
A_{\bfw,R_S}^\vee&=(F_{\bfw,R_S}^\vee)^*\otimes_{R_S}F_{\bfw,R_S}^\vee
=F_{\bfw,R_S}\otimes_{R_S}F_{\bfw,R_S}^\vee.
\end{align*}
The transpose relatively to the pairing \eqref{PP2} yields the $R_S$-module isomorphism
\begin{align}\label{T2}(-)^\T:A_{\bfw,R_S}\to A_{\bfw,R_S}^\vee.\end{align}

\end{remark}

\medskip

\subsection{The cohomological Hall algebras}\label{sec:Hall}
This section is a reminder from \cite{SV18}.
Let 
$$\X_\Pi^\vee(v)\subset\X_\Pi(v)$$ 
be the set of $v$-dimensional 
representations of the preprojective algebra of $Q$ and
the Lusztig nilpotent subvariety.
Let $\calX_\Pi(v)$ and $\calX_\Pi^\vee(v)$ be the quotient stacks 
$$\calX_\Pi(v)=\X_\Pi(v)/G_v
,\quad
\calX_\Pi^\vee(v)=\X_\Pi^\vee(v)/G_v.$$
%Recall that $\X_\Pi^\vee(v)$ is a $G_v$-invariant closed subvariety of $\X_{\overline Q}(v)$.
%It is not of pure dimension in general.
By \cite[\S 4.1]{SV18}, we have
\begin{align}\label{frakL}
\frakL(v,w)=(\overline\X(v,w)_s\cap(\X_\Pi^\vee(v)\times\Hom_I(W,V)))\,/\,G_v.
\end{align}
We define
$$\Y_{R_S}=\bigoplus_{v\in\bbN I}\bigoplus_{l\in\bbN}\Y_{-v,l,R_S},
\quad
\Y_{-v,l,R_S}=H^{G_v\times S}_{2d_v-2l}(\X_\Pi(v),\bbQ)
.$$
We also consider the nilpotent version $\Y_{R_S}^\vee$ of $\Y_{R_S}$
such that
$$\Y_{R_S}^\vee=\bigoplus_{v\in\bbN I}\bigoplus_{l\in\bbN}\Y^\vee_{-v,l,R_S},
\quad
\Y^\vee_{-v,l,R_S}=H^{G_v\times S}_{2d_v-2l}(\X_\Pi^\vee(v),\bbQ).$$
We'll call $v$ the weight and $l$ the cohomological grading.
The $\bbZ I\times\bbZ$-graded $R_S$-modules $\Y_{R_S}$ and
$\Y_{R_S}^\vee$ are equipped with an associative multiplication
given by the convolution relative to the diagram of stacks
$$\xymatrix{ \calX_\Pi(v_1) \times \calX_\Pi(v_2) & 
\widetilde\calX_\Pi(v_1,v_2) \ar[l]_-{q} \ar[r]^-{p} & \calX_\Pi(v_1+v_2)}$$
and its nilpotent version. Here 
$\widetilde\calX_\Pi(v_1,v_2)$ is the stack parametrizing nested pairs
 of representations of $\Pi$ of  dimensions $v_1$ and $v_1+v_2$.
See \cite[\S 5]{SV18} for details. 
Here we consider the multiplication opposite to the one used in loc. cit., because
the stability condition for quiver varieties used in \cite{SV18} is opposite to the one used here.

\medskip

Let us give another interpretation of the multiplication.
To do that, following \cite{VV22}, we consider the convolution diagram of derived Artin stacks 
\begin{align}\label{indiag}
	\xymatrix{\frakX_\Pi(v_1) \times \frakX_\Pi(v_2) & \widetilde\frakX_\Pi(v_1,v_2) \ar[l]_-{q} \ar[r]^-{p} & \frakX_\Pi(v_1+v_2)}
\end{align}
Here $\frakX_\Pi(v_a)$ is the derived Hamiltonian reduction of the moment map
$\X_{\overline Q}(v_a)\to\text{Lie}(G_{v_a})^*$ as in \cite{VV22} for $a=1,2$, 
see also \cite{P12},
and $\widetilde\frakX_\Pi(v_1,v_2)$
is the derived Artin stack of nested pairs which can be defined as follows.
Let $\calE_a$ be the tautological sheaf over $\frakX_\Pi(v_a)$.
It is an $I$-graded vector bundle whose fibers carries a representation of $\Pi$.
The complex over $\frakX_\Pi(v_1) \times \frakX_\Pi(v_2)$
\begin{align}\label{C}\calC=\RHom_\Pi(\calE_1,\calE_2)[1]
\end{align}
has perfect amplitude $[-1,1]$.
Its total space is the derived Artin stack over $\frakX_\Pi(v_1) \times \frakX_\Pi(v_2)$
$$q:\widetilde\frakX_\Pi(v_1,v_2)=\Spec(\Sym(\calC^*))\to\frakX_\Pi(v_1) \times \frakX_\Pi(v_2)$$
Its classical truncation is the stack $\widetilde\calX_\Pi(v_1,v_2)$ above, i.e., we have
$$\widetilde\calX_\Pi(v_1,v_2)=\Spec(\Sym(\tau_{\leqsl 0}\calC^*))|_{\calX_\Pi(v_1) \times \calX_\Pi(v_2)},$$
where $\tau_{\leqsl 0}\calC$ is the truncation of the complex $\calC$ of perfect amplitude $[-1,0]$.
This Artin stack classifies the short exact sequences of $\Pi$-modules 
$0\to x_1\to x\to x_2\to 0$ with $x_a$ of dimension $v_a$.
A morphism of derived Artin stacks is quasi-smooth if its relative cotangent complex has perfect amplitude
$[-1,1]$. Hence $q$ is quasi-smooth.
Thus we can apply the formalism of \cite{K19}, see also \cite{PY22}.

Recall that for any derived Artin stack $\frakX$ there is a stable $\infty$-category $\mathrm{Sh}_{\bbQ}(\frakX)$ of 
constructible $\bbQ$-sheaves on $\frakX$ which satisfies a six-functor formalism. 
The dualizing complex is defined as $\mathbb{D}_\frakX=a^!\bbQ$ where 
$a: \frakX \to \Spec(\mathbb{C})$ is the structural map. The sheaf of Borel-Moore chains on $\frakX$ is 
$a_*\mathbb{D}_\frakX$ in $\mathrm{Sh}_\bbQ(\Spec(\bbC)) = D(\bbQ\text{-mod}).$
The Borel-Moore homology is obtained by taking derived global sections 
$H_l(\frakX,\bbQ)=H^{-l}(a_*\bbD_\frakX)$. 
It satisfies the usual properties, see \cite[\S 2]{K19}. 
Borel-Moore homology is insensitive to the derived structure, in the sense that the direct image map 
$H_\bullet(\frakX^{cl},\bbQ) \to H_\bullet(\frakX,\bbQ)$ is invertible. 
For any quasi-smooth morphism $f: \frakX \to \frakY$ of dimension $d$
there is a Gysin map 
$f^!: H_l(\frakY,\bbQ) \to H_{l+2d}(\frakX,\bbQ)$
called the virtual pullback.

Going back to \eqref{indiag}, the map $q$ is quasi-smooth
and we define the multiplication in $\Y_{R_S}$ to be the composed map
$p_*\circ q^!$. By definition we have
$$\Y_{-v,l,R_S}=H^{S}_{-(v,v)_Q-2l}(\calX_\Pi(v),\bbQ)$$
where the pairing $(-,-)_Q$ is as in \eqref{pairing1}.
The multiplication of  $\Y_{R_S}^\vee$ is built as above from the induction diagram
$$\xymatrix{\frakX_\Pi^\vee(v_1) \times \frakX_\Pi^\vee(v_2) & 
	\widetilde\frakX_\Pi^\vee(v_1,v_2) \ar[l]_-{q^\vee} \ar[r]^-{p^\vee} & \frakX_\Pi^\vee(v_1+v_2)}$$
where $\frakX_\Pi^\vee(v_a)\subset\frakX_\Pi(v_a)$ and 
$\widetilde\frakX_\Pi^\vee(v_1,v_2)\subset \widetilde\frakX_\Pi(v_1,v_2)$ are the closed derived Artin substacks
of nilpotent representations. Since the map $q^\vee$ is obtained from $q$ by base change, it is also quasi-smooth, and of 
the same virtual dimension.

\medskip

For a future use, let us quote the following result.

\begin{lemma} The multiplication on $\Y_{R_S}$ and $\Y_{R_S}^\vee$
is homogeneous of degree 0 for the $\bbZ I\times\bbZ$-grading.
\end{lemma}

\begin{proof}
This can be shown using the combinatorics in \cite[\S 5.1]{SV18}.
Alternatively, to prove the lemma for $Y_{R_S}$ we may simply check that the virtual rank of the relative cotangent complex, i.e., the Euler characteristic of the complex $\calC$ in \eqref{C}, 
is $(v_1,v_2)_Q$.
This follows from \cite[prop~3.1]{SV18}. The same argument applies for $Y_{R_S}^\vee$.
\end{proof}

We now consider larger algebras than $\Y_{R_S}$ and $\Y_{R_S}^\vee$.
We define the $\bbZ I\times \bbZ$-graded $R_S$-algebra 
\begin{align}\label{Rinfty}
R_{\infty,S}=R_S[p_{i,l}\,;\,i\in I\,,\,l\in\bbN]
\end{align}
where $p_{i,l}$ is of degree $(i,l)$. Note that we allow the degree $l$ to be 0.
Identifying the tautological bundle $\calW_i$ with the $i$th standard representation of the affine 
group $G_w$, we may
view the characteristic class $\ch_l(\calW_i)$ as an element in $R_{w,S}$  for each
$i\in I$, $l\in\bbN$ and $w\in\bbN I$.
There is a unique $R_S$-algebra homomorphism 
\begin{align}\label{Rinftyw}
R_{\infty,S}\to R_{w,S}
,\quad
p_{i,l}\mapsto\ch_l(\calW_i).
\end{align}
We define $\overline\Y_{R_S}$ and $\overline\Y_{R_S}^\vee$ to be the $R_S$-algebras
\begin{align}\label{oYY}
\overline\Y_{R_S}=\Y_{R_S}\otimes_{R_S} R_{\infty,S}
,\quad
\overline\Y_{R_S}^\vee=\Y_{R_S}^\vee\otimes_{R_S} R_{\infty,S}
\end{align}
Here the subalgebras $\Y_{R_S}$ and $\Y_{R_S}^\vee$ commute with $R_{\infty,S}$.
For each $w\in\bbN I$, the map \eqref{Rinftyw} yields an $R_S$-algebra homomorphism
\begin{align}\label{YY}\overline\Y_{R_S}^\vee\to\Y_{R_S}^\vee\otimes_{R_S} R_{w,S}.\end{align}
It is proved in \cite[\S 5.6]{SV18} that there are $R_S$-algebra homomorphisms
\begin{align}\label{rhoCOHA}
\begin{split}
\rho_w:\overline\Y_{R_S}\to A_{w,R_{w,S}}
,\quad
\rho_w^\vee:\overline\Y_{R_S}^\vee\to A_{w,R_{w,S}}^\vee
\end{split}
\end{align}
yielding representations on $F_{w,R_{w,S}}$ and $F_{w,R_{w,S}}^\vee$ respectively,
with a
$\overline\Y_{R_S}^\vee$-module homomorphism 
$$\iota:F_{w,R_{w,S}}^\vee\to F_{w,R_{w,S}}$$
such that $\phi_w^\vee\mapsto\phi_w$. 
The element $p_{i,l}\in R_{\infty,S}$ 
acts on $F_{w,R_{w,S}}^\vee$ and $F_{w,R_{w,S}}$ through the multiplication with
the Chern class $\ch_l(\calW_i)$ of the tautological bundle $\calW_i$.
Let us briefly recall the definition of the action of $\Y_{R_S}$ and $\Y_{R_S}^\vee$ on $F_{w,R_{w,S}}$ and $F^\vee_{w,R_{w,S}}$.
There is an induction diagram
$$\xymatrix{{\frakX}_\Pi(v_1) \times \frakX_\Pi(v_2,w) & \widetilde{\frakX}_\Pi(v_1,v_2,w) \ar[l]_-{q'} \ar[r]^-{p'} & \frakX_\Pi(v_1+v_2,w)}$$
where $\frakX_\Pi(v,w)=R\mu_{v,w}^{-1}(0)$ is the derived Artin stack parametrizing framed representations of the preprojective algebra of $Q$ of dimension $(v,w)$, and where $\widetilde{\frakX}_\Pi(v_1,v_2,w)$ is the derived Artin stack parametrizing nested framed representations $x_1 \subset x$ with $x_1$ of dimension $(v_1,0)$ and $x_2:=x/x_1$ of dimension $(v_2,w)$. As before, the morphism $q'$ is quasi-smooth and $p'$ is proper. Since the stability is an open condition, there are open embeddings $\frakM(v,w) \to \frakX_\Pi(v,w)$. 
The derived fiber product
$$\widetilde{\frakM}(v_1,v_2,w)= \widetilde{\frakX}_\Pi(v_1,v_2,w) \underset{\frakX_\Pi(v_1+v_2,w)}{\times} \frakM(v_1+v_2,w)$$
parametrizes nested representations $x_1 \subset x$ for which $x$ is stable. Note that this automatically implies that $x_2$ is stable as well, because a quotient of a cyclic module remains cyclic. This results in a diagram
$$\xymatrix{\frakX_\Pi(v_2) \times \frakM(v_2,w) \ar[d] & \widetilde{\frakM}(v_1,v_2,w) \ar[l]_-{q} \ar[r]^-{p} \ar[d]& \frakM(v_1+v_2,w)\ar[d]\\
{\frakX}_\Pi(v_2) \times \frakX_\Pi(v_1,w) & \widetilde{\frakX}_\Pi(v_1,v_2,w) \ar[l]_-{q'} \ar[r]^-{p'} & \frakX_\Pi(v_1+v_2,w)}$$
in which all the vertical arrows are open embeddings. The right square is cartesian by construction. The left 
square isn't, but there is a factorization $$ \widetilde{\frakM}(v_1,v_2,w) \to (\frakX_\Pi(v_2) \times \frakM(v_2,w)) \underset{\frakX_\Pi(v_1) \times \frakX_\Pi(v_2,w)}{\times} \widetilde{\frakX}_\Pi(v_1,v_2,w) \to \widetilde{\frakX}_\Pi(v_1,v_2,w)$$
in which the first map is an open embedding. It follows by base change that $q$ is quasi-smooth as well. We now define an action of $\Y_{R_S}$ on $F_{w,R_{w,s}}$ by the formula $p_*q^!$.
The commutativity of the above diagram ensures that this indeed defines an action. 
The same argument works for the nilpotent versions.
 
\medskip

The representation of the $R_S$-algebra $\overline\Y_{R_S}^\vee$ on $F_{w,R_{w,S}}^\vee$ yields an 
$R_{w,S}$-linear map called the evaluation map
$$\ev_w^\vee:\overline\Y_{R_S}^\vee\to F_{w,R_{w,S}}^\vee
,\quad
1\mapsto \phi_w^\vee$$ 
It factorizes through 
the map \eqref{YY}. 
We'll consider all of these representations at once. We abbreviate
$$
F_{2,R_S}^\vee=\prod_{w\in \bbN I}F_{w,R_{w,S}}^\vee
,\quad
\ev_2^\vee=\prod_{w\in \bbN I} \ev_w^\vee : \overline\Y_{R_S}^\vee\to F_{2,R_S}^\vee.$$
When $S=T$ we'll omit the subscript $T$ everywhere. For instance, we set
$$R_\bfw=R_{\bfw,T}
,\quad
F_{\bfw,R_\bfw}=F_{\bfw,R_{\bfw,T}}
,\quad
R=R_T
,\quad
R_\infty=R_{\infty,T}$$
$$\overline\Y_R^\vee=\overline\Y_{R_T}^\vee
,\quad
\Y_R^\vee=\Y_{R_T}^\vee
,\quad
\Y^\vee_{v,k,R}=\Y^\vee_{v,k,R_T}
,\quad
\text{etc.}$$ 
We equip the set $\bbN I$ with the partial order given by
$$w'\,\geqsl\, w\iff w'_i\,\geqsl\, w_i,\,\forall i\in I.$$
We write 
\begin {align}\label{large}w\gg 0
\iff  w_i>n,\,\forall i\in I\end{align}
for some large enough integer $n$.
Then, we'll say that $w$ is large enough. 
Let 
$$I_w=(p_{i,0}-w_i\,;\,i\in I)$$
be the ideal in $R_{\infty,S}$ generated by the elements $p_{i,0}-w_i$ for all $i\in I$. 
Recall that
$$\overline\Y_{R_S}^\vee\otimes_{R_S} K_S=\overline\Y_{K_S}^\vee
,\quad 
\overline\Y_{R_S}\otimes_{R_S} K_S=\overline\Y_{K_S}.$$

\begin{proposition}\label{prop:COHA}
Let $v,w\in\bbN I$ and $l\in\bbN$.
\hfill
\begin{enumerate}[label=$\mathrm{(\alph*)}$,leftmargin=8mm,itemsep=1.2mm]
\item
The map $\ev_w^\vee$ is surjective.
\item
If $w$ is large enough
the map $\ev_w^\vee$ yields an isomorphism
\begin{align}\label{YF}
\overline\Y^\vee_{-v,l,R_S}\,/\,\overline\Y_{R_S}^\vee I_w\cap\overline\Y^\vee_{-v, l,R_S}
= F_{v,w, l, R_{w,S}}^\vee.
\end{align}
\item 
The map $\ev_2^\vee$ is injective.
\item
The base change  yields $R_S$-algebra isomorphisms
$$\overline\Y_R^\vee\otimes_R R_S=\overline\Y_{R_S}^\vee
,\quad 
\Y_R^\vee\otimes_R R_S=\Y_{R_S}^\vee.$$
\item
The pushforward by the inclusion $\X_\Pi^\vee(v)\subset \X_\Pi(v)$ is an
algebra isomorphism $\overline\Y^\vee_{K_T}=\overline\Y_{K_T}$.
\item
$\overline\Y_{R_S}$ and $\overline\Y_{R_S}^\vee$ are free $R_S$-modules.
\end{enumerate}
\end{proposition}

\begin{proof}
Part (a) follows from \cite[prop.~5.19]{SV18}.
Let us concentrate on (b). 
To do that, we consider the following quotient stack
$$M(v,w)=(\X_\Pi^\vee(v)\times\Hom_I(W,V))\,/\,G_v.$$
By \eqref{frakL} the Lagrangian $\frakL(v,w)$ is an open substack of $M(v,w)$.
We abbreviate 
$$\partial\frakL(v,w)=M(v,w)\setminus\frakL(v,w)$$
Let $e_{v,w}$ be the dimension of $\partial\frakL(v,w)$.
We have
$$H_l^S(M(v,w),\bbQ)=H_l^S(\frakL(v,w),\bbQ)
,\quad
l>2e_{v,w}$$
The Thom isomorphism yields a graded $R_{w,S}$-module isomorphism
$$H_\bullet^{G_w\times S}(M(v,w),\bbQ)=H_\bullet^S(M(v,w),\bbQ)\otimes_{R_S}R_{w,S}$$
Since the $G_w\times S$-variety $\frakL(v,w)$ is formal, we also have
a graded $R_{w,S}$-module isomorphism
$$F_{v,w,R_{w,S}}^\vee=H^S_\bullet(\frakL(v,w),\bbQ)\otimes_{R_S}R_{w,S}.$$
We deduce that if $l<d_{v,w}-e_{v,w}$ the restriction from $M(v,w)$ to the open subset $\calL(v,w)$ 
yields an isomorphism
$$H_{2d_{v,w}-2l}^{G_w\times S}(M(v,w),\bbQ)=F_{v,w,l,R_{w,S}}^\vee.
$$
On the other hand, we have
\begin{align*}
\Y^\vee_{-v,l,R_S}&=H^S_{-(v,v)-2l}(\X_\Pi^\vee(v)/G_v,\bbQ)
=H^S_{2d_{v,w}-2l}(M(v,w),\bbQ)
\end{align*}
Thus the restriction yields also an isomorphism
$$(\Y^\vee_{-v,R_S}\otimes_{R_S}R_{w,S})_l=F_{v,w,l,R_{w,S}}^\vee
,\quad
l<d_{v,w}-e_{v,w}$$
Further, we have
$$d_{v,w}-e_{v,w}\geqsl\dim M(v,w)-\dim\big(\{(x,f)\in\X_\Pi^\vee(v)\times\Hom_I(W,V))\,;\,f\ \text{not\ surjective}\}\,/\,G_v\big).$$
Hence, as $w\to\infty$ we have $d_{v,w}-e_{v,w}\to\infty$.
If $w$ is large enough then we also have
$$\overline\Y^\vee_{-v,l,R_S}\,/\,\overline\Y_{R_S}^\vee I_w\cap\overline\Y^\vee_{-v,l,R_S}=
(\Y^\vee_{-v,R_S}\otimes_{R_S}R_{w,S})_l.$$ 
We deduce that the evaluation maps yield the following isomorphisms for $w$ large enough
$$\ev_w^\vee:\overline\Y^\vee_{-v,l,R_S}\,/\,\overline\Y_{R_S}^\vee I_w\cap\overline\Y^\vee_{-v,l,R_S}\to
F_{v,w,l,R_{w,S}}^\vee$$ 
Part (c) is a consequence of (b).
Let us prove Part (d).
The obvious $R_S$-linear map 
$$\Y_R^\vee\otimes_R R_S\to\Y_{R_S}^\vee$$
is surjective by \cite[thm 5.18]{SV18}. This gives
a system of generators of the right hand side. Next, observe that $\Y_{R_S}^\vee$ is a free $R_S$-module by 
\cite[thm.~3.2]{SV18}. For each $v\in \bbN I$ we set
$$M_{v,R_S}=\Y^\vee_{v,R_S}\,/\,R_{S,+}\Y^\vee_{v,R_S}$$
where $R_{S,+}$ is the augmentation ideal of $R_S$.
This is an $\bbZ$-graded $\bbQ$-vector space whose graded dimension is given by 
\cite[thm 5.4]{SV18}.
In particular, the dimension of the subspace
$M_{v, l,R_S}$ is finite and independent of the choice of the subtorus $S$ of $T$.
We have the following commutative square
$$\xymatrix{
\Y^\vee_{v,R}\otimes_RR_S\ar@{->>}[r]\ar@{->>}[d]_-\pi&\Y^\vee_{v,R_S}\ar@{->>}[d]^-{\pi_S}\\
M_{v,R}\ar[r]&M_{v,R_S}
}$$
We deduce that the lower map is surjective. Hence it is invertible because
$$\dim M_{v, l,R}=\dim M_{v, l,R_S}.$$
Since the $R_S$ modules
$\pi^{-1}(M_{v, l,R})$ and $\pi_S^{-1}(M_{v, l,R_S})$ are free of the same rank with a surjective map
$\pi^{-1}(M_{v, l,R})\to\pi_S^{-1}(\dim M_{v, l,R_S})$ they are isomorphic.
Hence the upper map in the square above is also injective.
Part (e) is \cite[prop.~5.2]{SV18}.
Part (f) is \cite[thm.~3.2]{SV18}.
\end{proof}

\bigskip

\section{Yangians}

This section is a reminder from \cite{MO19}.
We first recall some basic facts on stable envelopes and R-matrices.
Then we define the Yangian $\bfY_R$ and the Lie algebra $\frakg_R$.
We also introduce some nilpotent analogues $\bfY_R^\vee$ and $\frakg_R^\vee$.

\medskip

\subsection{Stable envelopes and $R$-matrices}\label{sec:stable+R}
\subsubsection{Stable envelopes}\label{sec:stable}
Let $G$, $S$ be affine groups.
Let 
\begin{align}\label{pi}
\pi:X\to X_0
\end{align}
 be a $G\times S$-equivariant symplectic resolution of singularities.
The map $\pi$ is projective.
The variety $X$ is smooth and quasi-projective.
We'll assume that $X$ is $G\times S$-equivariantly formal.
The variety $X_0$ is affine.
We'll assume that $G$ fixes the symplectic form and that $S$ is a torus.
Let $\hbar:S\to\bbG_m$ be the weight of the symplectic form on $X$.
Unless specified otherwise we'll assume that the restriction of 
$\hbar$ to the torus $S$ is non trivial.
Let $A\subset G\times S$ be a central torus which fixes the symplectic form.
Let $Fix$ be the set of all connected components of the fixed points locus $X^A$.
%Let $\fraka_\bbR=X_*(A)\otimes_\bbZ\bbR$.
Let $\Delta\subset X^*(A)$ be the set of all weights of the normal bundles to $X^A$ in $X$.
Fix a cocharacter $\sigma$ of $A$.
Assume that $\sigma\notin\Delta$.
We say that a weight $\lambda\in X^*(A)$ is positive if $\lambda\circ\sigma$ is positive.
For each $F\in Fix$ the normal bundle $N_FX$ to $F$ in $X$ splits as a direct sum
$$N_FX=N_F^{\sigma,+}X\oplus N_F^{\sigma,-}X$$ according to the sign of the weights.
We have $N_F^{\sigma,+}X=\hbar^{-1}(N_F^{\sigma,-}X)^*$.
Hence, the bundles $N_F^{\sigma,+}X$, $N_F^{\sigma,-}X$ have the same ranks and the 
subsets of weights in $X^*(A)$ are related by the following formula
$$wt(N_F^{\sigma,+}X)=-\hbar-wt(N_F^{\sigma,-}X).$$
The leaf of a closed subset $F\subset X^A$ is the locally closed subset
of $X$ given by
$$X^{\sigma,+}_F=\{x\in X\,;\,\lim_{t\to 0}\sigma(t)\cdot x\in F\}.$$
Let $\preceq$ be the minimal partial order on the set $Fix$ such that
$F'\preceq F$ if the closure of $X^{\sigma,+}_F$ intersects $F'$.
We define the following closed subsets of $X$
\begin{align}\label{leaf}
X^{\sigma,+}_{\preceq F}=\bigsqcup_{F'\preceq F}X^{\sigma,+}_{F'}
,\quad
X^{\sigma,+}=\bigsqcup_{F\in Fix}X^{\sigma,+}_F.
\end{align}
The set $X^{\sigma,+}$ is called the attracting subvariety of $X$.
We have
$$X^{\sigma,+}=\pi^{-1}(X^{\sigma,+}_0).$$
Since $X_0$ is affine, there is a well-defined morphism $X_0^{\sigma,+}\to X_0^A$. 
Composing it with the map $\pi$ in \eqref{pi}, we get a morphism $X^{\sigma,+}\to X_0^A$.
We consider the fiber product
\begin{align}\label{A}
\calA^\sigma=X^{\sigma,+}\times_{X_0^A}X^A.
\end{align}
The set $\calA^\sigma$ is known to be a closed Lagrangian subvariety of $X\times X^A$,
where $X\times X^A$ is equipped with the antidiagonal symplectic form \cite[\S 3.5]{MO19}.
Set also
$$
\calA^\sigma_F=\calA^\sigma\cap(X\times F)
,\quad
\calA^{\sigma,+}_F=\{(x,x_0)\in X^{\sigma,+}_F\times F\,;\,\lim_{t\to 0}\sigma(t)\cdot x=x_0\}.
$$
%,\quad Z^{\sigma,+}_{\preceq F}=Z^{\sigma,+}\cap(X\times F).$$
Then, we have
$$\calA^{\sigma,+}_F\subset\calA^\sigma_F\subset X^{\sigma,+}_{\preceq F}\times F.$$ 
In particular,   if the connected component $F$ is minimal 
in the poset $Fix$ then
$$\calA^{\sigma,+}_F=\calA^\sigma_F.$$

By a polarization $\varepsilon$ we mean a choice
of a square root $\varepsilon_F\in H_A^\bullet$ of the monomial
$$(-1)^{\codim_XF/2}\eu(N_FX)|_{H_A^\bullet}$$
for each connected component $F\in Fix$.
Since $\varepsilon_F$ differs from $\eu(N_F^{\sigma,-}X)$ by some sign, we may view
the polarization as a collection of signs. We may write
$\varepsilon_F=\pm\eu(N_F^{\sigma,-}X)$.
Assume that we have fixed a splitting of $S$-equivariant vector bundles
\begin{align}\label{splitting}
TX=T^{\frac{1}{2}}X\oplus\hbar^{-1}(T^{\frac{1}{2}}X)^*.
\end{align}
Then, there is a polarization such that $\varepsilon_F$ is the product of
all $A$-weights of the vector bundle 
$$N^{\frac{1}{2}}_FX=T^{\frac{1}{2}}X|_F\cap N_FX.$$
Equip $H_\bullet^{G\times S}(X^A,\bbQ)$ 
with the increasing filtration by the $A$-degree given by the decomposition
$$H_\bullet^{G\times S}(X^A,\bbQ)=H_\bullet^{G\times S/A}(X^A,\bbQ)\otimes H^\bullet_A$$
According to \cite[prop.~3.5.1-2]{MO19}, there is a unique Lagrangian cycle 
$$\calS^\sigma\in H_\bullet^{G\times S}(X\times X^A,\bbQ)$$ 
such that $\calS^\sigma$ is supported on $\calA^\sigma$ and
\hfill
\begin{enumerate}[label=$\mathrm{(\alph*)}$,leftmargin=8mm,itemsep=1mm]
\item 
the Gysin restriction of $\calS^\sigma$ to $X\times F$ is supported on $\calA^\sigma_F$,
\item
the Gysin restriction of $\calS^\sigma$ to $F\times F$ is $\pm\eu(N_F^{\sigma,-}X)\cap[\Delta_F]$,
\item
the Gysin restriction of $\calS^\sigma$ to $F'\times F$ has $A$-degree $<\codim_XF'$ for each $F'\prec F$.
\end{enumerate}
Let $\calS^\sigma_F$ be the Gysin restriction of the cycle $\calS^\sigma$ to $X\times F$.
The cycle $\calS^\sigma$ is the sum of all $\calS^\sigma_F$'s.
The cycle $\calS^\sigma$ is proper over $X$. 
The convolution by $\calS^\sigma$ defines an $R_{G\times S}$-linear map 
$$\stab(\sigma):H^{G\times S}_\bullet(X^A,\bbQ)\to H^{G\times S}_\bullet(X,\bbQ)$$
From (a) and (b), we deduce that $\stab(\sigma)$ is invertible if 
the equivariant Euler class $\eu(N_FX)$ is invertible for each $F\in Fix$.

\subsubsection{Stable envelopes of quiver varieties}\label{sec:stableQV}
We abbreviate 
$\bfN_0=(\bbN I)^{(\bbN)}$.
Fix $\bfw\in\bfN_0$. Set $w=|\bfw|$.
We write
\begin{align*}
F_{w,R_{\bfw,S}}&=F_{w,R_{w,S}}\otimes_{R_{w,S}}R_{\bfw,S},\\
F_{\bfw,K_{\bfw,S}}&=F_{\bfw,R_{\bfw,S}}\otimes_{R_{\bfw,S}}K_{\bfw,S},\\
A_{\bfw,K_{\bfw,S}}&=A_{\bfw,R_{\bfw,S}}\otimes_{R_{\bfw,S}}K_{\bfw,S}
,\quad
\text{etc}.
\end{align*}
We apply the previous constructions to the torus $S$ in \S\ref{sec:quivers} and the varieties 
$$X=\frakM(w)
,\quad
X_0=\frakM'_0(w).$$
The weight of the symplectic form on $\frakM(w)$ is
$$\hbar(z_\alpha,z)=z
,\quad
(z_\alpha,z)\in T.$$
Recall that that the restriction of 
$\hbar$ to the subtorus $S$ is assumed to be non trivial.
Set $\bfw=(w_1,w_2,\dots,w_s)$.
We consider the cocharacter of $G_\bfw$ given by
\begin{align}\label{sigma}
\sigma(z)=\bigoplus_{r=1}^sz^{s-r+1}\id_{W_r}
,\quad
z\in\bbC^\times.
\end{align}
Set $A=\sigma(\bbG_m)$. The fixed points locus  is 
$$\frakM(w)^\sigma=\frakM(w)^A=\frakM(\bfw).$$
By \cite[(3.4)]{MO19}, the partial order on the set $Fix$ of connected components of the fixed point locus
is such that, for each
$\bfv=(v_1,v_2)$, $\bfv'=(v'_1,v'_2)$ and
$\bfw=(w_1,w_2)$,
\begin{align}\label{order}
\frakM(\bfv',\bfw)\;\prec\;\frakM(\bfv,\bfw)
\Rightarrow
v_1+v_2=v'_1+v'_2\,,\,v'_1< v_1
\end{align}
Let $\frakS_s$ be the permutation group of  $\{1,2,\dots,s\}$.
Given $\omega\in\frakS_s$,
let $\omega\sigma$ be the cocharacter 
$$\omega\sigma(z)=\bigoplus_{r=1}^sz^{s-r+1}\id_{W_{\omega(r)}}
,\quad
z\in\bbC^\times.$$
There is a canonical splitting of the tangent bundle of $\frakM(w)$ as in \eqref{splitting},
which is given in \cite[\S 2.2.7, ex.~3.3.3]{MO19}. 
We choose the polarization associated with this splitting. 
We abbreviate
$$\stab(\omega)=\stab(\omega\sigma)
,\quad
\calA^\omega=\calA^{\omega\sigma}
,\quad
\calS^\omega=\calS^{\omega\sigma}
,\quad
\omega\bfw=(w_{\omega(1)},w_{\omega(2)},\dots,w_{\omega(s)})
,\quad
\text{etc.}$$
The convolution with $\calS^{\omega}$ is the stable envelope
\begin{align}\label{defstable}
\begin{split}
\stab(\omega) :F_{\bfw,R_{\bfw,S}}\to F_{w,R_{\bfw,S}}.
\end{split}
\end{align}
Since the stable envelope is the convolution by a Lagrangian cycle,
it is homogeneous of degree 0 for the cohomological grading of
$F_{\bfw,R_{\bfw,S}}$ and $F_{w,R_{\bfw,S}}$ introduced in \eqref{cohdeg}.
%It is also homogeneous of degree 0 for the cohomological grading.
We consider the cycle $(\calS^\omega)^\op$ in $\frakM(\bfw)\times\frakM(w)$ given by
permuting the factors of $\calS^{\omega}$.
The cycle $(\calS^\omega)^\op$ 
is not proper over $\frakM(\bfw)$, but the restriction to the $A$-fixed locus is the cycle
$$(\calS^\omega)^\op|_{\frakM(\bfw)\times\frakM(\bfw)}$$
which is proper over $\frakM(\bfw)$, because it is supported on the fiber product
$$\calA^\omega\cap(\frakM(\bfw)\times\frakM(\bfw))=\frakM(\bfw)\times_{\frakM'_0(w)}\frakM(\bfw).$$
Thus $(\calS^\omega)^\op$ acts by convolution on a localized version of equivariant homology.
This yields an operator
$$\stab(\omega)^\op:F_{w,K_{\bfw,S}}\to F_{\bfw,K_{\bfw,S}}.$$
For a similar reason $\stab(\omega)$ is generically invertible, yieldding an operator
$$\stab(\omega)^{-1}:F_{w,K_{\bfw,S}}\to F_{\bfw,K_{\bfw,S}}.$$
Maulik and Okounkov define
the R-matrix to be the operator in $A_{\bfw,K_{\bfw,S}}$ given by
$$\R_{F_\bfw}(\omega_1,\omega_2)=\stab(\omega_1)^{-1}\circ\stab(\omega_2).$$
Let $\omega_0$ be the longest permutation of $\{1,2,\dots,s\}$.
For any permutation $\omega$ we abbreviate $$\overline\omega=\omega_0\omega.$$
By \cite[thm~4.4.1]{MO19} we have 
\begin{align}\label{stabstab}\stab(\overline\omega)^\op=\stab(\omega)^{-1}.
\end{align}
We set
\begin{align}\label{R1}
\R_{F_\bfw}=\R_{F_\bfw}(\omega_0,1)=\stab(1)^\op\circ\stab(1).
\end{align}
The R-matrix $\R_{F_\bfw}$
is the convolution with the following cycle in localized equivariant Borel Moore homology
$$(\calS^1)^\op\star\calS^1\in 
H_\bullet^{G_\bfw\times S}(\frakM(\bfw)\times_{\frakM'_0(w)}\frakM(\bfw),\bbQ)_{K_{\bfw,S}}.$$
%Exchanging the cocharacters $\sigma$ and $\tau$ we define in the same way the R-matrix
%$\R_{F_\bfw}(\tau)$. We have $$\R_{F_\bfw}(\tau)=\R_{F_\bfw}(\sigma)^{-1}.$$
We write
\begin{align}\label{R2}
\begin{split}
\R_{F_w,F_\bfw}=
\R_{F_{w,w_1}}
\R_{F_{w,w_2}}
\cdots\R_{F_{w,w_s}}
,\quad
w\in\bbN I
,\,
\bfw\in\bfN_0.
\end{split}
\end{align}
By \cite[(4.6)]{MO19}, we have $\R_{F_\bfw}(\omega_s,1)=\R_{F_{w_1},F_{\bfw'}}$ where 
$\bfw'=(w_2,w_3,\dots,w_s)$ and $\omega_s$ is the $s$-cycle
\begin{align*}
\omega_s=(s,\dots,2,1)
\end{align*}

\subsubsection{Stable envelopes of nilpotent quiver varieties}
The cycle $\calS^{\omega}$ is supported in
$$\calA^{\omega}=\frakM(w)^{\omega,+}\times_{\frakM'_0(w)}\frakM(\bfw).$$
Thus, the cycle $(\calS^{\overline\omega})^{\op}$ is supported in 
$$(\calA^{\overline\omega})^\op=\frakM(\bfw)\times_{\frakM'_0(w)}\frakM(w)^{\overline\omega,+}$$
The closed subset
$$(\calA^{\overline\omega})^\op\cap(\frakM(\bfw)\times\frakL(w))\subset
\frakM(\bfw)\times\frakM(w)$$ is proper over $\frakM(\bfw)$, and the first projection maps into $\frakL(\bfw)$.
Hence, the convolution with $(\calS^{\overline\omega})^\op$
yields a morphism of $\bbZ I\times\bbZ$-graded $R_{\bfw,S}$-modules
\begin{align}\label{defstablevee}
\begin{split}
\stab^\vee(\omega) :F_{w,R_{\bfw,S}}^\vee\to F_{\bfw,R_{\bfw,S}}^\vee
\end{split}
\end{align}

%%%%%%%%%%%%%%%%%%%%%%%

\begin{lemma}\label{lem:ss}
\hfill
\begin{enumerate}[label=$\mathrm{(\alph*)}$,leftmargin=8mm,itemsep=1.2mm]
\item 
$\stab^\vee(\omega)$ and $\stab(\omega)$ are injective and generically invertible over $R_{\bfw,S}$.
\item
$\stab^\vee(\overline\omega)$ is the transpose of $\stab(\omega)$ relatively to the pairing \eqref{PP1}.
\item
$\stab^\vee(\omega)$ and $\stab(\omega)^{-1}$ are intertwined by the map
$\iota : F_{\bfw,K_{\bfw,S}}^\vee\to F_{\bfw,K_{\bfw,S}}$ in \eqref{iota1}.
\end{enumerate}
\end{lemma}

\begin{proof}
Part (a) follows from the fact that $F_{w,R_{\bfw,S}}^\vee$ and $F_{\bfw,R_{\bfw,S}}$ are free modules over 
$R_{\bfw,S}$, and the $R_{\bfw,S}$-linear maps $\stab^\vee(\omega)$ and $\stab(\omega)$ become
invertible after a localization $R'_{\bfw,S}$
of $R_{\bfw,S}$ such that the Euler class of the normal bundle $N_{\frakM(\bfw)}\frakM(w)$ is invertible
in $$H^\bullet_{G_\bfw\times S}(\frakM(\bfw),\bbQ)\otimes_{R_{\bfw,S}}R'_{\bfw,S}.$$
Part (b) follows from the fact that the transpose under the pairing \eqref{PP1} of the map 
$F_{\bfw,R_{\bfw,S}}\to F_{w,R_{\bfw,S}}$ given by the convolution product with a cycle 
$$\calS\in H_\bullet^{G_\bfw\times S}(\frakM(w)\times_{\frakM'_0(w)}\frakM(\bfw),\bbQ)$$ is the map 
$F_{w,R_{\bfw,S}}^\vee\to F_{\bfw,R_{\bfw,S}}^\vee$ given by the convolution product with the cycle
$$\calS^\op\in H_\bullet^{G_\bfw\times S}(\frakM(\bfw)\times_{\frakM'_0(w)}\frakM(w),\bbQ)$$ 
given by permuting the factors of $\calS$.
Part (c) follows from \eqref{stabstab}, because the pushforward by the closed embedding
$\frakL(\bfw)\subset\frakM(\bfw)$ commutes with the convolution product.
\end{proof}

\iffalse%%%%%%%%%%%%%%%%%%
it is enough to observe that since 
$\zeta(\omega)$ and $\zeta^\vee(\omega)$ are both given by
convolution with the cycle $\calS^{\omega}$, we have the commutative diagram
\begin{align}
\label{FZF}
\begin{split}
\xymatrix{
F_{\bfw,R_{\bfw,S}}\ar@{=}[r]^-{\zeta(\omega)}
& Z_{\bfw,\omega,R_{\bfw,S}}\ar[r]^-{i_*}
&F_{w,R_{\bfw,S}}\\
F_{\bfw,R_{\bfw,S}}^\vee\ar@{=}[r]^-{\zeta(\omega)^\vee}\ar[u]^-\iota
& Z_{\bfw,\omega,R_{\bfw,S}}^\vee\ar[u]
&\ar[l]_-{j_*}F_{w,R_{\bfw,S}}^\vee\ar[u]_-\iota
}
\end{split}
\end{align}
The vertical maps are the pushforward maps by the obvious inclusions.
The composed horizontal maps are $\stab(\omega)$ and $\stab^\vee(\omega)$.
To prove Part (c) we consider the diagram \eqref{FZF}.
We must check that the transpose of 
$i_*$ and $\zeta(\omega)$ relatively to the pairings \eqref{PP1} and \eqref{PP3}
are $\zeta^\vee(\overline\omega)^{-1}$ and $j_*$ respectively.
\fi%%%%%%%%%%%%%%%%%%

The R-matrix of the nilpotent quiver varieties is the operator in $A_{\bfw,K_{\bfw,S}}^\vee$ given by
$$\R_{F_\bfw^\vee}(\omega_1,\omega_2)=\stab^\vee(\omega_1)\circ\stab^\vee(\omega_2)^{-1}.$$
The conjugation by the isomorphism $\iota : F_{\bfw,K_{\bfw,S}}^\vee\to F_{\bfw,K_{\bfw,S}}$
 yields a $K_{\bfw,S}$-algebra isomorphism 
$$A_{\bfw,K_{\bfw,S}}=A_{\bfw,K_{\bfw,S}}^\vee.$$ 
The transpose \eqref{T1} yields
a $K_{\bfw,S}$-algebra anti-isomorphism 
$$(-)^\T:A_{\bfw,K_{\bfw,S}}\to A_{\bfw,K_{\bfw,S}}^\vee.$$

\begin{lemma}\label{lem:RR}
The following hold.
\hfill
\begin{enumerate}[label=$\mathrm{(\alph*)}$,leftmargin=8mm,itemsep=1.2mm]
\item 
$\iota\circ\R_{F_\bfw^\vee}(\omega_1,\omega_2)=\R_{F_\bfw}(\omega_1,\omega_2)\circ\iota$.
\item 
$\R_{F_\bfw^\vee}(\omega_1,\omega_2)
=\R_{F_\bfw}(\overline\omega_2,\overline\omega_1)^\T
=(\R_{F_\bfw}(\overline\omega_1,\overline\omega_2)^\T)^{-1}$.
\end{enumerate}
\end{lemma}

\begin{proof} The lemma follows from Lemma \ref{lem:ss} and the obvious identification of
operators in $A_{\bfw,K_{\bfw,S}}$
$$\R_{F_\bfw}(\overline\omega_2,\overline\omega_1)
=\R_{F_\bfw}(\overline\omega_1,\overline\omega_2)^{-1}.$$
\end{proof}

Following \eqref{R1} and \eqref{R2}, we write
\begin{align}\label{Rvee}
\begin{split}
\R_{F_\bfw^\vee}&=\R_{F_\bfw^\vee}(\omega_0,1),\quad 
\bfw\in\bfN_0\\
\R_{F_w^\vee,F_\bfw^\vee}&=
\R_{F_{w,w_1}^\vee}
\R_{F_{w,w_2}^\vee}
\cdots\R_{F_{w,w_s}^\vee}
,\quad
w\in \bbN I
.
\end{split}
\end{align}
We have 
$\R_{F_\bfw^\vee}=(\R_{F_\bfw})^\T$
and
$\R_{F_w^\vee,F_\bfw^\vee}=(\R_{F_w,F_{\overline\bfw}})^\T$
where $\overline\bfw=\omega_0\bfw.$

\medskip

\subsection{The Yangian}
\subsubsection{Definition of the Yangian}\label{sec:defY}
We define
\begin{align}\label{M}
M_{R_S}=\bigoplus_{i\in I}A_{\delta_i,R_S}.
\end{align}
Assigning to $u$ the 1st Chern class
of the linear character of the group $G_{\delta_i}$
yields an isomorphism $R_{\delta_i,S}=R_S[u]$
for any vertex $i\in I$.
Since $F_{\delta_i,R_{\delta_i,S}}=H_\bullet^{G_{\delta_i}\times S}(\frakM(\delta_i),\bbQ)$
and $G_{\delta_i}$ acts trivially on $\frakM(\delta_i)$, the Kunneth formula yields a canonical
$R_{\delta_i,S}$-linear isomorphism 
\begin{align}\label{isom1}F_{\delta_i,R_{\delta_i,S}}=F_{\delta_i,R_S}[u].
\end{align}
Hence, we have
$$M_{R_S}[u]=\bigoplus_{i\in I}A_{\delta_i,R_{\delta_i,S}}.$$
Let $U(M_{R_S}[u])$ be the tensor $R_S$-algebra of the $R_S$-module $M_{R_S}[u]$
$$U(M_{R_S}[u])=\bigoplus_{n\in\bbN}M_{R_S}[u]^{\otimes n}$$
The $R_S$-algebra $M_{R_S}$ is $\bbZ I\times\bbZ$-graded
by the weight and the cohomological grading.
Hence $U(M_{R_S}[u])$ is also $\bbZ I\times\bbZ$-graded, with
$u$ of weight 0 and cohomological degree 1.
Further, the $\bbZ$-grading on $M_{R_S}[u]$ given by the 
$u$-degree yields a $u$-grading on the tensor algebra $U(M_{R_S}[u])$.
For each $\bfw\in\bfN_0$, $m(u)\in M_{R_S}[u]$,
we consider the operator on $F_{\bfw,K_{\bfw,S}}$ given by
\begin{align}\label{E}
E_{\bfw,R_S}(m(u))=-\frac{1}{\hbar}\sum_{i\in I}\Res_u(\Tr\otimes 1)(m(u)\R_{F_{\delta_i},F_\bfw}),
\end{align}
where $\Res_u$ is the residue at $u=0$.
The trace is well-defined because the weight subspaces of $F_{\delta_i,R_S}$ are finitely generated free 
$R_S$-modules.
The following is proved in \cite[\S\S 4-5]{MO19}.

\begin{lemma}\label{lem:E}
The map $E_{\bfw,R_S}$ yields a $\bbZ I\times\bbZ$-graded $R_S$-algebra homomorphism
$$E_{\bfw,R_S}:U(M_{R_S}[u])\to A_{\bfw,R_{\bfw,S}}$$
\qed
\end{lemma}

We consider the subsets of $\bfN_0$ given by
$$\bfN_1=I^{(\bbN)}
,\quad
\bfN_2=\bbN I
.$$ 
For each $a=0,1,2$  we define 
$$A_{a,R_S}=\prod_{\bfw\in \bfN_a}A_{\bfw,R_{\bfw,S}}
,\quad
F_{a,R_S}=\prod_{\bfw\in\bfN_a}F_{\bfw,R_{\bfw,S}}
,\quad
E_{a,R_S}=\prod_{\bfw\in \bfN_a}E_{\bfw,R_S}
.$$
The map $E_{a,R_S}$
is an $R_S$-algebra homomorphism $U(M_{R_S}[u])\to A_{a,R_S}$.
We define the Yangian $\bfY_{R_S}$ to be the image of the map $E_{0,R_S}$.
We'll use the following results from \cite{MO19}.

\begin{lemma}[\!{\cite[\S 4.1, \S 5.2]{MO19}}]\label{lem:rho} 
Let $\bfw\in \bfN_a$ and $a=0,1,2$.
\hfill
\begin{enumerate}[label=$\mathrm{(\alph*)}$,leftmargin=8mm,itemsep=1.2mm]
\item 
The map $E_{\bfw,R_S}$ factors through a $\bbZ I\times\bbZ$-graded
$R_S$-algebra homomorphism
$\bfY_{R_S}\to A_{\bfw,R_{\bfw,S}}$.
\item
The map $E_{a,R_S}$ factors through a $\bbZ I\times\bbZ$-graded
$R_S$-algebra embedding
$\bfY_{R_S}\to A_{a,R_S}$.
\item
The R-matrix
has a formal series expansion 
\begin{align}\label{Rii}
\R_{F_{\delta_i},F_{\delta_j}}=1+\sum_{l\geqsl 0}\frac{\hbar}{(u_i-u_j)^{l+1}}R_{F_{\delta_i},F_{\delta_j},l}
,\quad
R_{F_{\delta_i},F_{\delta_j},l}\in A_{(\delta_i,\delta_j),R_S}.
\end{align}
\qed
\end{enumerate}
\end{lemma}

The Yangian commutes with base change. More precisely, we have the following.

\begin{lemma}\label{lem:bcY} There is a $\bbZ I\times\bbZ$-graded $R_S$-algebra isomorphism
$\bfY_R\otimes_R R_S=\bfY_{R_S}.$
\end{lemma}

\begin{proof}
The equivariant formality of the Nakajima quiver varieties implies that
$$F_{\bfw,R_\bfw}\otimes_RR_S=F_{\bfw,R_{\bfw,S}}
,\quad
A_{\bfw,R_\bfw}\otimes_RR_S=A_{\bfw,R_{\bfw,S}}
,\quad
A_{0,R}\otimes_RR_S=A_{0,R_S}.$$
It also implies that
$$M\otimes_RR_{S}=M_{R_S}
,\quad
U(M_R[u])\otimes_RR_S=U(M_{R_S}[u]).$$
Further, the uniqueness of stable envelopes implies that
the map $E_{0,R_S}$ is obtained by base change from the map
$E_{0,R}$. 
We deduce that
$$\bfY_R\otimes_R R_S=\bfY_{R_S}.$$
\end{proof}

Because of Lemma~\ref{lem:bcY} we may concentrate on the particular case where $S=T$. 
Then, we omit the subscript $T$ everywhere. 
For instance, we abbreviate 
$$\bfY_R=\bfY_{R_T}
,\quad
M_R=M_{R_T}
,\quad
\text{etc.}$$
We'll be mostly interested in the particular case where 
$\bfw=(\delta_{i_1},\delta_{i_2}, ...,\delta_{i_s})$
is a tuple in $\bfN_1$.
Then, we abbreviate 
$\bfw=(i_1,i_2,\dots,i_s)$.
In this case, the isomorphisms \eqref{isom1} yield
\begin{align}\label{Fw}
R_{\bfw,S}=R_S[u_1,\dots,u_s]
,\quad
F_{\bfw,R_{\bfw,S}}=F_{\bfw,R_S}[u_1,\dots,u_s].
\end{align}

\subsubsection{Definition of the nilpotent Yangian}\label{sec:defYvee}
Set
\begin{align}\label{Mvee}
M_{R_S}^\vee=\bigoplus_{i\in I}A_{\delta_i,R_S}^\vee.
\end{align}
Since $F_{\delta_i,R_{\delta_i,S}}^\vee=H_\bullet^{G_{\delta_i}\times S}(\frakL(\delta_i),\bbQ)$
and $G_{\delta_i}$ acts trivially on $\frakL(\delta_i)$, the Kunneth formula yields a canonical
$R_{\delta_i,S}$-linear isomorphism 
\begin{align}\label{isom2}
F_{\delta_i,R_{\delta_i,S}}^\vee=F_{\delta_i,R_S}^\vee[u].
\end{align}
If $\bfw\in \bfN_1$ then the isomorphisms \eqref{isom2} yield an isomorphism
\begin{align}\label{Fwvee}
F_{\bfw,R_{\bfw,S}}^\vee=F_{\bfw,R_S}^\vee[u_1,\dots,u_s].
\end{align}
Let $U(M_{R_S}^\vee[u])$ be the tensor algebra of the $R_S$-module $M_{R_S}^\vee[u]$.
For each $\bfw\in\bfN_0$ and $m^\vee(u)\in M_{R_S}^\vee[u]$,
we consider the operator on $F_{\bfw,K_{\bfw,S}}^\vee$ given by
\begin{align}\label{Evee}
E_{\bfw,R_S}^\vee(m^\vee(u))=
-\frac{1}{\hbar}\sum_{i\in I}\Res_u(\Tr\otimes 1)(m^\vee(u)\R_{F_{\delta_i}^\vee,F_{\bfw}^\vee}).
\end{align}

\begin{lemma}\label{lem:Evee}
The map $E_{\bfw,R_S}^\vee$ yields a $\bbZ I\times\bbZ$-graded $R_S$-algebra homomorphism
$$E_{\bfw,R_S}^\vee:U(M_{R_S}^\vee[u])\to A_{\bfw,R_{\bfw,S}}^\vee$$
\end{lemma}

\begin{proof}
The $R_S$-modules $U(M_{R_S}[u])$ and $U(M_{R_S}^\vee[u])$ are free and we have
\begin{align}\label{URK}
\begin{split}
U(M_{R_S}[u])\otimes_{R_S}K_S=U(M_{K_S}[u])=U(M^\vee_{K_S}[u])=U(M_{R_S}^\vee[u])\otimes_{R_S}K_S.
\end{split}
\end{align}
See Remark \ref{rem:RS} for details.
By Lemma \ref{lem:F} the $R_{\bfw,S}$-modules
 $A_{\bfw,R_{\bfw,S}}$ and $A_{\bfw,R_{\bfw,S}}^\vee$ are free. Further, we have
\begin{align*}
A_{\bfw,R_{\bfw,S}}\otimes_{R_{\bfw,S}}K_{\bfw,S}=A_{\bfw,K_{\bfw,S}}=
A_{\bfw,K_{\bfw,S}}^\vee=A_{\bfw,R_{\bfw,S}}^\vee\otimes_{R_{\bfw,S}}K_{\bfw,S}.
\end{align*}
By base change, to prove the lemma 
it is enough to check that the map
$E_{\bfw,K_S}^\vee$
is a $K_S$-algebra homomorphism.
This follows from Lemma \ref{lem:E} because $E_{\bfw,K_S}^\vee=E_{\bfw,K_S}$.
\end{proof}

We abbreviate
$$A_{a,R_S}^\vee=\prod_{\bfw\in \bfN_a}A_{\bfw,R_{\bfw,S}}^\vee
,\quad
F_{a,R_S}^\vee=\prod_{\bfw\in\bfN_a}F_{\bfw,R_{\bfw,S}}^\vee
,\quad
E_{a,R_S}^\vee=\prod_{\bfw\in \bfN_a}E_{\bfw,R_S}^\vee
,\quad
a=0,1,2.$$
The map $E_{a,R_S}^\vee$
is an $R_S$-algebra homomorphism $U(M_{R_S}^\vee[u])\to A_{a,R_S}^\vee$.
We define the nilpotent Yangian
$\bfY_{R_S}^\vee$ to be the image of the map $E_{0,R_S}^\vee$.
We also define
$$\bfY_{K_S}=\bfY_{R_S}\otimes_{R_S}K_S
,\quad
\bfY_{K_S}^\vee=\bfY_{R_S}^\vee\otimes_{R_S}K_S.$$

\begin{lemma}\label{lem:rhovee} 
\hfill
\begin{enumerate}[label=$\mathrm{(\alph*)}$,leftmargin=8mm,itemsep=1.2mm]
\item 
There is a $\bbZ I\times\bbZ$-graded $R_S$-algebra isomorphism
$\bfY_R^\vee\otimes_R R_S=\bfY_{R_S}^\vee.$
\item
There is a $\bbZ I$-graded $K_S$-algebra isomorphism
$\bfY_{K_S}^\vee=\bfY_{K_S}.$
\item 
$E_{\bfw,R_S}^\vee$ factors through a $\bbZ I\times\bbZ$-graded
$R_S$-algebra homomorphism
$\bfY_{R_S}^\vee\to A_{\bfw,R_{\bfw,S}}^\vee$
\item
$E_{a,R_S}^\vee$ factors through a $\bbZ I\times\bbZ$-graded
$R_S$-algebra embedding
$\bfY_{R_S}^\vee\to A_{a,R_S}^\vee$
\item
The R-matrix has a formal series expansion 
\begin{align}\label{Riivee}
\R_{F_{\delta_i}^\vee,F_{\delta_j}^\vee}=
1+\sum_{l\geqsl 0}\frac{\hbar}{(u_i-u_j)^{l+1}}R_{F_{\delta_i}^\vee,F_{\delta_j}^\vee,l}
,\quad
R_{F_{\delta_i}^\vee,F_{\delta_j}^\vee,l}\in A_{(\delta_i,\delta_j),R_S}^\vee.
\end{align}
\item
$\iota\circ R_{F_{\delta_i}^\vee,F_{\delta_j}^\vee,l}=R_{F_{\delta_i},F_{\delta_j},l}\circ\iota.$
\item
$R_{F_{\delta_i}^\vee,F_{\delta_j}^\vee,l}=(R_{F_{\delta_j},F_{\delta_i},l})^\T.$
\end{enumerate}
\end{lemma}

\begin{proof}
Part (a) is proved as in Lemma \ref{lem:bcY}.
Parts (c), (d) are proved as in Lemma \ref{lem:rho}.
To prove (b), note that $A_{\bfw,K_{\bfw,S}}^\vee=A_{\bfw,K_{\bfw,S}}$ for each tuple $\bfw$,
and that
\begin{align*}
\bfY_{K_S}^\vee&=E_{0,K_S}^\vee(U(M_{K_S}^\vee[u]))
,\\
\bfY_{K_S}&=E_{0,K_S}(U(M_{K_S}[u])).
\end{align*}
Thus the claim follows from Lemma \ref{lem:RR} and \eqref{URK}.
Part (e) follows from Lemma \ref{lem:RR} and \eqref{Rii}.
Part (f) follows from Lemma \ref{lem:RR}.
Part (g) follows from Lemma \ref{lem:RR}.
% and the unitarity of the R-matrices  proved in \cite[prop.~4.5.1]{MO19}.
\end{proof}

\begin{lemma}
For each $a=0,1,2$ we have 
\begin{align*}
\bfY_{R_S}&=U(M_{R_S}[u])\,/\,\Ker (E_{a,R_S}),\\
\bfY_{R_S}^\vee&=U(M_{R_S}^\vee[u])\,/\,\Ker (E_{a,R_S}^\vee).
\end{align*}
\end{lemma}

\begin{proof}
By definition, both equalities hold for $a=0$.
Further, for $a=1,2$ we have
\begin{align*}
&\Ker(E_{0,R_S})\subset \Ker(E_{a,R_S})
,\\
&\Ker(E_{0,R_S}^\vee)\subset \Ker(E_{a,R_S}^\vee)
\end{align*}
Set $w=|\bfw|$. 
By \eqref{Delta} and \eqref{stab} below, we have a commutative square
\begin{align}
\begin{split}
\xymatrix{
F_{\bfw,R_{\bfw,S}}\ar[rr]^-{\stab(1)}&&F_{w,R_{\bfw,S}}\\
F_{\bfw,R_{\bfw,S}}\ar[rr]^-{\stab(1)}\ar[u]^-{E_{\bfw,R_S}(x)}&&F_{w,R_{\bfw,S}}\ar[u]_-{E_{w,R_S}(x)}}
\end{split}
\end{align}
Since the horizontal maps are generically invertible by Lemma \ref{lem:ss},
we deduce that 
$$E_{\bfw,R_S}(x)=0\iff E_{w,R_S}(x)=0.$$
This yields 
$$\Ker(E_{0,R_S})= \Ker(E_{1,R_S})=\Ker(E_{2,R_S}).$$
In a similar way we prove that
$$\Ker(E_{0,R_S}^\vee)= \Ker(E_{1,R_S}^\vee)=\Ker(E_{2,R_S}^\vee).$$
\end{proof}
We denote the $\bbZ I\times\bbZ$-graded $R_S$-algebra 
homomorphisms in Lemmas \ref{lem:rho}, \ref{lem:rhovee} by
\begin{align}\label{rhorho}
\begin{split}
\rho_\bfw&:\bfY_{R_S}\to A_{\bfw,R_{\bfw,S}},\\
\rho_a&:\bfY_{R_S}\to A_{a,R_S}
\end{split}
\end{align}
and by 
\begin{align}\label{rhorhovee}
\begin{split}
\rho_\bfw^\vee&:\bfY_{R_S}^\vee\to A_{\bfw,R_{\bfw,S}}^\vee,\\
\rho_a^\vee&:\bfY_{R_S}^\vee\to A_{a,R_S}^\vee.
\end{split}
\end{align}

\subsubsection{The Lie algebras $\frakg_R$ and $\frakg_R^\vee$}\label{sec:g}
The classical r-matrices are the operators
\begin{align*}
\bfr_{\bfw',\bfw}&=\sum_{r'=1}^{s'}\sum_{r=1}^sR_{F_{\delta_{j_{r'}}},F_{\delta_{i_r}},0}
,\\
\bfr_{\bfw',\bfw}^\vee&=\sum_{r'=1}^{s'}\sum_{r=1}^sR_{F_{\delta_{j_{r'}}}^\vee,F_{\delta_{i_r}}^\vee,0}
\end{align*}
for each tuples $\bfw=(i_1,\dots,i_s)$ and $\bfw'=(j_1,\dots,j_{s'})$ in $\bfN_1$.

\begin{lemma}\label{lem:sym}
\hfill
\begin{enumerate}[label=$\mathrm{(\alph*)}$,leftmargin=8mm,itemsep=1.2mm]
\item 
$\bfr_{\bfw',\bfw}=\bfr_{\bfw,\bfw'}$ after identifying $F_{\bfw'\bfw,R_S}$ with 
$F_{\bfw\bfw',R_S}$.
\item
$\bfr_{\bfw',\bfw}^\vee=-(\bfr_{\bfw',\bfw})^\T$.
\item
$\iota\circ\bfr_{\bfw',\bfw}^\vee=\bfr_{\bfw',\bfw}\circ\iota.$
\end{enumerate}
\end{lemma}

\begin{proof}
Part (a) is the symmetry of classical r-matrices proved in \cite[\S5.3.3]{MO19}.
Part (b) follows from Lemma \ref{lem:rhovee} and (a).
Part (c) follows from Lemma \ref{lem:rhovee}.
\end{proof}

We abbreviate 
$
\bfr_\bfw=\sum_{i\in I}\bfr_{i,\bfw}$
and
$
\bfr_\bfw^\vee=\sum_{i\in I}\bfr_{i,\bfw}^\vee.$
Lemma \ref{lem:rho} and \eqref{E} yields
\begin{align}\label{e}
E_{\bfw,R_S}(m)=-(\Tr\otimes 1)(m\,\bfr_\bfw)
,\quad
m\in M_{R_S}.
\end{align}
Lemma \ref{lem:rhovee} and \eqref{Evee} yield
\begin{align}\label{evee}
E_{\bfw,R_S}^\vee(m^\vee)=-(\Tr\otimes 1)(m^\vee\,\bfr_\bfw^\vee)
,\quad
m^\vee\in M_{R_S}^\vee.
\end{align}
We define the $R_S$-submodules $\frakg_{R_S}\subset\bfY_{R_S}$ and 
$\frakg_{R_S}^\vee\subset\bfY_{R_S}^\vee$  to be
\begin{align}\label{defg}
\begin{split}
\frakg_{R_S}&=E_{1,R_S}(M_{R_S})
,\\
\frakg_{R_S}^\vee&=E_{1,R_S}^\vee(M_{R_S}^\vee).
\end{split}
\end{align}
Recall that the $R_S$-modules
$F_{\bfw,R_S}$ and $F_{\bfw,R_S}^\vee$ have cohomological $\bbZ$-gradings defined in \eqref{cohdeg}.

\begin{lemma}\label{lem:gg1} 
\hfill
\begin{enumerate}[label=$\mathrm{(\alph*)}$,leftmargin=8mm,itemsep=1.2mm]
\item
$\frakg_{R_S}$, $\frakg_{R_S}^\vee$ are $\bbZ I\times\bbZ$-graded
Lie subalgebras of $\bfY_{R_S}$, $\bfY_{R_S}^\vee$ 
for the standard Lie bracket.
\item 
$\frakg_R\otimes_R R_S=\frakg_{R_S}$
and
$\frakg_R^\vee\otimes_R R_S=\frakg_{R_S}^\vee$
as $\bbZ I\times\bbZ$-graded $R_S$-Lie algebras.
\item
$\frakg_{K_S}=\frakg_{K_S}^\vee$
as $\bbZ I$-graded Lie $K_S$-algebras.
\end{enumerate}
\end{lemma}

\begin{proof}
There obvious $R_S$--linear isomorphisms 
$$\frakg_R\otimes_R R_S=\frakg_{R_S}
,\quad
\frakg_R^\vee\otimes_R R_S=\frakg_{R_S}^\vee.$$
There are also obvious $K_S$-linear isomorphisms 
$$\frakg_{K_S}=\frakg_{R_S}\otimes_{R_S}K_S
=E_{1,K_S}(M_{K_S})=E_{1,K_S}(M_{K_S}^\vee)=
\frakg_{R_S}^\vee\otimes_{R_S}K_S=\frakg_{K_S}^\vee.$$
Part (a) for $\frakg_{R_S}$ follows from the following formula proved in \cite[(5.12)]{MO19}
\begin{align}\label{512}
[E_{1,R_S}(m_1u^{l_1}),E_{1,R_S}(m_2u^{l_2})]=
E_{1,R_S}((\Tr\otimes 1)([\bfr,m_1\otimes m_2])u^{l_1+l_2})+lower
%,\quad m_1,m_2\in M_{R_S}.
\end{align}
Here we abbreviated $\bfr=\sum_{i,j\in I}\bfr_{\delta_i,\delta_j}$, 
and \emph{lower} means a term of lower degree for the $u$-filtration.
Since 
$\frakg_{R_S}^\vee\subset\frakg_{K_S}^\vee$ and
$\frakg_{K_S}^\vee=\frakg_{K_S}$,
from \eqref{512} we deduce that
\begin{align}\label{512vee}
[E_{1,R_S}^\vee(m_1^\vee u^{l_1}),E_{1,R_S}^\vee(m_2^\vee u^{l_2})]=
E_{1,R_S}^\vee\Big((\Tr\otimes 1)[\bfr,m_1^\vee\otimes m_2^\vee]u^{l_1+l_2}\Big)+lower
%,\quad m_1,m_2\in M_{R_S}^\vee,
\end{align}
The Lie algebras $\frakg_{R_S}$, $\frakg_{R_S}^\vee$ have $\bbZ I\times\bbZ$-gradings 
induced by the weight and the cohomological grading of $\bfY_{R_S}$, $\bfY_{R_S}^\vee$.
%They also have $\bbZ$-gradings induced by the cohomological gradings of 
%$F_{\bfw,R_S}$, $F_{\bfw,R_S}^\vee$, because the operators $E_{\bfw,R_S}$
%and $E_{\bfw,R_S}^\vee$ in \eqref{e}, \eqref{evee} are homogeneous of degree 0 for the cohomological grading.
\end{proof}

We consider the decomposition into weight subspaces 
$$\frakg_{R_S}=\bigoplus_{v\in\bbZ I}\frakg_{v,R_S}
,\quad
\frakg_{R_S}^\vee=\bigoplus_{v\in\bbZ I}\frakg_{v,R_S}^\vee.$$
A nonzero element $v\in\bbN I$ is called a root of $\frakg_{R_S}$ or $\frakg_{R_S}^\vee$
if $\frakg_{v,R_S}\neq\{0\}$ or $\frakg_{v,R_S}^\vee\neq\{0\}$.
Let $\frakg_{R_S}[u]$ be the current algebra of $\frakg_{R_S}$.
Let $\bfU_{R_S}$ be the enveloping algebra
$$\bfU_{R_S}=U(\frakg_{R_S}[u]).$$
We define $\bfU_{R_S}^\vee=U(\frakg^\vee_{R_S}[u])$ similarly.
By \cite[\S 5.3]{MO19}, there is a triangular decomposition
$$\frakg_{R_S}=\frakg_{R_S}^-\oplus\frakh_{R_S}\oplus\frakg_{R_S}^+$$
with
$$
\frakg_{R_S}^\pm=\bigoplus_{\substack{v\in\bbN I\\v\neq0}}\frakg_{\pm v,R_S}
,\quad
\frakh_{R_S}=\frakg_{0,R_S}
.$$
The multiplication yields an isomorphism
\begin{align}\label{Upm0}
\bfU_{R_S}^-\otimes_{R_S}\bfU_{R_S}^0\otimes_{R_S}\bfU_{R_S}^+=\bfU_{R_S}
\end{align}
with
$\bfU_{R_S}^\pm=U(\frakg_{R_S}^\pm[u])$
and
$\bfU_{R_S}^0=U(\frakh_{R_S}[u]).$
Similarly, we have a triangular decomposition 
$$\frakg_{R_S}^\vee=\frakg_{R_S}^{-,\vee}\oplus\frakh_{R_S}^\vee\oplus\frakg_{R_S}^{+,\vee}
$$
with
$\frakg_{R_S}^{\pm,\vee}=\bigoplus_{v\in\bbN I,\,v\neq0}\frakg_{\pm v,R_S}^\vee$
and
$\frakh_{R_S}^\vee=\frakg_{R_S,0}^\vee$,
yielding the following triangular decomposition 
\begin{align}\label{Upm0vee}
\bfU_{R_S}^{-,\vee}\otimes_{R_S}\bfU_{R_S}^{0,\vee}\otimes_{R_S}\bfU_{R_S}^{+,\vee}=
\bfU_{R_S}^\vee
\end{align}
with
$\bfU_{R_S}^{\pm,\vee}=U(\frakg_{R_S}^{\pm,\vee}[u])$
and
$\bfU_{R_S}^{0,\vee}=U(\frakh_{R_S}^\vee[u]).$
For a future, we consider the decomposition into weight subspaces 
\begin{align}\label{weightU}
\bfU_{R_S}^\flat=\bigoplus_{v\in\bbZ I}\bfU_{v,R_S}^\flat
,\quad
\bfU_{R_S}^{\flat,\vee}=\bigoplus_{v\in\bbZ I}\bfU_{v,R_S}^{\flat,\vee}
,\quad
\flat=\pm,0,\emptyset.
\end{align}

\begin{lemma}\label{lem:pairingg} 
Let $v$ be a root.
\hfill
\begin{enumerate}[label=$\mathrm{(\alph*)}$,leftmargin=8mm,itemsep=1.2mm]
\item
There is a coroot $h_v\in\frakh_{R_S}$ and a perfect pairing of $R_S$-modules
\begin{align}\label{pairing2}\langle-,-\rangle:\frakg_{v,R_S}\times\frakg_{-v,R_S}\to R_S
\end{align}
such that 
$[x,y]=\langle x,y\rangle\,h_v$ for each
$x\in\frakg_{v,R_S}$ and $y\in\frakg_{-v,R_S}.$
\item
There is a coroot $h_v^\vee\in\frakh_{R_S}^\vee$ and a perfect pairing of $R_S$-modules
\begin{align}\label{pairing2vee}\langle-,-\rangle:\frakg_{v,R_S}^\vee\times\frakg_{-v,R_S}^\vee\to R_S
\end{align}
such that 
$[x,y]=\langle x,y\rangle\,h_v^\vee$
for each
$x\in\frakg_{v,R_S}^\vee$ and 
$y\in\frakg_{-v,R_S}^\vee.$
\item
The root subspaces of $\frakg_{R_S}$ and $\frakg_{R_S}^\vee$
are free $R_S$-modules of finite rank.
\item
An element $v\in\bbZ I$ is a root for $\frakg_{R_S}$ if and only if it is root for $\frakg_{R_S}^\vee$.
\end{enumerate}
\end{lemma}

\begin{proof}
Part (a) is proved in \cite[\S 5.3]{MO19}.
The proof of (b) is similar. 
More precisely, for each $v\in\bbZ I$ we
fix some elements $x_{v,k}$, $y_{v,k}$ in $\frakg_{v,R_S}^\vee$  with $k=1,2,\dots$ such that
$$\bfr_{\bfw',\bfw}^\vee=
\sum_{v\in\bbN^I}\sum_{k}\rho_{\bfw'}^\vee(x_{v,k})\otimes\rho_\bfw^\vee(y_{-v,k})
,\quad
\bfw,\bfw'\in\bfN_1.$$
Then, we define
\begin{align}\label{coroot}h_v^\vee=\sum_kv(x_{0,k})\,y_{0,k}.\end{align}
From  \cite[(5.19)]{MO19},  Lemma \ref{lem:sym} and the inclusion 
$$\frakg_{R_S}^\vee\subset\frakg_{K_S}^\vee=\frakg_{K_S}$$
we deduce that
\begin{align}\label{form1}
\xi\otimes h_v=\sum_kx_{v,k}\otimes[\xi,y_{-v,k}]
,\quad
\xi\in\frakg_{v,R_S}^\vee.
\end{align}
Part (b) follows from \eqref{form1} as in \cite[thm 5.3.11]{MO19}.
The proof of (c) for $\frakg_{R_S}$ is done in \cite[thm 5.3.11]{MO19}.
The case of $\frakg_{R_S}^\vee$ is similar.
Part (d) follows from (c) and the isomorphism $\frakg_{K_S}^\vee=\frakg_{K_S}$.
\end{proof}

\subsubsection{The duality}\label{sec:duality}
Applying the transpose in \eqref{T2} to \eqref{M} and \eqref{Mvee}, we get the
$R_S$-module isomorphism
$(-)^\T:M_{R_S}\to M_{R_S}^\vee.$

\begin{lemma}\label{lem:FF} Let $\bfw\in \bfN_0$, $m\in M_{R_S}$ and $l\in\bbN$.
\hfill
\begin{enumerate}[label=$\mathrm{(\alph*)}$,leftmargin=8mm,itemsep=1.2mm]
\item 
$E_{\bfw,R_S}(m u^l)^\T=E_{\bfw,R_S}^\vee(m^\T u^l)$.
\item 
The transpose yields an $R_S$-algebra anti-isomorphism 
$(-)^\T:\bfY_{R_S}\to\bfY_{R_S}^\vee$.
\item
The transpose yields a $R_S$-Lie algebra anti-isomorphism
$(-)^\T:\frakg_{R_S}\to\frakg_{R_S}^\vee$.
\item
The transpose intertwines $\rho_\bfw^\vee$ with the right representation $(\rho_\bfw)^\T$ of $\bfY_{R_S}$.
\end{enumerate}
\end{lemma}

\begin{proof}
Part (a) follows from Lemma \ref{lem:rhovee} and the formulas \eqref{E}, \eqref{Evee}.
Parts (b), (c), (d) follow from (a) and the definitions of 
$\bfY_{R_S}$, $\bfY_{R_S}^\vee$,
$\frakg_{R_S}$ and $\frakg_{R_S}^\vee$.
\end{proof}

The transpose $A_{\bfw,R_S}\to A_{\bfw,R_S}^\vee$ in \eqref{T2} restricts to an $R_S$-linear operator
$$F_{v_1,\bfw,R_S}^\vee\otimes_{R_S} F_{v_2,\bfw,R_S}\to
F_{v_2,\bfw,R_S}\otimes_{R_S}
F_{v_1,\bfw,R_S}^\vee$$
which takes an element of degree $(v_2-v_1,l)$ to an element of degree 
$(v_1-v_2,-l)$ for each $l\in\bbZ$.
Let $(-)^*$ denote the dual of graded $R_S$-modules.

\begin{proposition}\label{prop:grading}
For each $v\in\bbN$ we have $\bbZ$-graded $R_S$-module isomorphisms
$$\frakg_{v,R_S}=(\frakg_{-v,R_S})^*=(\frakg_{v,R_S}^\vee)^*=\frakg_{-v,R_S}^\vee.$$
\end{proposition}

\begin{proof}
The proposition follows from Lemmas \ref{lem:pairingg} and  \ref{lem:FF}.
\end{proof}

\subsubsection{The $u$-filtrations}\label{sec:filt}
By \eqref{Fw} and \eqref{Fwvee}, 
the $R_{\bfw,S}$-modules 
$$F_{\bfw,R_{\bfw,S}}
,\quad
F_{\bfw,R_{\bfw,S}}^\vee
,\quad
\bfw=(i_1,\dots,i_s)\in\bfN_1$$ 
admit the
increasing $\bbN$-filtration whose $l$th terms are the $\bbZ I$-graded $R_S$-submodules
$$F_{\bfw,R_{\bfw,S}}[\leqsl l]
,\quad
F_{\bfw,R_{\bfw,S}}^\vee[\leqsl l]$$ consisting of all 
polynomials in $u_1,\dots, u_s$ of degree $\leqsl l$
with coefficients in $F_{\bfw,R_S}$ and $F_{\bfw,R_S}^\vee$
respectively.
Let $\gr(F_{\bfw,R_{\bfw,S}})$ and $\gr(F_{\bfw,R_{\bfw,S}}^\vee)$ 
be the associated graded. We'll identify 
$$\gr(F_{\bfw,R_{\bfw,S}})=F_{\bfw,R_{\bfw,S}}
,\quad
\gr(F_{\bfw,R_{\bfw,S}}^\vee)=F_{\bfw,R_{\bfw,S}}^\vee.$$ 
The isomorphisms \eqref{isom1}, \eqref{isom2} commute with the embeddings
$\iota$ in \eqref{iota1}, \eqref{iota2}, making the following square to commute
\begin{align*}
\begin{split}
\xymatrix{
F_{\delta_i,R_{\delta_i,S}}\ar@{=}[r]&F_{\delta_i,R_{\delta_i,S}}[u]\\
F_{\delta_i,R_{\delta_i,S}}^\vee\ar@{=}[r]\ar[u]^-\iota&F_{\delta_i,R_S}^\vee[u]\ar[u]_-{\iota\otimes 1}
}
\end{split}
\end{align*}
Hence, the map $\iota:F_{\bfw,R_{\bfw,S}}^\vee\to F_{\bfw,R_{\bfw,S}}$ in \eqref{iota1} is a strict morphism of filtered $R_S$-modules.
Taking endomorphisms and products, we get $u$-filtrations on 
$$A_{\bfw,R_{\bfw,S}}
,\quad
F_{1,R_S}
,\quad
A_{1,R_S}
,\quad
\text{etc}.$$
We equip $U(M_{R_S}[u])$ and $U(M_{R_S}^\vee[u])$ with the $\bbN$-filtrations 
such that 
$$U(M_{R_S}[u])[\leqsl l]
,\quad
U(M_{R_S}^\vee[u])[\leqsl l]$$
are the $R_S$-submodules consisting of all
elements of degree $\leqsl l$ in $u$.
We equip $\bfY_{R_S}$ and $\bfY_{R_S}^\vee$ with the quotient filtrations. 
The $R_S$-algebras $\bfU_{R_S}$ and $\bfU_{R_S}^\vee$
are $\bbZ$-graded by the $u$-degree.
Let $\bfU_{R_S}[\leqsl l]$ and $\bfU_{R_S}^\vee[\leqsl l]$ be the corresponding $\bbN$-filtrations.
Let $\sigma_l(x)$ denote the degree $l$ component 
in the associated graded of an element
$x$ of degree $\leqslant l$ for the $u$-filtration.
By \eqref{R2}, \eqref{Rii} and \eqref{Rvee}, \eqref{Riivee} the maps 
$\rho_\bfw$, $\rho_1$ and $\rho_\bfw^\vee$, $\rho_1^\vee$ 
in \eqref{rhorho}, \eqref{rhorhovee} are morphisms of filtered $R_S$-algebras.

%We'll call these $\bbN$-filtrations the $u$-filtrations of 
%$F_{\bfw,R_\bfw}$, $A_{\bfw,R_\bfw}$, $F_{1,R}$, $A_{1,R}$, $\bfY_R$ and $\bfU_R$.

\begin{lemma}\label{lem:ufilt}
\hfill
\begin{enumerate}[label=$\mathrm{(\alph*)}$,leftmargin=8mm,itemsep=1.2mm]
\item
The multiplications of $\bfY_{R_S}$ and $\bfY_{R_S}^\vee$
preserve the $\bbZ I\times\bbZ$-grading and the $u$-filtration.
The associated graded $\gr(\bfY_{R_S})$ and
$\gr(\bfY_{R_S}^\vee)$ are isomorphic to the $\bbZ I\times\bbZ$-graded
$R_S$-algebras $\bfU_{R_S}$ and
$\bfU_{R_S}^\vee$.
\item
The maps $\rho_\bfw$, $\rho_1$, $\rho_\bfw^\vee$, $\rho_1^\vee$
are morphisms of $\bbZ I\times\bbZ$-graded filtered $R_S$-algebras.
We have
$$\rho_\bfw(\bfY_{R_S}[\leqsl l])\subset A_{\bfw,R_{\bfw,S}}[\leqsl l]
,\quad
\rho_\bfw(\bfY_{R_S}^\vee[\leqsl l])\subset A_{\bfw,R_{\bfw,S}}^\vee[\leqsl l]
,\quad
\bfw\in\bfN_1,\,l\in\bbN.$$
\end{enumerate}
\end{lemma}

\begin{proof}
The case of $\bfY_{R_S}$ is done in \cite[thm.~5.5.1]{MO19}.
Hence, it is enough to check that $\gr(\bfY_{R_S}^\vee)=\bfU_{R_S}^\vee$.
From \cite[prop.~5.5.2]{MO19} we deduce that 
\begin{align}\label{552}
E_{1,R_S}(m)=0\Rightarrow E_{1,R_S}(mu^l)\in E_{1,R_S}(U(M_{R_S}[u])[<\!l])
,\quad
m\in M_{R_S}.
\end{align}
The base change formulas \eqref{URK} imply that we also have
\begin{align}\label{552vee}
E_{1,R_S}^\vee(m)=0\Rightarrow E_{1,R_S}^\vee(mu^l)\in E_{1,R_S}^\vee(U(M_{R_S}^\vee[u])[<\!l])
,\quad
m\in M_{R_S}^\vee.
\end{align}
From \eqref{512vee} and \eqref{552vee} we get as in \cite[\S 5.5.3]{MO19} 
a surjective $R_S$-algebra homomorphism 
$$\bfU_{R_S}^\vee\to\gr(\bfY_{R_S}^\vee).$$
To check it is invertible, it is enough to observe that
$\gr(\bfY_{R_S})=\bfU_{R_S}$ and to use the base change formulas
$$\bfY_{R_S}^\vee\subset\bfY_{K_S}^\vee=\bfY_{K_S}
,\quad
\bfU_{R_S}^\vee\subset\bfU_{K_S}^\vee=\bfU_{K_S}.$$
\end{proof}

Taking the associated graded, the maps $\rho_\bfw$ and $\rho_\bfw^\vee$
yield $R_S$-algebra homomorphisms 
$$\gr\rho_\bfw:\bfU_{R_S}\to A_{\bfw,R_{\bfw,S}}
,\quad
\gr\rho_\bfw^\vee:\bfU_{R_S}^\vee\to A_{\bfw,R_{\bfw,S}}^\vee$$ 
Hence, we have representations of $\bfU_{R_S}$ and $\bfU_{R_S}^\vee$
on $F_{\bfw,R_{\bfw,S}}$ and $F_{\bfw,R_{\bfw,S}}^\vee$ respectively.
The maps 
$$\rho_1:\bfY_{R_S}\to A_{1,R_S}
,\quad
\rho_1^\vee:\bfY_{R_S}^\vee\to A_{1,R_S}^\vee$$
are embedding of filtered $R_S$-algebras.
The associated graded are $R_S$-algebra homomorphisms
\begin{align}\label{grrho1}
\gr\rho_1:\bfU_{R_S}\to A_{1,R_S}
,\quad
\gr\rho_1^\vee:\bfU_{R_S}^\vee\to A_{1,R_S}^\vee.
\end{align}
The following lemma is proved in the next section.

\begin{lemma}\label{lem:rhostrict}
\hfill
\begin{enumerate}[label=$\mathrm{(\alph*)}$,leftmargin=8mm,itemsep=1.2mm]
\item
$\rho_1$ and $\rho_1^\vee$ are strict embeddings of $\bbZ I\times\bbZ$-graded
filtered $R_S$-algebras, i.e., we have
\begin{align*}
\rho_1(\bfY_{R_S}[\leqsl l])=\rho_1(\bfY_{R_S})\cap A_{1,R_S}[\leqsl l]
,\quad
\rho_1^\vee(\bfY_{R_S}^\vee[\leqsl l])=\rho_1^\vee(\bfY_{R_S}^\vee)\cap A_{1,R_S}^\vee[\leqsl l].
\end{align*}
\item
$\gr\rho_1:\bfU_{R_S}\to A_{1,R_S}$ and $\gr\rho_1^\vee:\bfU_{R_S}^\vee\to A_{1,R_S}^\vee$
are embedding of $\bbZ I\times\bbZ$-graded $R_S$-algebras.
\qed
\end{enumerate}
\end{lemma}

\subsubsection{The comultiplication}
Let $\bfw=\bfw_1\bfw_2$ denote the tuple
given by glueing the tuple $\bfw_1$ with $\bfw_2$.
The Yangian is equipped with the topological comultiplication 
$$\Delta:\bfY_{R_S}\to\bfY_{R_S}\hat\otimes_{R_S}\bfY_{R_S}$$ defined in \cite[\S 5.2]{MO19}. 
By definition of $\Delta$,  we have
\begin{align}\label{Delta}(\rho_{\bfw_1}\otimes\rho_{\bfw_2})\circ\Delta=\rho_\bfw.\end{align}
We define similarly a topological comultiplication 
$$\Delta:\bfY_{R_S}^\vee\to\bfY_{R_S}^\vee\hat\otimes_{R_S}\bfY_{R_S}^\vee$$ 
such that
\begin{align*}(\rho_{\bfw_1}^\vee\otimes\rho_{\bfw_2}^\vee)\circ\Delta=\rho_\bfw^\vee
\end{align*}
We equip the enveloping algebras $\bfU_{R_S}$ and $\bfU_{R_S}^\vee$ with the trivial coproduct.

\begin{lemma}\label{lem:Delta}
\hfill
\begin{enumerate}[label=$\mathrm{(\alph*)}$,leftmargin=8mm,itemsep=1.2mm]
\item
If $w=|\bfw|$ and $\bfw=(w_1,w_2)$ then the following diagram commutes
\begin{align}\label{stab}
\begin{split}
\xymatrix{
F_{\bfw,R_{\bfw,S}}\ar[rr]^-{\stab(1)}&&F_{w,R_{\bfw,S}}&&\ar[ll]_-{\stab(\omega_0)}
F_{\bfw,R_{\bfw,S}}\\
F_{\bfw,R_{\bfw,S}}\ar[rr]^-{\stab(1)}\ar[u]^-{\Delta(x)}&&F_{w,R_{\bfw,S}}\ar[u]_-x&&\ar[ll]_-{\stab(\omega_0)}F_{\bfw,R_{\bfw,S}}\ar[u]_-{\Delta^\op(x)}}
\end{split}
\end{align}
\item
$\Delta$ preserves the $\bbZ I\times\bbZ$-grading and the $u$-filtration. 
\item
$\gr(\bfY_{R_S})=\bfU_{R_S}$ as $\bbZ I\times\bbZ$-graded bialgebras. 
\end{enumerate}
\end{lemma}

\begin{proof}
Part (a) follows from \cite[prop.~4.2.1]{MO19}.
We have $\gr(\bfY_{R_S})=\bfU_{R_S}$ as $\bbZ I\times\bbZ$-graded algebras
by \cite[thm.~5.5.1]{MO19}.
Let us check Parts (b) and (c).
Lemma \ref{lem:rho} and \eqref{E} imply that
for each tuple $\bfw=(i_1,i_2,\dots,i_s)$  in $\bfN_1$ we have
\begin{align}\label{Emul}
E_{\bfw,R_S}(mu^l)=\sum_{r=1}^sE_{\delta_{i_r}}\!(m)u_r^l\ \mod\ A_{\bfw,R_{\bfw,S}}[<\!l]
,\quad
m\in M_{R_S}
,\,
l\in\bbN.
\end{align}
Let $x\in\bfY_{R_S}[\leqsl l]$ of the form $x=E_{0,R_S}(mu^l)$.
By \eqref{Delta} and \eqref{Emul} we have
$$(\otimes_{r=1}^s\rho_{\delta_{i_r}})\Delta^{\otimes(s-1)}(x)
=(\otimes_{r=1}^s\rho_{\delta_{i_r}})\Delta^{\otimes(s-1)}(E_{0,R_S}(m)u^l)\ \mod\ A_{\bfw,R_{\bfw,S}}[<\!l]$$
On the right hand side $\Delta$ is the trivial comultiplication.
Note that 
$$\sigma_l(x)=E_{0,R_S}(m)u^l\in\frakg_{R_S}[u].$$
Taking products of such elements, for all $x\in\bfY_{R_S}[\leqsl l]$ and all
$\bfw_1,\bfw_2\in \bfN_1$ we get 
\begin{align*}
(\rho_{\bfw_1}\otimes\rho_{\bfw_2})\Delta(x)&\in A_{\bfw_1\bfw_2,R_{\bfw_1\bfw_2,S}}[\leqsl l],\\
\sigma_l(\rho_{\bfw_1}\otimes\rho_{\bfw_2})\Delta(x)&=
(\gr\rho_{\bfw_1}\otimes\gr\rho_{\bfw_2})\Delta(\sigma_l(x))
\end{align*}
\end{proof}

\begin{lemma}\label{lem:Deltavee}
\hfill
\begin{enumerate}[label=$\mathrm{(\alph*)}$,leftmargin=8mm,itemsep=1.2mm]
\item
If $w=|\bfw|$ and $\bfw=(w_1,w_2)$ then the following diagram commutes
\begin{align}\label{stabvee}
\begin{split}
\xymatrix{
F_{\bfw,R_{\bfw,S}}^\vee&&\ar[ll]_-{\stab^\vee(1)}
F_{w,R_{\bfw,S}}^\vee\ar[rr]^-{\stab^\vee(\omega_0)}&&
F_{\bfw,R_{\bfw,S}}^\vee\\
\ar[u]^-{\Delta(x)}F_{\bfw,R_{\bfw,S}}^\vee&&\ar[ll]_-{\stab^\vee(1)}
F_{w,R_{\bfw,S}}^\vee\ar[u]_-x\ar[rr]^-{\stab^\vee(\omega_0)}&&
F_{\bfw,R_{\bfw,S}}^\vee\ar[u]_-{\Delta^\op(x)}}
\end{split}
\end{align}
\item
$\Delta$ preserves the $\bbZ I\times\bbZ$-grading and the $u$-filtration. 
\item
$\gr(\bfY_{R_S}^\vee)=\bfU_{R_S}^\vee$ as $\bbZ I\times\bbZ$-graded bialgebras. 
\end{enumerate}
\end{lemma}

\begin{proof}
The proof is similar to the proof of Lemma \ref{lem:Delta}.
We are reduced to check that 
$$\gr(\bfY_{R_S}^\vee)=\bfU_{R_S}^\vee$$ as $\bbZ I\times\bbZ$-graded algebras.
This follows from Lemmas \ref{lem:FF} and \ref{lem:Delta}.
\end{proof}

\begin{proof}[Proof of Lemma $\ref{lem:rhostrict}$]
We only give the proof for $\rho_1$ and $\gr\rho_1$.
The proof for $\rho_1^\vee$ and $\gr\rho_1^\vee$ follows using Lemma \ref{lem:FF}.
Let us prove Part (a).
We claim that for each 
$x\in U(M_{R_S}[u])[\leqsl l]$ such that $E_{1,R_S}(x)\in A_{1,R_S}[<\!l]$, there is an element
$y\in U(M_{R_S}[u])[<\!l]$ such that $E_{1,R_S}(x)=E_{1,R_S}(y)$.
We may assume that $x=mu^l$ with $m\in M_{R_S}$ and $l\in\bbN$.
Since $E_{1,R_S}(mu^l)\in A_{1,R_S}[<\!l]$, we have 
$$E_{\bfw,R_S}(mu^l)\in A_{\bfw,R_{\bfw,S}}[<\!l]
,\quad
\bfw\in\bfN_1.$$
Comparing with the formula \eqref{Emul}, we deduce that $E_{\bfw,R_S}(m)=0$ for each $\bfw\in\bfN_1$.
Thus, we have $E_{1,R_S}(m)=0$.
From \eqref{552} we deduce that $$E_{1,R_S}(mu^l)\in E_{1,R_S}(U(M_{R_S}[u])[<\!l]),$$
proving the claim.
Part (b) follows from (a).
\end{proof}

\subsubsection{The Cartan loop algebra}\label{sec:Y0}
Let $\bfY_{R_S}^0$ be the $R$-subalgebra of $A_{2,R_S}$ generated by 
the Chern characters of the tautological bundles 
\begin{align}\label{Ch}\{\ch_l(\calV_i), \ch_l(\calW_i)\,;\,i\in I\,,\,l\in\bbN\}.\end{align} 
Here $\ch_0$ is the rank.
By Lemma \ref{lem:rho} the map $\rho_2$ embeds the Yangian $\bfY_{R_S}$ into $A_{2,R_S}$.
By \cite[prop.~5.5.3]{MO19} the $R_S$-algebra $\bfY_{R_S}^0$ is contained in $\rho_2(\bfY_{R_S})$.
Hence $\bfY_{R_S}^0$ can be identified with an $R_S$-subalgebra of $\bfY_{R_S}$.
Let $\bfY_{R_S}^{0,v}$ be the $R_S$-subalgebra generated by the subset
$$\{\ch_l(\calV_i)\,;\,i\in I\,,\,l\in\bbN\}.$$
Let $\bfY_{R_S}^{0,w}$ the $R_S$-subalgebra generated by the subset
$$\{\ch_l(\calW_i)\,;\,i\in I\,,\,l\in\bbN\}.$$
Let $\bfU_{R_S}^{0,v},$ $\bfU_{R_S}^{0,w}\subset\bfU_{R_S}$ be the graded $R_S$-subalgebras 
associated with
the filtrations on $\bfY_{R_S}^{0,v}$ and $\bfY_{R_S}^{0,w}$  induced by the $u$-filtration of $\bfY_{R_S}$.

\begin{lemma}\label{lem:Y01}
\hfill
\begin{enumerate}[label=$\mathrm{(\alph*)}$,leftmargin=8mm,itemsep=1.2mm]
\item 
$\bfU_{R_S}^0$ is the associated graded of 
$\bfY_{R_S}^0$ for the filtration induced by the $u$-filtration of $\bfY_{R_S}$.
\item
$\bfU_{R_S}^{0,v}=U(\frakh_{R_S}^v[u])$, $\bfU_{R_S}^{0,w}=U(\frakh_{R_S}^w[u])$
and $\bfU_{R_S}^0=U(\frakh_{R_S}[u])$
with $\frakh_{R_S}=\frakh_{R_S}^v\oplus\frakh_{R_S}^w$.
\item
$\bfY_{R_S}^{0,v}$, $\bfY_{R_S}^{0,w}$, $\bfY_{R_S}^0$ are isomorphic to 
$\bfU_{R_S}^{0,v},$ $\bfU_{R_S}^{0,w}, $ $\bfU_{R_S}^0$ respectively as $R_S$-algebras.
\item
The multiplication yields an isomorphism
$\bfY_{R_S}^{0,v}\otimes_{R_S}\bfY_{R_S}^{0,w}\to\bfY_{R_S}^0$.
\item
$\bfY_{R_S}^{0,w}$ is a central $R_S$-subalgebra of $\bfY_{R_S}$, and $\frakh_{R_S}^w$ 
is a central Lie $R_S$-subalgebra of $\frakg_{R_S}$.
\end{enumerate}
\end{lemma}

The transpose $\bfU_{R_S}\to\bfU_{R_S}^\vee$ and $\bfY_{R_S}\to\bfY_{R_S}^\vee$
in Lemma \ref{lem:FF} is compatible with the triangular 
decompositions \eqref{Upm0} and \eqref{Upm0vee}.
We define the $R_S$-subalgebras
$$\bfU_{R_S}^{0,v,\vee},\bfU_{R_S}^{0,w,\vee}\subset\bfU_{R_S}^{0,\vee}
,\quad
\bfY_{R_S}^{0,v,\vee},\bfY_{R_S}^{0,w,\vee}\subset\bfY_{R_S}^{0,\vee}$$
by \emph{transport de structure} via the transpose map.
We define similarly the Lie $R_S$-subalgebras 
$$\frakh_{R_S}^{v,\vee},\frakh_{R_S}^{w,\vee}\subset\frakh_{R_S}^\vee.$$
For $x\in\bfU_{R_S}$ or $\bfU_{R_S}^\vee$ we abbreviate
$$x(\phi_\bfw)=\gr\rho_\bfw(x)(\phi_\bfw)
,\quad
x(\phi_\bfw^\vee)=\gr\rho_\bfw^\vee(x)(\phi_\bfw^\vee)
,\quad
\bfw\in\bfN_1.$$

\begin{lemma}\label{lem:Y02} There are $R_S$-bases $(a_i)_{i\in I}$ and $(b_i)_{i\in I}$
of $\frakh_{R_S}^v$ and $\frakh_{R_S}^w$ such that
\hfill
\begin{enumerate}[label=$\mathrm{(\alph*)}$,leftmargin=8mm,itemsep=1.2mm]
\item
$\sigma_l \ch_l(\calV_i)=a_iu^l$ and
$\sigma_l \ch_l(\calW_i)=b_iu^l$ for each 
$i\in I,\,l\in\bbN,$
\item 
$a_iu^l(\phi_\bfw)=0$ and
$b_iu^l(\phi_\bfw)=\ch_l(\calW_i)\cap\phi_\bfw$
for each $\bfw\in \bfN_1$,
\item
$h_{\delta_i}=b_i-\sum_{j\in I}\bfc_{i,j}a_j$.
Further $h_{\delta_i}$ acts on $F_{v,w}$ 
by multiplication by $w_i-(\delta_i,v)_Q.$
\end{enumerate}
\end{lemma}

\begin{proof}[Proof of the lemmas]
The lemmas follow from \cite[\S\S 5.4-5]{MO19}.
More precisely, we abbreviate 
$$v_i=\ch_0(\calV_i)
,\quad
w_i=\ch_0(\calW_i),$$
and we view both
as operators acting on the $R_S$-module $F_{v,w,R_S}$ by multiplication by the integers
$v_i$ and $w_i$ respectively.
We can also view them as elements of the $R_S$-algebra $M_{R_S}$, 
hence, applying the map $E_{1,R_S}$, as elements of the Lie algebra
$\frakg_{R_S}$ which belong to the Lie subalgebra $\frakh_{R_S}$.
Recall that the map $E_{1,R_S}$ is the product of all maps 
$E_{\bfw,R_S}$ as the tuple $\bfw$ runs over $\bfN_1$.
By \cite[p.~90]{MO19}, the formula in \cite[thm.~4.9.1]{MO19} for the 
diagonal matrix coefficients of the R-matrix imply that for each $m\in M_{R_S}$ and each $l\in\bbN$ we have
$$\frac{1}{l!}E_{\bfw,R_S}(mu^l)=\sum_{i\in I}\Tr(m v_i)\ch_l(\calW_i)+\sum_{i\in I}\Tr(m w_i)\ch_l(\calV_i)-
\sum_{i,j\in I}\bfc_{i,j}\Tr(m v_i)\ch_l(\calV_j)+lower$$
where $\Tr$ is the trace $M_{R_S}\to R_S$ and \emph{lower} means lower for the $u$-filtration of
$A_{\bfw,R_{\bfw,S}}$. 
Let 
$$g_i=1-\bfc_{i,i}/2$$ be the number of loops $i\to i$ in $Q_1$.
We have 
\begin{align}\label{Mdelta}
\frakM(\delta_j,\delta_i)=\emptyset
,\quad
\frakM(\delta_i,\delta_i)=\bbC^{2g_i}
,\quad
\frakM(0,\delta_i)=\Spec(\bbC)
,\quad
i\neq j.
\end{align}
Let $\langle -\rangle$ be the grading shift functor.
From \eqref{Mdelta} we deduce that
\begin{align}\label{Fdelta}
F_{\delta_j,\delta_i,R_{\delta_i}}=\{0\}
,\quad
F_{\delta_i,\delta_i,R_{\delta_i}}=R_{\delta_i,S}\langle g_i\rangle
,\quad
F_{0,\delta_i,R_{\delta_i,S}}=R_{\delta_i,S}
,\quad
R_{\delta_i,S}=R_S[u]
,\quad
i\neq j.
\end{align}
Let $P_i,Q_i\in M_{R_S}$ such that
$$P_i=\id_{F_{0,\delta_i,R_S}}
,\quad 
Q_i=\id_{F_{\delta_i,\delta_i,R_S}}-\id_{F_{0,\delta_i,R_S}}.$$
Then, we get
$$E_{\bfw,R_S}(P_iu^l)=\ch_l(\calV_i)+lower
,\quad
E_{\bfw,R_S}(Q_iu^l)=\ch_l(\calW_i)-\sum_{j\in I}\bfc_{i,j}\ch_l(\calV_j)+lower$$
Thus, to prove Part (a) of Lemma \ref{lem:Y02} it is enough to set
$$a_i=E_{1,R_S}(P_i)
,\quad
b_i=E_{1,R_S}(Q_i)+\sum_{j\in I}\bfc_{i,j}E_{1,R_S}(P_j).$$
The part (b) follows from (a), and (c) from \cite[(5.18)]{MO19}.
Using this, the Lemma \ref{lem:Y01} is easy to check. For instance, the $R_S$-subalgebra
$\bfU_{R_S}^0\subset\bfU_{R_S}$ is contained in the associated graded  
$\gr(\bfY_{R_S}^0)\subset\gr(\bfY_{R_S})=\bfU_{R_S}$ of $\bfY_{R_S}^0$ for the filtration 
induced by the $u$-filtration of $\bfY_{R_S}$.
Since $\bfY_{R_S}^0$ is commutative and is generated be the set \eqref{Ch}, 
its graded dimension is bounded above by the graded dimension of $\bfU_{R_S}^0$.
Hence $\gr(\bfY_{R_S}^0)=\bfU_{R_S}^0$.
\end{proof}

Let $a^\vee_i$, $b^\vee_i$ be the images of $a_i$, $b_i$ 
in $\frakh_{R_S}^{\vee,v}$, $\frakh_{R_S}^{\vee,w}$
by the transpose in Lemma  \ref{lem:FF}.
The transpose by the pairing \eqref{PP2} of the operator on $F_{w,R_{w,S}}$ given by the cap product
by the characteristic classes $\ch_l(\calV_i)$ or $\ch_l(\calW_i)$  is the 
the operator on $F_{w,R_{w,S}}^\vee$ given by the cap product
by the restrictions  $\iota^*\ch_l(\calV_i)$ or $\iota^*\ch_l(\calW_i)$  to the nilpotent subvariety.
Hence, from the previous lemma we deduce the following one.

\begin{lemma} \label{lem:Y03}
\hfill
\begin{enumerate}[label=$\mathrm{(\alph*)}$,leftmargin=8mm,itemsep=1.2mm]
\item
$\sigma_l\, \iota^*\ch_l(\calV_i)=a_i^\vee u^l$ and
$\sigma_l\, \iota^*\ch_l(\calW_i)=b_i^\vee u^l$ for each 
$i\in I,\,l\in\bbN,$
\item 
$a_i^\vee u^l(\phi_\bfw^\vee)=0$ and
$b_i^\vee u^l(\phi_\bfw)=\ch_l(\calW_i)\cap\phi_\bfw^\vee$
for each $\bfw\in \bfN_1$,
\item
$h_{\delta_i}^\vee=b_i^\vee-\sum_{j\in I}\bfc_{i,j}a_j^\vee$.
Further $h_{\delta_i}^\vee$ acts on $F_{v,w}^\vee$ 
by multiplication by $w_i-(\delta_i,v)_Q.$
\qed
\end{enumerate}
\end{lemma}

We have $h_v=h_v^\vee$, $a_i=a_i^\vee$ and $b_i=b_i^\vee$ under the identification
$$\frakg_{R_S},\frakg_{R_S}^\vee\subset\frakg_{K_S}=\frakg_{K_S}^\vee$$
 in Lemma \ref{lem:gg1}. From now on we'll omit the upperscript $(-)^\vee$ for these operators.
Note also that the $R_S$-algebras $\bfU_{R_S}^{0,w}$ and
$\bfU_{R_S}^{0,w,\vee}$ are central in  $\bfU_{R_S}$ and $\bfU_{R_S}^\vee$ respectivelly.
We'll identify both with the $R_S$-algebra
$R_{\infty,S}$ in \eqref{Rinfty} via the $R$-algebra isomorphisms
\begin{align}\label{UR}\bfU_{R_S}^{0,w}\,,\,\bfU_{R_S}^{0,w,\vee}\to R_{\infty,S}
,\quad
b_iu^l\mapsto p_{i,l}.
\end{align}

\medskip

\subsection{The evaluation map}\label{sec:ev}

The representations $\rho_\bfw$ and $\rho_\bfw^\vee$ yield the $R_S$-linear maps
\begin{align}\label{evw}
\ev_\bfw:\bfY_{R_S}\to F_{\bfw,R_{\bfw,S}}
,\,
x\mapsto x(\phi_\bfw)
,\quad
\ev_\bfw^\vee:\bfY_{R_S}^\vee\to F_{\bfw,R_{\bfw,S}}^\vee
,\,
x\mapsto x(\phi_\bfw^\vee).
\end{align}
For each $a=0,1,2$, taking the product over all $\bfw\in \bfN_a$ we get the evaluation maps
\begin{align}\label{evavee}
\ev_a:\bfY_{R_S}\to F_{a,R_S}
,\quad
\ev_a^\vee:\bfY_{R_S}^\vee\to F_{a,R_S}^\vee
\end{align}
Fix a tuple $\bfw\in\bfN_1$.
The evaluation map $\ev_\bfw$ is the composition of the filtered morphism $\rho_\bfw$ 
and the action
of the filtered algebra  $A_{\bfw,R_{\bfw,S}}$ on the vacuum vector of the
filtered module $F_{\bfw,R_{\bfw,S}}$.
Since the vacuum has the degree zero for the $u$-filtration, 
the map $\ev_\bfw$ preserves the $u$-filtrations of $\bfY_{R_S}$ and $F_{\bfw,R_{\bfw,S}}$.
The map $\ev_\bfw^\vee$ is also filtered for a similar reason.
Thus,
\begin{align}\label{evfilt}
\ev_\bfw(\bfY_{R_S}[\leqsl l])\subset F_{\bfw,R_{\bfw,S}}[\leqsl l]
,\quad
\ev_\bfw^\vee(\bfY_{R_S}^\vee[\leqsl l])\subset F_{\bfw,R_{\bfw,S}}^\vee[\leqsl l]
,\quad
l\in\bbN.
\end{align}
For a similar reason, the evaluation maps $\ev_\bfw$ and
$\ev_\bfw^\vee$ are homogeneous of degree 0 for the $\bbZ I\times\bbZ$-grading.
Taking the associated graded, we get the $\bbZ I\times\bbZ$-graded $R_S$-linear maps
$$\gr \ev_\bfw:\bfU_{R_S}\to F_{\bfw,R_{\bfw,S}}
,\,
x\mapsto x(\phi_\bfw)
,\quad
\gr \ev_\bfw^\vee:\bfU_{R_S}^\vee\to F_{\bfw,R_{\bfw,S}}^\vee
,\,
x\mapsto x(\phi_\bfw^\vee).$$
Taking the product over all tuples $\bfw$, we get the following maps
$$\gr \ev_1=\prod_{\bfw\in \bfN_1}\gr \ev_\bfw:\bfU_{R_S}\to F_{1,R_S}
,\quad
\gr \ev_1^\vee=\prod_{\bfw\in \bfN_1}\gr \ev_\bfw^\vee:\bfU_{R_S}^\vee\to F_{1,R_S}^\vee.$$

\begin{lemma}\label{lem:KK}
$\Ker(\ev_a)$ and $\Ker(\ev_a^\vee)$ are independent of the choice of $a=0,1,2$.
\end{lemma}

\begin{proof}
Let $\bfw\in \bfN_0$ and $w=|\bfw|$.
The diagram \eqref{stab} and Lemma \ref{lem:ss} yield
an injective morphism of $\bfY_R$-modules preserving the vacuum vectors
$F_{\bfw,R_{\bfw,S}}\to F_{w,R_{\bfw,S}}$.
Similarly, the diagram \eqref{stabvee}  yields
injective morphisms of $\bfY_R^\vee$-modules preserving the vacuum vectors
$F_{w,R_{\bfw,S}}^\vee\to F_{\bfw,R_{\bfw,S}}^\vee$.
Hence, we have 
\begin{align}\label{Ker3}
\Ker(\ev_\bfw)=\Ker(\ev_w)
,\quad
\Ker(\ev_\bfw^\vee)=\Ker(\ev_w^\vee).\end{align}
\end{proof}

For each
$\bfw\in \bfN_1$ we consider the ideal
\begin{align}\label{J}
J_w=\Ker(\gr\ev_\bfw)\cap \bfU_{R_S}^{0,w}
\end{align}
where $w=|\bfw|$.
By Lemma \ref{lem:Y02}, the isomorphism \eqref{UR}
identifies $J_w$ with the Kernel of the
$R_S$-algebra homomorphism 
$R_{\infty,S}\to R_{w,S}$
in \eqref{Rinftyw}.

\begin{lemma} \label{lem:kerU}
\hfill
\begin{enumerate}[label=$\mathrm{(\alph*)}$,leftmargin=8mm,itemsep=1.2mm]
\item
$\Ker(\gr \ev_1)=\bfU_{R_S}\frakh_{R_S}^v[u]\oplus\bfU_{R_S}\frakg_{R_S}^+[u]$.
\item
$\Ker(\gr \ev_1^\vee)=\bfU_{R_S}^\vee\frakh_{R_S}^{v,\vee}[u]\oplus\bfU_{R_S}^\vee\frakg_{R_S}^{+,\vee}[u]$.
\end{enumerate}
\end{lemma}

\begin{proof}
Let $\bfw\in \bfN_1$. Set $w=|\bfw|$.
We first prove that for each root $v$ with $w(h_v)\neq 0$ we have
\begin{align}\label{a}
\Ker(\gr \ev_\bfw)\cap(\frakg^-_{-v,R_S}[u]\cdot \bfU_{R_S}^{0,w})=\frakg_{-v,R_S}^-[u]\cdot J_\bfw.
\end{align}
The proof is similar to the proof of \cite[prop.~5.3.4]{MO19}.
Let 
$$x\in\Ker(\gr \ev_\bfw)\cap(\frakg_{-v,R_S}^-[u]\cdot \bfU_{R_S}^{0,w}).$$
Set $x=\sum_kx_kf_k$ with 
$x_k\in\frakg_{-v,R_S}^-[u]$ and $f_k\in \bfU_{R_S}^{0,w}$ for all $k$.
We'll assume that the $x_k$'s are linearly independent.
For each element $y\in\frakg_{v,R_S}^+$ we have
$$w(h_v)\sum_k (y,x_k)f_k(\phi_\bfw)=\sum_k (y,x_k)h_vf_k(\phi_\bfw)=\sum_k [x_k,y]f_k(\phi_\bfw)
=-yx(\phi_\bfw)=0.$$
Hence, we have
$$\sum_k (y,x_k)f_k(\phi_\bfw)=0.$$
Since the $x_k$'s are linearly independent, this implies that $f_k\in J_\bfw$, for each $k$.
We deduce that
$$\Ker(\gr \ev_\bfw)\cap(\frakg_{-v,R_S}^-[u]\cdot \bfU_{R_S}^{0,w})\subset\frakg_{-v,R_S}^-[u]\cdot J_\bfw.$$
The reverse inclusion is obvious, proving \eqref{a}.
The $R_S$-module 
$$\Ker(\gr \ev_1)\subset \bfU_{R_S}$$ is $\bbZ I$-graded. Hence, we have
$$\Ker(\gr \ev_1)\cap\frakg_{R_S}[u]=(\Ker(\gr \ev_1)\cap\frakg_{R_S}^-[u])\oplus
(\Ker(\gr \ev_1)\cap\frakh_{R_S}[u])\oplus(\Ker(\gr \ev_1)\cap\frakg_{R_S}^+[u]).$$
By \eqref{a} and Lemma \ref{lem:Y02} we have
\begin{align}\label{Y02}
\begin{split}
\Ker(\gr \ev_1)\cap\frakg_{R_S}^-[u]&=\{0\},\\
\Ker(\gr \ev_1)\cap\frakh_{R_S}[u]&=\frakh_{R_S}^v[u],\\
\Ker(\gr \ev_1)\cap\frakg_{R_S}^+[u]&=\frakg_{R_S}^+[u].
\end{split}
\end{align}
We deduce that
\begin{align}\label{b}
\Ker(\gr \ev_1)\cap\frakg_{R_S}[u]=\frakh_{R_S}^v[u]\oplus\frakg_{R_S}^+[u].
\end{align} 
Next, we have
$$\bfU_{R_S}^{0,v}\bfU_{R_S}^+=
\bfU_{R_S}^{0,v}\bfU_{R_S}^+\frakh_{R_S}^v[u]\oplus\bfU_{R_S}^{0,v}\bfU_{R_S}^+\frakg_{R_S}^+[u]
\oplus R_S$$
Hence, from \eqref{Y02} we deduce that
\begin{align}\label{c}
\Ker(\gr \ev_1)\cap \bfU_{R_S}^{0,v}\bfU_{R_S}^+=
\bfU_{R_S}^{0,v}\bfU_{R_S}^+\frakh_{R_S}^v[u]\oplus\bfU_{R_S}^{0,v}\bfU_{R_S}^+\frakg_{R_S}^+[u].
\end{align}
Next, we prove that
\begin{align}\label{d}
\Ker(\gr \ev_1)\cap \bfU_{R_S}^-\bfU_{R_S}^{0,w}=\{0\}.
\end{align}
To do that, we'll prove by induction on $v\in\bbN I$ that
\begin{align}\label{kerUa}
\Ker(\gr \ev_1)\cap \bfU_{-v,R_S}^-\bfU_{R_S}^{0,w}=\{0\}
\end{align}
If $v=\delta_i$ then we have
$$\bfU_{-v,R_S}^-\bfU_{R_S}^{0,w}=\frakg^-_{-v,R_S}\cdot \bfU_R^{0,w}.$$ 
See \eqref{weightU} for the notation. Hence, the claim follows from \eqref{a}. 
Assume that
\begin{align}\label{hyp}
\Ker(\gr \ev_1)\cap \bfU_{>-v,R_S}^-\bfU_{R_S}^{0,w}=\{0\}
\end{align}
where
$$\bfU_{>-v,R_S}^-=\bigoplus_{\substack{v'\in\bbN I\\v'<v}}\bfU_{-v',R_S}^-.$$
The coproduct of $\bfU_{R_S}$ is trivial.
Hence, for each $\bfw=(\bfw_1,\bfw_2)$ with $\bfw_1,\bfw_2\in \bfN_1$ we have
$$\gr \ev_\bfw=(\gr \ev_{\bfw_1}\otimes\gr \ev_{\bfw_2})\circ\Delta.$$
We deduce that
$$\Delta(\Ker(\gr \ev_1))\subset\Ker(\gr \ev_1^{\otimes 2})=
\Ker(\gr \ev_1)\otimes_{R_S} \bfU_{R_S}+\bfU_{R_S}\otimes_{R_S}\Ker(\gr \ev_1).$$
Hence $\Ker(\gr \ev_1)$ is a coideal of $\bfU_{R_S}$.
Thus \eqref{kerUa} implies that
\begin{align}\label{kerUb}
\Delta(\Ker(\gr \ev_1))\cap (\bfU_{>-v,R_S}^-\bfU_{R_S}^{0,w})^{\otimes 2}=\{0\}.
\end{align}
The weight decomposition of $\bfU_{R_S}^{\otimes 2}$ yields the following decomposition of the coproduct
$$\Delta=\sum_{\bfv_1,\bfv_2\in\bbN I}\Delta_{\bfv_1,\bfv_2}$$
We define $\Delta_{st}$ to be the sum of the maps
$\Delta_\bfv$
over all tuples $\bfv=(v_1,v_2)$ such that $v_1,v_2<v$.
Set $\delta=\Delta-\Delta_{st}$.
Fix any element 
$$x\in \Ker(\gr \ev_1)\cap \bfU_{-v,R_S}^-\bfU_{R_S}^{0,w}.$$
We have $\Delta_{st}(x)=0$ by \eqref{kerUb}.
We deduce that
$\Delta(x)=\delta(x)$.
This implies that 
$$x\in\frakg_{v,R_S}^-[u]\cdot \bfU_{R_S}^{0,w}.$$
Hence \eqref{kerUa} follows from \eqref{a}.
We can now prove Part (a) of the lemma.
It is enough to prove the first equality.
To do so, we are reduced to prove the inclusion
$$\Ker(\gr \ev_1)\subset \bfU_{R_S}\frakh_R^v[u]\oplus\bfU_{R_S}\frakg_{R_S}^+[u]$$
because the reverse inclusion is obvious.
Let $x\in\Ker(\gr \ev_1)$.
Let 
$$x=\sum_kx^-_kx^+_k
,\quad
x^-_k\in \bfU_{R_S}^-\bfU_{R_S}^{0,w}
,\quad
x^+_k\in \bfU_{R_S}^{0,v}\bfU_{R_S}^+.$$
We may assume that 
$$x=x_1+x_2
,\quad 
x_c=\sum_kx^-_kx^+_{c,k}
,\quad
c=1,2$$
where
$$
x^+_{1,k}\in R_S
,\quad
x^+_{2,k}\in \bfU_{R_S}^{0,v}\bfU_{R_S}^+\frakh_{R_S}^v[u]\oplus\bfU_{R_S}^{0,v}\bfU_{R_S}^+\frakg_{R_S}^+[u]
.$$
Since $x,x_2\in\Ker(\gr \ev_1)$ we deduce that $x_1\in\Ker(\gr \ev_1)$.
Hence $x_1=0$ by \eqref{d}, proving Part (a).
Part (b) is proved in a similar way.
\end{proof}

From the lemma we deduce the following.

\begin{corollary} \label{cor:kerU}
\hfill
\begin{enumerate}[label=$\mathrm{(\alph*)}$,leftmargin=8mm,itemsep=1.2mm]
\item
$\bfU_{R_S}/\Ker(\gr \ev_1)\simeq\bfU_{R_S}^-\bfU_{R_S}^{0,w}$.
\item
$\bfU_{R_S}^\vee/\Ker(\gr \ev_1^\vee)\simeq\bfU_{R_S}^{-,\vee}\bfU_{R_S}^{0,w,\vee}$.
\end{enumerate}
\end{corollary}

We equip $\Ker(\ev_1)$ with the filtration induced by the $u$-filtration on $\bfY_{R_S}$.
Let $\gr\Ker(\ev_1)$ be the associated graded

\begin{lemma}\label{lem:evstrict}
\hfill
\begin{enumerate}[label=$\mathrm{(\alph*)}$,leftmargin=8mm,itemsep=1.2mm]
\item 
$\ev_1$ is a strict morphism of $u$-filtered $R_S$-modules $\bfY_{R_S}\to F_{1,R_S}$.
\item
$\Ker(\gr \ev_1)=\gr\Ker(\ev_1)$.
\item 
$\ev_1^\vee$ is a strict morphism of $u$-filtered $R_S$-modules $\bfY_{R_S}^\vee\to F_{1,R_S}^\vee$.
\item
$\Ker(\gr \ev_1^\vee)=\gr\Ker(\ev_1^\vee)$.
\end{enumerate}
\end{lemma}

\begin{proof}
Fix any element $x\in\bfY_{R_S}[\leqsl l]$ such that $\ev_1(x)\in F_{1,R_S}[<\!l]$.
We must prove there is an element $y\in \bfY_{R_S}[<\!l]$ such that $\ev_1(x)=\ev_1(y)$.
Since $x\in\bfY_{R_S}[\leqsl l]$, by Lemma \ref{lem:ufilt} we have 
$$\rho_\bfw(x)\in A_{\bfw,R_{\bfw,S}}[\leqsl l]
,\quad
\bfw\in \bfN_1.$$ 
Hence $\sigma_l(x)$ and $\sigma_l\rho_\bfw(x)$ are well-defined 
elements in $\bfU_{R_S}[l]$ and $A_{\bfw,R_{\bfw,S}}[l]$ respectively.
By hypothesis,  we have $\ev_\bfw(x)\in F_{\bfw,R_{\bfw,S}}[<\!l]$.
Thus 
$$\gr\ev_\bfw(\sigma_l(x))=\gr\rho_\bfw(\sigma_l(x))(\phi_\bfw)=0.$$
Taking the product over all $\bfw\in\bfN_1$, we deduce that 
$$\sigma_l(x)\in\Ker(\gr\ev_1)\cap \bfU_{R_S}[l].$$
By Lemma \ref{lem:kerU} we deduce that
\begin{align}\label{form4}
\sigma_l(x)\in (\bfU_{R_S}\frakh_{R_S}^v[u]\oplus\bfU_{R_S}\frakg_{R_S}^+[u])\cap \bfU_{R_S}[l].
\end{align}
On the other hand, we have 
\begin{align}\label{form5}
\bfU_{R_S}\frakh_{R_S}^v[u]\oplus\bfU_{R_S}\frakg_{R_S}^+[u]\subset \gr\Ker(\ev_1).
\end{align}
Indeed, since $\gr\Ker(\ev_1)$ is a left ideal of $\bfU_{R_S}$ it is enough to check that
$$\frakh_{R_S}^v[u]\cup\frakg_{R_S}^+[u]\subset \gr\Ker(\ev_1).$$
The inclusion $$\frakh_{R_S}^v[u]\subset \gr\Ker(\ev_1)$$ follows from Lemma \ref{lem:Y02}.
The inclusion $$\frakg_{R_S}^+[u]\subset \gr\Ker(\ev_1)$$ is obvious for weight reason.
From \eqref{form4} and \eqref{form5} 
we deduce that there is an element $$z\in\Ker(\ev_1)\cap\bfY_{R_S}[\leqsl l]$$ such that
$\sigma_l(x)=\sigma_l(z)$.
Hence, setting $y=x-z$ we get 
$$y\in\bfY_{R_S}[<\!l]
,\quad
\ev_1(x)=\ev_1(y).$$ 
This proves Part (a).
Part (b) follows from (a).
Parts (c), (d) are proved in a similar way.
\end{proof}

We equip the $R_S$-module $\bfY_{R_S}/\Ker(\ev_a)$
with the quotient of the $u$-filtration of $\bfY_{R_S}$.
From Lemmas \ref{lem:KK}, \ref{lem:kerU}, \ref{lem:evstrict} we deduce the following.

\begin{lemma}\label{lem:YU}
Let $a=0,1$ or $2$. We have
\hfill
\begin{enumerate}[label=$\mathrm{(\alph*)}$,leftmargin=8mm,itemsep=1.2mm]
\item 
$\gr(\bfY_{R_S}/\Ker(\ev_a))=\bfU_{R_S}/\Ker(\gr \ev_1)\simeq\bfU_{R_S}^-\bfU_{R_S}^{0,w}$,
\item
$\gr(\bfY_{R_S}^\vee/\Ker(\ev_a^\vee))=\bfU_{R_S}^\vee/\Ker(\gr \ev_1^\vee)\simeq
\bfU_{R_S}^{-,\vee}\bfU_{R_S}^{0,w,\vee}.$
\end{enumerate}
\qed
\end{lemma}

For each $\bfw\in\bfN_1$ with $w=|\bfw|$, let 
$I_w\subset\bfU_{R_S}^{0,w}$ be the ideal generated the subset
$J_w\cap U(\frakh_{R_S}^w).$
The isomorphism  \eqref{UR}
identifies $I_w$ with the ideal 
\begin{align}\label{I}(p_{i,0}-w_i\,;\,i\in I)\subset R_{\infty,S}.
\end{align}
Let $\bfU_{-v,\leqsl l,R_S}$ and $\bfY_{-v,\leqsl l,R_S}$ denote the
filtrations of the weight spaces 
$\bfU_{-v,R_S}$ and $\bfY_{-v,R_S}$ relatively to the cohomological grading.

\begin{lemma}\label{lem:39} Let $v\in\bbN I$, $l\in\bbN$ and $\bfw\in \bfN_1$. Set $w=|\bfw|$.
Assume $w$ is large enough.
\hfill
\begin{enumerate}[label=$\mathrm{(\alph*)}$,leftmargin=8mm,itemsep=1.2mm]
\item 
$\Ker(\gr \ev_\bfw)\cap\bfU_{-v,\leqsl l,R_S}=(\Ker(\gr \ev_1)+\bfU_{R_S} I_w)\cap\bfU_{-v,\leqsl l,R_S},$
\item
$\Ker(\ev_w)\cap\bfY_{-v,\leqsl l,R_S}=(\Ker(\ev_1)+\bfY_{R_S} I_w)\cap \bfY_{-v,\leqsl l,R_S}.$
\item
$\Ker(\ev_w^\vee)\cap\bfY_{-v,\leqsl l,R_S}^\vee=(\Ker(\ev_1^\vee)+\bfY_{R_S}^\vee I_w)\cap 
\bfY_{-v,\leqsl l,R_S}^\vee.$
\end{enumerate}
\end{lemma}

\begin{proof}
Let us first prove Part (a).
There is an obvious inclusion
\begin{align}\label{inclusion}\Ker(\gr \ev_\bfw)\supset\Ker(\gr \ev_1)+\bfU_{R_S} I_w\end{align}
We must prove that
$$\Ker(\gr \ev_\bfw)\cap\bfU_{-v,\leqsl l,R_S}\subset\Ker(\gr \ev_1)+\bfU_{R_S} I_w.$$
Fix an element $x$ in the left hand side.
By Lemma \ref{lem:kerU} we have
$$\bfU_{R_S}=\Ker(\gr \ev_1)\oplus\bfU_{R_S}^-\bfU_{R_S}^{0,w}.$$
Write $x=y+z$ with $y\in \Ker(\gr \ev_1)$ and $z\in\bfU_{R_S}^-\bfU_{R_S}^{0,w}.$
By \eqref{inclusion}, we have $z\in\Ker(\gr\ev_\bfw)$.
Set $B=\bfU_{R_S}^-\bfU_{R_S}^{0,w}\cap\bfU_{-v,\leqsl l,R_S}$.
We must check that if $w$ is large enough then
\begin{align}\label{39c}
\Ker(\gr \ev_\bfw)\cap B\subset\bfU_{R_S} I_w.
\end{align}
By \eqref{d} we have 
$$\bigcap_{\bfw\in\bfN_1}\Ker(\gr \ev_\bfw)\cap B=\{0\}.$$
Since $B$ is finite dimensional,
there are tuples $\bfw_1,\bfw_2,\dots,\bfw_s\in\bfN_1$ such that
\begin{align}\label{39d}\bigcap_{r=1}^s\Ker(\gr \ev_{\bfw_r})\cap B=\{0\}.
\end{align}
Set 
\begin{align}\label{39d}\bfw=\bfw_1\bfw_2\cdots\bfw_s\in\bfN_1.
\end{align}
Fix any $x\in\Ker(\gr \ev_\bfw)\cap B.$
Write 
$$x=\sum_kx_{1,k}x_{2,k}
,\quad
x_{1,k}\in U(\frakg_{R_S}^-[u]\oplus \frakh_{R_S}^w[u]u)
,\quad
x_{2,k}\in U(\frakh_{R_S}^w).$$
Since $\gr \ev_\bfw(x)=0$ and 
$$\phi_{\bfw_1}\otimes\cdots\otimes \phi_{\bfw_s}\in F_{\bfw,R_{\bfw,S}}
\subset\bigotimes_{r=1}^sF_{\bfw_r,R_{\bfw_r,S}},$$ 
we deduce that
$$\sum_kx_{2,k}x_{1,k}
(\phi_{\bfw_1}\otimes\cdots\otimes \phi_{\bfw_s})=0.$$
The tensor product $\bigotimes_{r=1}^sF_{\bfw_r,R_{\bfw_r,S}}$ is $(\bbZ I)^s$-graded.
For weight reasons, we have
$$\sum_kx_{2,k}(
\phi_{\bfw_1}\otimes\cdots\otimes x_{1,k}(\phi_{\bfw_r})\otimes\cdots\otimes\phi_{\bfw_s})=0,\quad\forall r=1,\dots,s.$$
By lemma \ref{lem:Y02} for each $k$ there is a polynomial $f_k(b_i)$ in the variables $b_i$ with $i\in I$,
such that $$x_{2,k}(\phi_\bfw)=f_k(w_i)\phi_\bfw.$$
Hence, we have
$$\sum_kf_k(w_i)x_{1,k}\in\bigcap_{r=1}^s\Ker(\gr \ev_{\bfw_r})\cap B.$$
Hence \eqref{39d} implies that 
$$\sum_kf_k(w_i)x_{1,k}=0.$$
We may assume that the elements $x_{1,k}$ are linearly independent.
We deduce that $f_k(w_i)=0$, for each $k$.
Hence the inclusion \eqref{39c} holds for the tuple $\bfw$ in \eqref{39d}.
Hence it holds again for any tuple $\bfw$ such that $w$ is large enough,
proving the part (a).
Now, we prove (b).
We have 
\begin{align}\label{39a}
\Ker(\ev_\bfw)\supset\Ker(\ev_1)+\bfY_{R_S} I_w.
\end{align}
By \eqref{Ker3}  we have
$\Ker(\ev_w)=\Ker(\ev_\bfw).$
Hence
$$\gr\Ker(\ev_w)\cap \bfU_{-v,\leqsl l,R_S}\subset\Ker(\gr \ev_\bfw)\cap \bfU_{-v,\leqsl l,R_S}.$$
By Part (a) we deduce that
$$\gr\Ker(\ev_w)\cap \bfU_{-v,\leqsl l,R_S}\subset(\Ker(\gr \ev_1)
+\bfU_{R_S} I_w)\cap \bfU_{-v,\leqsl l,R_S}.$$
Thus Lemma \ref{lem:evstrict} yields
\begin{align*}
\gr\Ker(\ev_w)\cap \bfU_{-v,\leqsl l,R_S}\subset(\gr\Ker(\ev_1)+
\bfU_{R_S} I_w)\cap \bfU_{-v,\leqsl l,R_S}.
\end{align*}
For each $R_S$-submodule $M\subset\bfY_{R_S}$, we have
$$\gr(M\cap\bfY_{-v,\leqsl l,R_S})=\gr(M)\cap \bfU_{-v,\leqsl l,R_S},\quad
v\in\bbN I,
l\in\bbN,$$
where the associated graded are relative to the
filtrations induced by the $u$-filtration on $\bfY_{R_S}$.
Hence
\begin{align*}
\gr(\Ker(\ev_w)\cap \bfY_{-v,\leqsl l,R_S})
&=\gr\Ker(\ev_w)\cap \bfU_{-v,\leqsl l,R_S}\\
&\subset(\gr\Ker(\ev_1)+\bfU_{R_S} I_w)\cap \bfU_{-v,\leqsl l,R_S}\\
&=(\gr\Ker(\ev_1)+\gr(\bfY_{R_S} I_w))\cap \bfU_{-v,\leqsl l,R_S}\\
&\subset\gr(\Ker(\ev_1)+\bfY_{R_S} I_w)\cap \bfU_{-v,\leqsl l,R_S}\\
&=\gr((\Ker(\ev_1)+\bfY_{R_S} I_w)\cap\bfY_{-v,\leqsl l,R_S})
\end{align*}
Comparing with \eqref{39a}, we deduce that
\begin{align}\label{39b}
\Ker(\ev_w)\cap \bfY_{-v,\leqsl l,R_S}=(\Ker(\ev_1)+\bfY_{R_S} I_w)\cap\bfY_{-v,\leqsl l,R_S}\end{align}
This proves Part (b). The part (c) is proved as the part (b) above, using 
Lemmas \ref{lem:Y03}, \ref{lem:kerU}(b) and \ref{lem:evstrict}(d)
instead of 
Lemmas \ref{lem:Y02}, \ref{lem:kerU}(a) and \ref{lem:evstrict}(b).
\end{proof}

\bigskip

\section{From COHA's to Yangians}

\medskip

\subsection{The embedding of the COHA into the Yangian}
To simplify the notation, in this section we set $S=T$.
Recall that 
$\bfY_K=\bfY_R\otimes_RK$ and
$\bfY_K^\vee=\bfY_R^\vee\otimes_RK.$
By base change, 
the evaluation map $\ev_\bfw$  yields the $K$-linear maps
$$\ev_\bfw:\bfY_K\to F_{\bfw,R_{\bfw}}\otimes_RK
,\quad
\ev_\bfw^\vee:\bfY_K^\vee\to F_{\bfw,R_{\bfw}}^\vee\otimes_RK
.$$
We define
$$
F_{a,K}=\prod_{\bfw\in\bfN_a}F_{\bfw,R_{\bfw}}\otimes_RK
,\quad
F_{a,K}^\vee=\prod_{\bfw\in\bfN_a}F_{\bfw,R_{\bfw}}^\vee\otimes_RK
,\quad
a=0,1,2.$$
Taking the product of the evaluation maps over all tuples $\bfw\in \bfN_a$ we get the maps
$$\ev_a:\bfY_K\to F_{a,K}
,\quad
\ev_a^\vee:\bfY_K^\vee\to F_{a,K}^\vee
.$$
Lemma \ref{lem:rhovee} 
yields a $K$-algebra isomorphism
$\bfY_{K}^\vee=\bfY_{K}.$
Similarly, Proposition \ref{prop:COHA} yields
$$\overline\Y_K^\vee=\overline\Y_R^\vee\otimes_R K
,\quad 
\overline\Y_K=\overline\Y_R\otimes_R K
,\quad
\overline\Y^\vee_K=\overline\Y_K.$$

\begin{proposition}\label{prop:Phi}
There is a $\bbZ I$-graded $K$-algebra embedding 
\begin{align}\label{Phi}
\Phi:\overline\Y_K\to \bfY_K
\end{align}
 such that $R_{\infty,K}$ maps isomorphically to $\bfY_K^{0,w}$.
The morphism $\Phi$ intertwines the representations 
of $\overline\Y_K$ and $\bfY_K$ on
$F_{w,R_w}\otimes_RK$ for each $w\in\bbN I$.
\end{proposition}

\begin{proof}
The proposition follows from the results in \cite{SV18b}.
We'll give a simpler proof using some computation in \cite{N23a}.
It is enough to construct a $\bbZ I$-graded $K$-algebra homomorphism 
\begin{align}\label{Phi'}\Y_K\to \bfY_K\end{align} 
compatible with the representations of $\Y_K$ and $\bfY_K$ on $F_{w,R}\otimes_RK$. 
Then, the map \eqref{Phi'} extends to a $\bbZ I$-graded $K$-algebra homomorphism 
$\Phi:\overline\Y_K\to \bfY_K$ such that
$R_{\infty,K}$ maps to $\bfY_K^{0,w}$
via the isomorphism \eqref{UR}.
The map $\Phi$ is injective because it intertwines the evaluation maps 
$$\ev_2: \overline\Y_K\to F_{2,K}
,\quad
\ev_2: \bfY_K\to F_{2,K},$$ and the first one 
is injective by
Proposition \ref{prop:COHA}.

To construct the map \eqref{Phi'}, we fix $v,w\in\bbN I$, $i\in I$, and
we choose the cocharacter $\sigma$ as in \eqref{sigma} with $s=2$, $w_1=\delta_i$ and $w_2=w$.
Let $F,F'\in Fix$ be the connected components of the fixed point set
$\frakM(v+\delta_i,w+\delta_i)^\sigma$
such that
\begin{align}\label{FF}
F=\frakM(\bfv,\bfw)
,\quad
F'=\frakM(\bfv',\bfw)
,\quad
\bfw=(\delta_i,w)
,\,
\bfv=(\delta_i,v)
,\,
\bfv'=(0,v+\delta_i).
\end{align}
By \eqref{order} the connected component $F'$ is minimal in the poset $Fix$, $F$ is subminimal and $F'\prec F$.
Recall the formulas \eqref{Fdelta}. 
We'll compute the operator 
$$E_{w,R}(m_iu^l)\in\Hom_{R_w}(F_{v,w,R_w},F_{v+\delta_i,w,R_w})$$
with
$$m_iu^l\in\Hom_{R[u]}(F_{\delta_i,\delta_i,R_{\delta_i}},F_{0,\delta_i,R_{\delta_i}})=R[u]\langle -g_i\rangle
.$$
Here $m_i$ denotes the obvious element of degree $g_i$ in $\Hom_{R[u]}(F_{\delta_i,\delta_i},F_{0,\delta_i})$.
To do so, we consider the Hecke correspondence
$$\frakP(v+\delta_i,v,w)\subset\frakM(v+\delta_i,w)\times\frakM(v,w)$$
which parametrizes the isomorphism classes of triples 
$(x,y,\iota)$ consisting of a surjection $\iota:x\to y$ with
$x\in\X_\Pi(v+\delta_i,w)_s$ and $y\in\X_\Pi(v,w)_s.$
Taking the kernel of the map $\iota$ yields a map
\begin{align}\label{map}\frakP(v+\delta_i,v,w)\to\calX_\Pi(\delta_i)\end{align}
where the right hand side is the quotient stack.
The linear character of the group $G_{\delta_i}$ yields a line bundle $\calL_i$ on the stack
$\calX_\Pi(\delta_i)$.
Let $\calL_i$ denote also its pullback by the map \eqref{map}, and
$\cc_1(\calL_i)$ be its 1st Chern class.
Following \S\ref{sec:defY}, let $u$ denote the 1st Chern class of the linear character of the group 
$G_{\delta_i}$.
The Hecke correspondence
$\frakP(v+\delta_i,v,w)$ is proper over the quiver variety $\frakM(v+\delta_i,w)$.
Hence, by convolution, it yields a $R_w$-linear map
$$F_{v,w,R_w}\to F_{v+\delta_i,w,R_w}.$$
The proposition is a consequence of the following lemmas, the second of which is proved in 
\cite{N23b} and \cite{SV18}.

\begin{lemma}\label{lem:c1P} For each integer $l\in\bbN$, the operator 
$$E_{w,R}(m_iu^l):F_{v,w,R_w}\to F_{v+\delta_i,w,R_w}$$ 
coincides with the convolution with the class
\begin{align}\label{c1P}
\pm\cc_1(\calL_i)^l\cap[\frakP(v+\delta_i,v,w)]
\end{align}
The sign is determined by the polarization.
\end{lemma}

\begin{lemma}[\!\cite{N23b},\cite{SV18}]\label{lem:generators}
\hfill
\begin{enumerate}[label=$\mathrm{(\alph*)}$,leftmargin=8mm,itemsep=1.2mm]
\item 
The $\K$-algebra $\Y_K$ is generated by the subset 
$\{\cc_1(\calL)^l\cap[\calX_\Pi(\delta_i)]\,;\,i\in I\,,\,l\in\bbN\}.$
\item
The element $\cc_1(\calL)^l\cap[\calX_\Pi(\delta_i)]$ in $\Y_R$
acts on $F_{w,R_w}$ by convolution with the class \eqref{c1P}.
%$\pm\cc_1(\calL_i)^l\cap[\frakP(v+\delta_i,v,w)]$ 
%in $H^{G_w\times T}_\bullet(\frakM(v+\delta_i,w)\times\frakM(v,w),\bbQ)$.
\end{enumerate}
\qed
\end{lemma}

\begin{proof}[Proof of Lemma $\ref{lem:c1P}$]
For any smooth $G\times T$-varieties $M_1$, $M_2$ and $M_3$, the
localization theorem yields $K_{G\times T}$-linear isomorphisms
\begin{align}\label{res}
\bfr_{ij}:H^{G\times T}_\bullet(M_i\times M_j,\bbQ)_{K_{G\times T}}\to
H^{G\times T}_\bullet(M_i^A\times M_j^A,\bbQ)_{K_{G\times T}}
,\quad
i,j=1,2,3,
\end{align}
which commute with the convolution. They are analogous to the map in K-theory given in \cite[thm.~5.11.10]{CG}.
The map $\bfr_{ij}$ is the Gysin restriction times a correction factor equal to the inverse of the 
equivariant euler class 
$$\eu(N_{M_j^A}M_j)^{-1}\in H_{G\times T}^\bullet(M_i^A\times M_j^A,\bbQ)_{K_{G\times T}}.$$
Let $X=\frakM(v+\delta_i,w+\delta_i)$ and let $F$, $F'$ be as in \eqref{FF}.
Recall the closed Lagrangian subvarieties 
$\calA^\sigma_F\subset X\times F$ and
$\calA^\sigma_{F'}\subset X\times F'.$
The connected component $F'$ of $X^A$ is the only component strictly smaller than $F$ in the poset $Fix$.
Hence, we have
\begin{align*}
\calA^\sigma_F\cap(X^A\times F)&=(F\times F)\sqcup(F'\times F),\\
\calA_{F'}\cap(X^A\times F')&=F'\times F'.
\end{align*}
Set $M_1=M_3=X^A$ and $M_2=X$.
Applying the map $\bfr_{23}$ to the cycle
$\calS^\sigma_F$ in $M_2\times M_3$
and $\bfr_{12}$ to the cycle $\calS_{F'}^{\sigma,\op}$ in $M_1\times M_2$ yields the cycles
\begin{align*}
\bfr_{23}(\calS^\sigma_F)&=
\calS^\sigma_{F,F}+\calS^\sigma_{F',F}\in H^{G_\bfw\times T}_\bullet(X^A\times F,\bbQ)_{K_\bfw}
,\\
\bfr_{12}(\calS_{F'}^{\sigma,\op})&=
\calS_{F',F'}^{\sigma,\op}\in H^{G_\bfw\times T}_\bullet(F'\times X^A,\bbQ)_{K_\bfw}
\end{align*}
which are supported on 
$(F\times F)\sqcup(F'\times F)$ and $F'\times F'$ respectively.
To compute the operator $E_{w,R}(m_iu^l)$ it is enough to compute the operator
$$R_{F,F'}:H^{G_\bfw\times T}_\bullet(F,\bbQ)_{K_\bfw}\to 
H^{G_\bfw\times T}_\bullet(F',\bbQ)_{K_\bfw}$$
given by the restriction of the R-matrix.
Since $\bfr_{13}=\id$ and the $\bfr_{ij}$'s are compatible with the convolution, 
the operator $R_{F,F'}$ is the convolution with the cycle 
\begin{align}\label{SS}
\calS_{F',F'}^{\sigma,\op}\star\calS^\sigma_{F',F}\in H^{G_\bfw\times T}_\bullet(F'\times F,\bbQ)_{K_\bfw}.
\end{align}
We must compare \eqref{c1P} with \eqref{SS}.
To do that we'll use the following notation
\begin{itemize}[label=$\mathrm{(\alph*)}$,leftmargin=8mm,itemsep=1.2mm]
\item[-] 
$u=$ 1st Chern class of the linear character of the group $G_{\delta_i}$,
\item[-]
$\{\beta_{j,x}\}=$ set of Chern roots (=weights) of the $j$th fundamental representation of $G_w$,
\item[-]
$\{\alpha_{j,x}\}=$ set of Chern roots of the tautological bundle $\calV_j$ over $\frakM(v,w)$, 
\item[-]
$\{\tilde\alpha_{j,x}\}=$ set of Chern roots of the tautological bundle $\calV_j$ over 
$\frakM(v+\delta_i,w)$. 
\end{itemize}

\begin{lemma}
\hfill
\begin{enumerate}[label=$\mathrm{(\alph*)}$,leftmargin=8mm,itemsep=1.2mm]
\item 
The Hecke correspondence $\frakP(v+\delta_i,v,w)$
is the zero set of a regular section of a 
$G_w\times T$-equivariant vector bundle over
$\frakM(v+\delta_i,w)\times\frakM(v,w)$ of rank
\begin{align*}
(v+\delta_i)\cdot w+v\cdot w-(v,v+\delta_i)-1
=d_{v,w}+d_{v+\delta_i,w}-g_i.
\end{align*}
We have
\begin{align*}
[\frakP(v+\delta_i,v,w)]&=\xi\cap[\frakM(v+\delta_i,w)\times\frakM(v,w)]
\end{align*}
where the cohomology class $\xi$ is given in the formula \eqref{gamma} below.
%\item
%The Hecke correspondence $\frakP(v+\delta_i,v,w)^{\frac{1}{2}}$
%is the zero set of a $G_w\times T$-equivariant regular section of the trivial
%vector bundle of rank $g_i$ over $\frakP(v+\delta_i,v,w)$.
%We have
%\begin{align*}
%[\frakP(v+\delta_i,v,w)^{\frac{1}{2}}]&=\gamma_i\cap[\frakP(v+\delta_i,v,w)]
%\end{align*}
%where the cohomology class $\gamma_i$ is given in the formula \eqref{gamma} below.
\item 
The variety $\calA^\sigma_F$ is the zero set of a regular section of a $G_\bfw\times T$-equivariant
vector bundle over 
$\frakM(v+\delta_i,w+\delta_i)\times \frakM(\delta_i,\delta_i)\times\frakM(v,w)$ of rank
\begin{align*}
(v+\delta_i)\cdot w+v\cdot(w+\delta_i)-(v,v+\delta_i)+2g_i
=d_{v+\delta_i,w+\delta_i}+d_{\delta_i,\delta_i}+d_{v,w}.
\end{align*}
We have $\calS^\sigma_{F}=\pm[\calA^\sigma_F]$, the sign being determined by the polarization, and
$$[\calA^\sigma_F]=\eta\cap[\frakM(v+\delta_i,w+\delta_i)\times \frakM(\delta_i,\delta_i)\times\frakM(v,w)]$$
where the cohomology class $\eta$ is given in the formula \eqref{eta} below.
\item 
The projection $\calA^\sigma_{F'}\to X$ is a closed embedding.
The projection $\calA^\sigma_{F'}\to F'=\frakM(v+\delta_i,w)$ is a fiber bundle isomorphic to $\calV_i$.
We have $\calS^\sigma_{F'}=\pm[\calA^\sigma_{F'}]$,
the sign being determined by the polarization.
%$$[\calA^\sigma_{F'}]=\eta'\cap[\frakM(v+\delta_i,w+\delta_i)\times\frakM(v+\delta_i,w)]$$
%where the cohomology class $\eta'$ is given in the formula \eqref{eta'} below.
\end{enumerate}
\end{lemma}

\begin{proof}
Let $\gamma_i$ and $\gamma_i^*$ be the products of the weights of the 
$g_i$ loops $i\to i$ in $Q_1$ and $Q_1^*$ respectively.
Part (a) is done in \cite[\S 5]{N98} for a quiver without 1-loops,
the case of a general quiver is done in \cite[\S 2.20]{N23a}.
%The cohomology class $\gamma_i$ in (b) is the product of the weights of the 
%$g_i$ loops $i\to i$ in $Q_1$. 
Using \cite[(2.72)]{N23a}, we get
\begin{align}\label{gamma}
\begin{split}
\gamma_i&=\prod_{\alpha:i\to i}t_\alpha\\
\gamma_i^*&=\prod_{\alpha:i\to i}(\hbar-t_\alpha)\\
\xi&=
\prod_{\alpha:j\to k}(t_\alpha+\tilde\alpha_{j,x}-\alpha_{k,y})(\hbar-t_\alpha+\tilde\alpha_{k,y}-\alpha_{j,x})\cdot
\prod_{j\in I}(\beta_{j,y}-\alpha_{j,x})(\hbar+\tilde\alpha_{j,x}-\beta_{j,y})\cdot\\
&\qquad\cdot\prod_{j\in I}(\tilde\alpha_{j,x}-\alpha_{j,y})^{-1}(\hbar+\tilde\alpha_{j,x}-\alpha_{j,y})^{-1}\cdot\hbar^{-1}
\end{split}
\end{align}
where $\alpha$ runs over $Q_1$.
The 1st claim in Part (b) is proved in \cite[claim 3.17]{N23a}.
The cohomology class $\eta$ is the Euler characteristic of the vector bundle in Part (b). 
Using \cite[(3.63)]{N23a}, we get
\begin{align}\label{eta}
\begin{split}
&\eta=
\hbar\cdot\xi\cdot\gamma_i\cdot\gamma_i^*\cdot\prod_x(u-\alpha_{i,x})
\end{split}
\end{align}
where $\alpha$ runs over the set $Q_1$.
To prove that $\calS^\sigma_{F}=\pm[\calA^\sigma_F]$ we must check the conditions (b) and (c) in \S\ref{sec:stable}.
Condition (b) follows from the inclusion $Z^+_F\subset\calA^\sigma_F$
which implies that the Gysin restriction of the cycle
$[\calA^\sigma_F]$ to $F\times F$ is $\eu(N_F^{\sigma,-}X)\cap[\Delta_F]$. 
Condition (c) is a direct computation.
More precisely, we have
\begin{align*}
[\calA^\sigma_F]|_{F'\times F}=(\eta|_{F'\times F})\cap[F'\times F]
\end{align*}
Hence, we must check that 
$$A\text{-}\deg(\eta|_{F'\times F})<\dim X-\dim F'=2d_{v+\delta_i,w+\delta_i}-2d_{v+\delta_i,w}=2v_i+2.$$
This follows from the formula \eqref{eta}.
Part (c) of the lemma follows from the minimality of the connected component $F'$ in the poset $Fix$,
which implies that $Z^+_{F'}=\calA^\sigma_{F'}$.
The second claim is a direct computation using the formula \cite[cor.~3.12]{N98}
for the tangent bundle of quiver varieties, which yields the following 
vector bundle isomorphisms
$N_{F'}^-X=\calV_i^*$ and $N_{F'}^+X=\calV_i$.
Note that \cite{N98} is done for quiver without 1-loops but the formula extends to our case as well.
\end{proof}

From the lemma and the definition of the map $\bfr_{ij}$ in \eqref{res}, we deduce that
\begin{align*}
\calS_{F',F'}^{\sigma,\op}
&=\pm\eu(N_{F'}X)^{-1}\cap[(\calA^\sigma_{F'})^*]|_{F'\times F'}\\
&=\pm\eu(N_{F'}X)^{-1}\cdot\eu(N_{F'}^-X)\cap[\Delta_{F'}]\\
%&=\pm\eu(N_{F'}^+X)^{-1}\cdot\eu(TF')\cap[F'\times F']\\
%=\pm\prod(\alpha_i-u)^{-1}\cap[F'\times F']\\
\calS^\sigma_{F',F}&=\pm[\calA^\sigma_F]|_{F'\times F},\\
&=\pm\eta\cap[F'\times F]
\end{align*}
Hence, the cycle in \eqref{SS} is
$$\calS_{F',F'}^{\sigma,\op}\star\calS^\sigma_{F',F}=
\pm\eu(N_{F'}^+X)^{-1}\cdot\eta\cap[F'\times F].$$
A computation using the formula \cite[cor.~3.12]{N98}
for the tangent bundle of quiver varieties yields 
\begin{align*}
\eu(T\frakM(\delta_i,\delta_i))&=\gamma_i\cdot\gamma_i^*\\
\eu(N_{F'}^+X)&=\prod_y(-u+\tilde\alpha_{i,y}).
\end{align*}
We have
$$\eu(N_{F'}^+X)^{-1}\cdot\eta=\hbar\cdot\xi\cdot\gamma_i\cdot\gamma_i^*\cdot\prod_x(u-\alpha_{i,x})
\cdot\prod_y(-u+\tilde\alpha_{i,y})^{-1}.$$
We deduce that
\begin{align*}
\calS_{F',F'}^{\sigma,\op}\star\calS^\sigma_{F',F}
&=\pm\hbar\cdot\xi\cdot\gamma_i\cdot\gamma_i^*\cdot\prod_x(u-\alpha_{i,x})
\cdot\prod_y(u-\tilde\alpha_{i,y})^{-1}\cap[F'\times F]\\
&=\pm\hbar\cdot\xi\cdot\prod_x(u-\alpha_{i,x})
\cdot\prod_y(u-\tilde\alpha_{i,y})^{-1}\cap[\frakM(v+\delta_i,w)\times\frakM(v,w)]\\
&=\pm\hbar\cdot\prod_x(u-\alpha_{i,x})
\cdot\prod_y(u-\tilde\alpha_{i,y})^{-1}\cap[\frakP(v+\delta_i,v,w)]\\
&=\pm\hbar\cdot(u-\cc_1(\calL_i))^{-1}\cap[\frakP(v+\delta_i,v,w)].
\end{align*}
Let $S(u)$ be the expansion of the expression above, viewed as a formal series in $u$.
We have 
$$E_{w,R}(m_iu^l)=\frac{1}{\hbar}\Res_u(u^l\cdot S(u)).$$
Hence, the operator $E_{w,R}(m_iu^l)$ is the convolution with the cycle 
$$\pm\cc_1(\calL_i)^l\cap[\frakP(v+\delta_i,v,w)]$$
considered in \eqref{c1P}. The lemma \ref{lem:c1P} is proved.
\end{proof}
\end{proof}

\smallskip

\begin{remark}
We consider the closed substack
$$\calX_\Pi^\vee(v)=\X_\Pi^\vee(v)\,/\,G_v\subset\calX_\Pi(v)$$
of nilpotent representations.
The nilpotent Hecke correspondence 
$\frakP^\vee(v+\delta_i,v,w)$ is the inverse image 
of $\calX_\Pi^\vee(\delta_i)$ by the map 
$$\frakP(v+\delta_i,v,w)\to\calX_\Pi(\delta_i).$$
See \S\ref{sec:Hall} for details. 
The linear character of the group $G_{\delta_i}$ yields a line bundle $\calL_i$ on
$\calX_\Pi^\vee(\delta_i)$.
Let $\calL_i$ denote the pullback of this line bundle to $\frakP^\vee(v+\delta_i,v,w)$.
We abbreviate 
$$e_{i,l}=\cc_1(\calL_i)^l\cap[\calX_\Pi^\vee(\delta_i)]\in\Y^\vee
,\quad
i\in I
,\,
l\in\bbN.$$
By Proposition \ref{prop:COHA} we have a $\K$-algebra isomorphism
$\Y_K=\Y^\vee_K.$
This isomorphism, together with Lemma \ref{lem:generators},
implies that the $\K$-algebra $\Y^\vee_K$ is generated by the set 
$\{e_{i,l}\,;\,i\in I\,,\,l\in\bbN\}.$
According to \cite{SV18}, the element $e_{i,l}$ of $\Y^\vee$
acts on the modules $F_w$ and $F_w^\vee$ by convolution with the class
$$\cc_1(\calL_i)^l\cap[\frakP^\vee(v+\delta_i,v,w)].$$
Thus, the $K$-algebra embedding $\Phi:\overline\Y_K\to\bfY_K$ is given by
$$\Phi(e_{i,l})=E_{0,K}(\gamma_i \gamma^*_i m_iu^l).$$
Here $\gamma_i,\gamma_i^*\in R$ are the degree $g_i$ elements given in \eqref{gamma}.
\end{remark}

\medskip

\subsection{The isomorphism of the COHA and the Yangian}
Recall that 
$$F_{a,K_S}^\vee=F_{a,R_S}^\vee\otimes_{R_S}K_S
,\quad
a=0,1,2.$$
The evaluation on the vacuum gives the maps
\begin{align*}
\ev_1^\vee&:\bfY_{K_S}^\vee\to F_{1,K_S}^\vee
,\\
\gr \ev_1^\vee&:\bfU_{K_S}^\vee\to F_{1,K_S}^\vee
\end{align*}
The map $\ev_1^\vee$ factors through a map
$$\overline\ev_1^\vee:\bfY_{K_S}^\vee\,/\,\Ker(\ev_1^\vee)_{K_S}\to F_{1,K_S}^\vee.$$
We consider the $R_S$-submodules 
\begin{align*}
(\overline\ev_1^\vee)^{-1}(F_{1,R_S}^\vee)&\subset\bfY_{K_S}^\vee\,/\,\Ker(\ev_1)_{K_S}
,\\
(\gr \ev_1^\vee)^{-1}(F_{1,R_S}^\vee)&\subset\bfU_{K_S}^\vee.
\end{align*}

\begin{lemma}\label{lem:last}
We have
\hfill
\begin{enumerate}[label=$\mathrm{(\alph*)}$,leftmargin=8mm,itemsep=1.2mm]
\item 
$\bfU_{R_S}^{-,\vee}\bfU_{R_S}^{0,w,\vee}=
\bfU_{K_S}^{-,\vee}\bfU_{K_S}^{0,w,\vee}\cap (\gr \ev_1^\vee)^{-1}(F_{1,R_S}^\vee).$
\item
$\bfY^\vee_{R_S}\,/\,\Ker(\ev_1^\vee)=(\overline\ev_1^\vee)^{-1}(F_{1,R_S}^\vee)$.
\end{enumerate}
\end{lemma}

\begin{proof}
Let us prove Part (a). Let $(x^{\pm}_n\,;\,n\in \bbN)$ be dual homogeneous $R_S$-bases of 
$\frakg_{R_S}^{\pm,\vee}$. We 
set 
$$v_n=\wt(x^+_n)=-\wt(x^-_n).$$ 
Relatively to the pairing \eqref{pairing2vee}, we have 
$$\la x^+_{n},x^-_{m}\ra=\delta_{n,m}.$$ Thus
\begin{align}\label{E:proofintegral1}
\begin{split}
v_n=v_m&\Rightarrow[x^+_nu^d,x^-_mu^l]=\delta_{n,m}h_{v_n}u^{d+l},\\
v_n\neq v_m&\Rightarrow[x^+_nu^d,x^-_mu^l] \in \frakg_{R_S}^{\pm,\vee} u^{d+l}
\end{split}
\end{align}
First, we consider an homogeneous element $x\in\bfU_{K_S}^{-,\vee}$ such that
 $\gr \ev_1^\vee(x) \in F_{1,R_S}^\vee$. Set 
\begin{equation}\label{E:proofintegral2}
x=\sum_{\bfn,\bfd} a_{\bfn,\bfd} \,x^-_{n_1}u^{d_1} \cdots  x^-_{n_s}u^{d_s}
, \quad 
a_{\bfn,\bfd} \in K_S
\end{equation}
The sum runs over all pairs $(\bfn,\bfd)$ which are ordered lexicographically 
$$(n_1,d_1) \leqsl (n_2,d_2) \leqsl \cdots\leqsl (n_s,d_s).$$
 We will prove that $a_{\bfn,\bfd} \in R_S$ for all $(\bfn,\bfd)$. 
 Let $s_0$ be the maximal length of tuples 
 $(\bfn,\bfd)$ such that $a_{\bfn,\bfd}\neq 0$. 
 Arguing by induction and using the fact that 
 $$\gr \ev_1^\vee(\bfU_{R_S}^\vee) \subset F_{1,R_S}^\vee,$$ 
 it is enough to prove that $a_{\bfn,\bfd} \in R_S$ for all tuples 
 $(\bfn,\bfd)$ of length $s_0$. 
 Let 
 $$(\bfn,\bfd)=(n_1, \ldots, n_{s_0},d_1,\ldots,d_{s_0})$$ 
 be such a tuple. Let us consider for each $\bfw \in \bfN_1$ the element
$$x_\bfw=x^+_{n_{s_0}} \cdots x^+_{n_1}x (\phi_\bfw^\vee) =
[x_{n_{s_0}}^+, \cdots [x^+_{n_1},x]\ldots]  (\phi_\bfw^\vee) \in R_{\bfw,S}\phi_\bfw^\vee.$$
We claim that
\begin{equation}\label{E:proofintegral3}
x_\bfw \in  c\;a_{\bfn,\bfd}\; h_{v_{n_1}}u^{d_1} \cdots h_{v_{n_{s_0}}}u^{d_{s_0}}(\phi_\bfw^\vee) + \sum_{s<s_0}
\sum_{\bfv',\bfd'} K_S\; h_{v'_1}u^{d'_1} \cdots h_{v'_s}u^{d'_s}(\phi_\bfw^\vee)
\end{equation}
for some positive integer $c$. To see this, observe that after applying to $x$ the operators 
$\ad(x^+_{n_1}),\dots,\ad(x^+_{n_{s_0}})$ and acting on $\phi_\bfw^\vee$ 
we obtain a linear combination of terms of the form
$$y=y_1u^{d_1} \cdots y_{s}u^{d_{s}}(\phi^\vee_{\bfw})
,\quad
y_r \in \frakg_{R_S}$$
with
$\sum_r\wt(y_r)=0$ and $s\leqsl s_0.$ Unless $\wt(y_r)=0$ for all $r$, 
there exists an integer $r\in[1,s]$ such that $\wt(y_r) >0$. 
Then, we have
$$y=\sum_{r'>r} y_1u^{d_1}\cdots \widehat{y_ru^{d_r}} \cdots y_{r'-1}u^{d_{r'-1}}\,[y_r,y_{r'}]u^{d_r+d_{r'}} \,y_{r'+1}u^{d_{r'+1}} \cdots (\phi^\vee_{\bfw})$$
where $\widehat{A}$ means that the term $A$ is omitted.
Iterating this process if necessary, we see that 
$$y \in \sum_{s<s_0}\sum_{\bfv',\bfd'} K_S\; h_{v'_1}u^{d'_1} \cdots h_{v'_s}u^{d'_s}(\phi_\bfw^\vee).$$
If $\wt(y_r)=0$ for all $r$ then the relations \eqref{E:proofintegral1} 
imply that $y$ comes from the term 
$$a_{\bfn,\bfd} x^-_{n_1}u^{d_1} \cdots  x^-_{n_{s_0}}u^{d_{s_0}}(\phi_\bfw^\vee)$$ and is equal to 
$$a_{\bfn,\bfd} h_{v_{n_1}}u^{d_1} \cdots h_{v_{n_{s_0}}}u^{d_{s_0}}(\phi_\bfw^\vee).$$ 
This term may occur several times in $x_\bfw$ if $n_1, \ldots, n_{s_0}$ are not all distinct,
hence the constant $c$.
This proves the formula \eqref{E:proofintegral3}.
Now Lemma \ref{lem:Y02} yields
$$h_{v}u^l(\phi^\vee_{\bfw})=\sum_{i\in I} v_i\; \ch_l(\mathcal{W}_i) \cap \phi_\bfw^\vee
,\quad
\ch_l(\mathcal{W}_i)\in R_{\bfw,S}. $$
The tuple formed the $\ch_l(\mathcal{W}_i)$'s for all $\bfw\in\bfN_1$ yields an element
$$\ch_l(\mathcal{W}_i)\in\prod_{\bfw \in \bfN_1} R_{\bfw,S}.$$
These tuples are algebraically independent
for all $i \in I$ and $l \in\bbN$. 
Moreover, for each
fixed $\bfw$, the set $\{\ch_l(\mathcal{W}_i)\,;\,i \in I,\,l \in\bbN\}$ 
generates the $R_S$-algebra $R_{\bfw,S}$.
Hence, any finite set of monomials in the $\ch_l(\mathcal{W}_i)$'s may be completed into an 
$R_S$-basis of $R_{\bfw,S}$ for $\bfw \gg 0$. Considering $x_\bfw$ for $\bfw \gg 0$ and using \eqref{E:proofintegral3} we deduce that $a_{\bfn,\bfd} \in R_S$. 
This proves (a) for $x \in \bfU_{K_S}^{-,\vee}$. 
The case of $x \in \bfU_{K_S}^{-,\vee} \bfU_{K_S}^{0,w,\vee}$ is proved by the same method,
because the $K_S$-algebra $\bfU_{K_S}^{0,w,\vee}$ is central in $\bfU_{K_S}^\vee$.

Let us prove Part (b).
The $R_S$-module $F_{1,R_S}^\vee$ is torsion free and the following square commutes
\begin{align}\label{square}
\begin{split}
\xymatrix{
\bfY_{R_S}^\vee\,/\,\Ker(\ev_1^\vee)\ar[r]\ar@{^{(}->}[d]_-{\overline\ev_1^\vee}
&\bfY_{K_S}^\vee\,/\,\Ker(\ev_1^\vee)_{K_S}\ar@{^{(}->}[d]^-{\overline\ev_1^\vee}\\
F_{1,R_S}^\vee\ar@{^{(}->}[r]&F_{1,K_S}^\vee}
\end{split}
\end{align}
Hence, the $R_S$-module $\bfY_{R_S}^\vee\,/\,\Ker(\ev_1^\vee)$ embeds into 
$\bfY_{K_S}^\vee\,/\,\Ker(\ev_1)_{K_S}$.
We equip $\bfY_{R_S}^\vee\,/\,\Ker(\ev_1^\vee)$ with the quotient of the $u$-filtration of $\bfY_{R_S}^\vee$.
By base change, this filtration yields a filtration on 
$\bfY_{K_S}^\vee/\Ker(\ev_1^\vee)_{K_S}$.
We equip  the $R_S$-submodule 
$$(\overline\ev_1^\vee)^{-1}(F_{1,R_S}^\vee)\subset\bfY_{K_S}^\vee/\Ker(\ev_1^\vee)_{K_S}$$ 
with the induced filtration. 
By Lemma \ref{lem:evstrict}, the associated graded of the evaluation map
$$\overline\ev_1^\vee:\bfY_{K_S}^\vee\,/\,\Ker(\ev_1^\vee)_{K_S}\to F_{1,K_S}^\vee$$
is a map
$$\gr\, \overline\ev_1^\vee:\bfU_{K_S}^\vee/\Ker(\gr\ev_1^\vee)_{K_S}\to F_{1,K_S}^\vee.$$
We consider the subset
$$(\gr\, \overline\ev_1^\vee)^{-1}(F_{1,R_S}^\vee)\subset\bfU_{K_S}^\vee/\Ker(\gr\ev_1^\vee)_{K_S}.$$
By Corollary \ref{cor:kerU}, the obvious inclusion
$\bfU_{K_S}^{-,\vee}\bfU_{K_S}^{0,w,\vee}\subset\bfU_{K_S}^\vee$ yields an isomorphism 
$$\bfU_{K_S}^{-,\vee}\bfU_{K_S}^{0,w,\vee}=\bfU_{K_S}^\vee\,/\,\Ker(\gr\ev_1^\vee)_K.$$
Under this isomorphism we have
\begin{align}\label{isom}
\bfU_{K_S}^{-,\vee}\bfU_{K_S}^{0,w,\vee}\cap (\gr \ev_1^\vee)^{-1}(F_{1,R_S}^\vee)=
(\gr\, \overline\ev_1^\vee)^{-1}(F_{1,R_S}^\vee)
\end{align}
The square \eqref{square} yields the inclusion
\begin{align*}\bfY_{R_S}^\vee/\Ker(\ev_1^\vee)\subset
(\overline\ev_1^\vee)^{-1}(F_{1,R_S}^\vee).
\end{align*}
Taking the associated graded for the $u$-filtrations, we get a chain of inclusions
\begin{align}\label{subset}\gr(\bfY_{R_S}^\vee/\Ker(\ev_1^\vee))\subset
\gr((\overline\ev_1^\vee)^{-1}(F_{1,R_S}^\vee))
\subset(\gr\, \overline\ev_1^\vee)^{-1}(F_{1,R_S}^\vee).
\end{align}
By Part (a), these inclusions fit into the following commutative diagram
\begin{align*}
\xymatrix{
\gr(\bfY_{R_S}^\vee/\Ker(\ev_1^\vee))\ar@{^{(}->}[r]\ar@{=}[d]
&\gr((\overline\ev_1^\vee)^{-1}(F_{1,R_S}^\vee))\ar@{^{(}->}[r]
&(\gr\,\overline\ev_1^\vee)^{-1}(F_{1,R_S}^\vee)\ar@{=}[d]^-{\eqref{isom}}\\
%\ar@{^{(}->}[r]&\bfU_K^\vee/\Ker(\gr\ev_1^\vee)_K\ar@{=}[d]^-{\text{lem.}\ref{lem:YU}}\\
\bfU_{R_S}^{-,\vee}\bfU_{R_S}^{0,w,\vee}\ar@{=}[rr]^-{\text{(a)}}&&
\bfU_{K_S}^{-,\vee}\bfU_{K_S}^{0,w,\vee}\cap (\gr \ev_1^\vee)^{-1}(F_{1,R_S}^\vee)
%\ar@{^{(}->}[r]&\bfU_K^-\bfU_K^{0,w}
}
\end{align*}
where the left isomorphism is given in Lemma \ref{lem:YU}.
We deduce that the inclusions \eqref{subset} are indeed equalities. In particular, we have
$$\gr(\bfY_{R_S}^\vee/\Ker(\ev_1^\vee))=\gr((\overline\ev_1^\vee)^{-1}(F_{1,R_S}^\vee)).$$
Since 
$\bfY_{R_S}^\vee/\Ker(\ev_1^\vee)\subset(\overline\ev_1^\vee)^{-1}(F_{1,R_S}^\vee)$
by \eqref{square}, this implies that
$$\bfY_{R_S}^\vee/\Ker(\ev_1^\vee)=(\overline\ev_1^\vee)^{-1}(F_{1,R_S}^\vee).$$
\end{proof}

We'll need a stronger version of the preceding lemma. To spell it, recall that for each tuple
$\bfw\in\bfN_1$ the stable envelope yields a $\bfY_{R_S}^\vee\otimes R_{\bfw,S}$-linear injective map
\begin{align}\label{F21}
\stab^\vee(1):F_{w,R_{\bfw,S}}^\vee\to F_{\bfw,R_{\bfw,S}}^\vee
,\quad
w=|\bfw|
\end{align}
We equip the $R_{\bfw,S}$-module $F_{w,R_{\bfw,S}}^\vee$ with the filtration induced by the $u$-filtration of
$F_{\bfw,R_{\bfw,S}}^\vee$ under the inclusion \eqref{F21}.
Let $\gr_\bfw(F_{w,R_{\bfw,S}}^\vee)$ denote the associated graded.
Next, the map \eqref{rhorhovee} yields a homomorphism of filtered modules
$$\rho_w^\vee:\bfY_{R_S}^\vee\to\End_{R_{\bfw,S}}(F_{w,R_{\bfw,S}}^\vee).$$
We consider its associated graded.
By base change from $R_S$ to $K_S$, we get the following maps
\begin{align*}
\rho_w^\vee&:\bfY_{K_S}^\vee\to\End_{R_{\bfw,S}}(F_{w,R_{\bfw,S}}^\vee)\otimes_{R_S}K_S,\\
\gr_\bfw\rho_w^\vee&:\bfU_{K_S}^\vee\to\End_{R_{\bfw,S}}(\gr_\bfw F_{w,R_{\bfw,S}}^\vee)\otimes_{R_S}K_S
\end{align*}

\begin{lemma}\label{lem:lastlast}
\hfill
\begin{enumerate}[label=$\mathrm{(\alph*)}$,leftmargin=8mm,itemsep=1.2mm]
\item Let  $x\in\bfU_{K_S}^\vee$ such that
$\gr_\bfw\rho_w^\vee(x)$ preserves $\gr_\bfw F_{w,R_{\bfw,S}}^\vee$ for all $\bfw\in\bfN_1$.
Then $x\in\bfU_{R_S}^\vee$.
\item Let $x\in\bfY_{K_S}^\vee$ such that
$\rho_w^\vee(x)$ preserves $F_{w,R_{\bfw,S}}^\vee$ for all $\bfw\in\bfN_1$.
Then $x\in\bfY_{R_S}^\vee$.
\end{enumerate}
\end{lemma}

\begin{proof}
The proof of (a) bears some similarity with that of Lemma~\ref{lem:last}(a). 
Set
$$\mathcal{U}=\bigcap_{\bfw \in \bfN_1}(\gr_\bfw\,\rho_w^\vee)^{-1}
(\End_{R_{\bfw,S}}(\gr_\bfw F_{w,R_{\bfw,S}}^\vee)).$$
It is an $R_S$-subalgebra of $\bfU_{K_S}^\vee$ containing $\bfU_{R_S}^\vee$. 
Hence, it is stable by left and right multiplication by elements of $\bfU_{R_S}^\vee$.
 Let $(x^{\pm}_n\,;\,n\in \bbN)$ be dual homogeneous $R_S$-bases of 
$\frakg_{R_S}^{\pm,\vee}$ as in the proof of Lemma \ref{lem:last}.
We set $v_n=\wt(x^+_n)$ and $\wt(x^-_n)=-v_n$. 
The triangular decomposition in \eqref{Upm0vee} yields
$$\bfU_{R_S}^{-,\vee}\otimes_{R_S}\bfU_{R_S}^{0,w,\vee}\otimes_{R_S}\bfU_{R_S}^{0,v,\vee}\otimes_{R_S}\bfU_{R_S}^{+,\vee}=\bfU_{R_S}^\vee.$$
The $R_S$-algebra $\bfU_{R_S}^{\pm,\vee}$ is spanned by monomials
$$x_{\bfn,\bfd}^\pm=x^\pm_{n_1}u^{d_1} \cdots  x^\pm_{n_s}u^{d_s}$$
where the tuple $(\bfn,\bfd)=(n_1, \ldots, n_s,d_1,\ldots,d_s)$ is ordered lexicographically 
$$(n_1,d_1) \leqsl (n_2,d_2) \leqsl \cdots\leqsl (n_s,d_s)$$
The $R_S$-algebras $\bfU_{R_S}^{0,w,\vee}$ and $\bfU_{R_S}^{0,v,\vee}$ are spanned by monomials
\begin{align*}
h^w_{\bfn^w,\bfd^w}&=b_{i_1}u^{d_1} \cdots  b_{i_s}u^{d_s},\\
h^v_{\bfn^v,\bfd^v}&=a_{i_1}u^{d_1} \cdots  a_{i_s}u^{d_s}
\end{align*}
with $(\bfn^w,\bfd^w)=(\bfn^v,\bfd^v)=(i_1, \ldots, i_s,d_1,\ldots,d_s).$
See Lemma \ref{lem:Y03} for more details. 
Let $l(\bfn)$ denote the length of a tuple $\bfn$, so that we have $l(n_1, \ldots, n_s)=s.$
We'll prove the inclusion 
$$\calU \subset \bfU_{R_S}^\vee$$ by contradiction. 
Fix an element $x\in\calU$. We write
$$x=\sum_{\bfn,\bfd} c_{\bfn,\bfd} \,x_{\bfn,\bfd}
,\quad
x_{\bfn,\bfd}=x^-_{\bfn^-,\bfd^-}\,h^w_{\bfn^w,\bfd^w}\,h^v_{\bfn^v,\bfd^v}\,x^+_{\bfn^+,\bfd^+}
$$
where 
\begin{gather*}
c_{\bfn,\bfd} \in K_S
,\quad
\bfn=(\bfn^-,\bfn^w,\bfn^v,\bfn^+)
,\quad
\bfd=(\bfd^-,\bfd^w,\bfd^v,\bfd^+).
\end{gather*}
Assume that not all coefficients $c_{\bfn,\bfd}$ belong to $R_S$. Subtracting all monomials with 
$c_{\bfn,\bfd} \in R_S$, we may assume that $c_{\bfn,\bfd} \notin R_S$ 
whenever $c_{\bfn,\bfd} \neq 0$. 
 Fix a tuple $\bfw\in\bfN_1$. Let $\phi_w^\vee$ be the vacuum vector in
 $\gr_\bfw F_{w,R_{\bfw,S}}^\vee$. Recall that $w=|\bfw|$. 
 By Lemma \ref{lem:Y03}, for $n,d\in\bbN$, $i\in I$ we have
\begin{gather*}
x(\phi_w^\vee)\in \gr_\bfw F_{w,R_{\bfw,S}}^\vee,\\
a_iu^{d+1}(\phi_w^\vee)=x_n^+u^d(\phi_w^\vee)=0.
 \end{gather*}
By Lemma~\ref{lem:last}, applying the evaluation map to $x$ yields
$$l(\bfn^v)=l(\bfn^+)=0\Rightarrow c_{\bfn,\bfd}=0.$$
Pick a tuple $( \bfn, \bfd)$ such that
\begin{itemize}[label=$\mathrm{(\alph*)}$,leftmargin=8mm,itemsep=1.2mm]
\item[-]
$c_{\bfn,\bfd} \neq 0$, 
\item[-]
the weight $\wt(x^+_{\bfn,\bfd})$ is minimal,
\item[-]
the length $l( \bfn^w)+l( \bfn^+)$ is maximal.
\end{itemize}
Set $s^w=l( \bfn^w)$ and $s=l( \bfn^+)$.
For weight reasons, for each tuple $(\bfn_1,\bfd_1)$ we have
$$ \wt(x^+_{\bfn^+,\bfd^+})<\wt(x^+_{\bfn_1^+,\bfd_1^+})\Rightarrow
\ad(x^-_{n^+_1}) \cdots  \ad(x^-_{n^+_s}) (x_{ \bfn_1, \bfd_1}) \in 
\bfU_{R_S}^\vee\bfU_{>0,R_S}^{+,\vee}.$$
By Lemmas \ref{lem:pairingg} and \ref{lem:Y02} , applying 
$\ad(x^-_{n^+_1}) \circ \cdots \circ \ad(x^-_{n^+_s})$ to $x^+_{ \bfn_1^+, \bfd_1^+}$ produces 
at most $s$ factors $b_i u^l$, and exactly $s$ only when $ \bfn^+= \bfn_1^+$.
We deduce that
$\ad(x^-_{n^+_1}) \cdots \ad(x^-_{n^+_s}) (x)$ is a sum of terms in 
$\bfU_{K_S}^\vee \bfU_{>0,K_S}^{+,\vee}$, 
plus terms of the form 
\begin{align}\label{47e}
c_{ \bfn_1, \bfd_1}\, x^-_{ \bfn_1^-,  \bfd_1^-}\,h_{ \bfn_1^w, \bfd_1^w}^w\,h_{ \bfn_1^v, \bfd_1^v}^v
,\quad
s^w+s>l( \bfn_1^w),
\end{align}
plus terms of the form
\begin{align}\label{47d}
c_{ \bfn_1, \bfd_1}\, x^-_{ \bfn_1^-,  \bfd_1^-}\,h_{ \bfn_1^w, \bfd_1^w}^w
\,h_{ \bfn_1^v, \bfd_1^v}^v\,(h_{n^+_1}u^{l_1}) \cdots (h_{n^+_s}u^{l_s}),
\quad
s^w=l(\bfn_1^w)
,\quad
\bfn^+= \bfn_1^+.
\end{align}
Compare the proof of Lemma~\ref{lem:last}.

We next apply a similar procedure to transform all the factors $a_iu^l$ into some $x^-_nu^k$. 
To do that, we pick among the terms in \eqref{47d} of the form
$$c_{\bfn_1,\bfd_1}x^-_{ \bfn_1^-,\bfd_1^-} h^{w}_{ \bfn_1^w,\bfd_1^w}h^v_{ \bfn_1^v,\bfd_1^v},$$ 
i.e., the terms having a maximal length $l(\bfn_1^w)$, one for which the length
$l(\bfn_1^-) + l( \bfn_1^v)$ is maximal.
Then, set $r=l( \bfn_1^v)$ and $t=l( \bfn_1^-) + l( \bfn_1^v)$. 
We apply  $\ad(x^-_{m_1}u^{p_1}) \circ \cdots \circ \ad(x^-_{m_r}u^{p_r})$ to
$\ad(x^-_{n^+_1}) \cdots \ad(x^-_{n^+_s}) (x)$, 
for generic tuples $\bfm=(m_1,\dots,m_r)$ and $\bfp=(p_1,\dots,p_r)$. 
As above, we obtain a sum of terms in
$\bfU_{K_S}^\vee \bfI$ where $\bfI$ is the augmentation ideal in $\bfU_{K_S}^{0,v,\vee}\bfU_{K_S}^{+,\vee}$,
plus terms of the form
$$c_{\bfn_2,\bfd_2}\, x^-_{ \bfn_2^-,\bfd_2^-}\,x^-_{\bfm,\bfq}\,h^w_{\bfn_2^w, \bfd_2^w},$$
plus terms of the form 
$$c_{\bfn_2,\bfd_2}\,x^-_{\bfn_2^-,  \bfd_2^-}\,h_{\bfn_2^w, \bfd_2^w}^w
,\quad
t>l(\bfn_2^-) \ \text{or}\  s^w+s>l(\bfn_2^w).$$ 
Therefore, after applying a sequence of operators $\ad(x_n^-u^l)$ to $x$ we have obtained an element 
$$x'=\sum_{\bfn,\bfd} c'_{\bfn,\bfd} \,x_{\bfn,\bfd}\in\calU$$
for which there exists a pair $(\bfn, \bfd)$ such that $l( \bfn^v)=l( \bfn^+)=0$ 
and  $c'_{\bfn,\bfd}\notin R_S$. 
This contradicts Lemma~\ref{lem:last}. This finishes the proof of Part (a).
	
\medskip

%%%%%%%%%%%%%%%%%%%%%%%%%%%%%%%%%%%

Now, we prove Part (b). We equip the $R_{\bfw,S}$-submodule 
$$(\rho_w^\vee)^{-1}(\End_{R_{\bfw,S}}(F_{w,R_{\bfw,S}}^\vee))\subset\bfY_{K_S}^\vee$$ 
with the filtration induced by the $u$-filtration. Let
$$\gr((\rho_w^\vee)^{-1}(\End_{R_{\bfw,S}}(F_{w,R_{\bfw,S}}^\vee)))\subset\bfU_{K_S}^\vee$$ 
be the associated graded. There is an obvious inclusion
$$\gr((\rho_w^\vee)^{-1}(\End_{R_{\bfw,S}}(F_{w,R_{\bfw,S}}^\vee)))\subset
(\gr_\bfw\rho_w^\vee)^{-1}(\End_{R_{\bfw,S}}(\gr_\bfw F_{w,R_{\bfw,S}}^\vee)).$$
Thus Part (a) implies that
$$\bigcap_{\bfw\in\bfN_1}\gr((\rho_w^\vee)^{-1}(\End_{R_{\bfw,S}}(F_{w,R_{\bfw,S}}^\vee)))
\subset\bfU_{R_S}^\vee$$
Since
$$\gr(\bigcap_{\bfw\in\bfN_1}(\rho_w^\vee)^{-1}(\End_{R_{\bfw,S}}(F_{w,R_{\bfw,S}}^\vee)))
\subset\bigcap_{\bfw\in\bfN_1}\gr((\rho_w^\vee)^{-1}(\End_{R_{\bfw,S}}(F_{w,R_{\bfw,S}}^\vee)))$$
we deduce that
\begin{align}\label{47a}
\gr(\bigcap_{\bfw\in\bfN_1}(\rho_w^\vee)^{-1}(\End_{R_{\bfw,S}}(F_{w,R_{\bfw,S}}^\vee)))
\subset\bfU_{R_S}^\vee
\end{align}
On the other hand, we have
\begin{align}\label{47b}
\bfY_{R_S}^\vee\subset\bigcap_{\bfw\in\bfN_1}(\rho_w^\vee)^{-1}(\End_{R_{\bfw,S}}(F_{w,R_{\bfw,S}}^\vee))
\end{align}
hence
\begin{align}\label{47c}
\bfU_{R_S}^\vee\subset\gr(\bigcap_{\bfw\in\bfN_1}(\rho_w^\vee)^{-1}(\End_{R_{\bfw,S}}(F_{w,R_{\bfw,S}}^\vee)))
\end{align}
From \eqref{47a}, \eqref{47c} we deduce that
$$\bfU_{R_S}^\vee=\gr(\bigcap_{\bfw\in\bfN_1}(\rho_w^\vee)^{-1}(\End_{R_{\bfw,S}}(F_{w,R_{\bfw,S}}^\vee))).$$
Comparing with \eqref{47b}, we deduce that
$$\bfY_{R_S}^\vee=\bigcap_{\bfw\in\bfN_1}(\rho_w^\vee)^{-1}(\End_{R_{\bfw,S}}(F_{w,R_{\bfw,S}}^\vee)),$$
proving the lemma.
\end{proof}

\smallskip

\begin{theorem}\label{thm:2}
\hfill
\begin{enumerate}[label=$\mathrm{(\alph*)}$,leftmargin=8mm,itemsep=1.2mm]
\item 
There is a $\bbZ I\times\bbZ$-graded $R_S$-linear isomorphism
$$\overline\Y_{R_S}^\vee=\bfY_{R_S}^\vee\,/\,\Ker(\ev_2^\vee).$$
\item
There is a $\bbZ I\times\bbZ$-graded $R_S$-linear isomorphism
$$\gr(\overline\Y_{R_S}^\vee)=\bfU_{R_S}^{-,\vee}\bfU_{R_S}^{0,w,\vee},$$
where the filtration on $\overline\Y_{R_S}^\vee$ is the pulbback by the isomorphism
$\operatorname{(a)}$ of the quotient of the $u$-filtration on $\bfY_{R_S}^\vee$.
\item
The map $\Phi$ is an $R_S$-algebra embedding $\overline\Y_{R_S}^\vee\to\bfY_{R_S}^\vee$.
\end{enumerate}
\end{theorem}

\begin{proof}
Composing the map $\Phi$ in \eqref{Phi} with the obvious projection yields the map
$$\Psi:\overline\Y_{K_S}^\vee\to\bfY_{K_S}^\vee\,/\,\Ker(\ev_1^\vee)_{K_S}$$
It fits into the following commutative square
\begin{align*}
\xymatrix{
\overline\Y_{K_S}^\vee\ar[r]^-\Psi\ar@{_{(}->}[d]_-{\ev_2^\vee}&\bfY_{K_S}^\vee\,/\,\Ker(\ev_1^\vee)_{K_S}
\ar@{^{(}->}[d]^-{\overline\ev_1^\vee}
\\
F_{2,K_S}^\vee\ar@{^{(}->}[r]&F_{1,K_S}^\vee}
\end{align*}
The injectivity of the left map is proved in Proposition \ref{prop:COHA}.
The lower map is the product of the stable envelopes
\begin{align*}
\stab^\vee(1):F_{w,R_{\bfw,S}}^\vee\to F_{\bfw,R_{\bfw,S}}^\vee
,\quad
w=|\bfw|
\end{align*}
in \eqref{F21} over all tuples $\bfw\in\bfN_1$.
It is homogeneous of degree 0 for the $\bbZ I\times\bbZ$-grading and
it commutes with the $\bfY^\vee$-action by  \eqref{stabvee}.
Hence Lemma \ref{lem:last} yields the following commutative square
\begin{align}\label{triangle2}
\begin{split}
\xymatrix{
\overline\Y_{R_S}^\vee\ar@{^{(}->}[r]^-\Psi\ar@{_{(}->}[d]_-{\ev_2^\vee}&\bfY^\vee_{R_S}\,/\,\Ker(\ev_1^\vee)
\ar@{^{(}->}[d]^-{\overline\ev_1^\vee}\\
F_{2,R_S}^\vee\ar@{^{(}->}[r]&F_{1,R_S}^\vee
}
\end{split}
\end{align}
The evaluation maps are homogeneous of degree 0 for the $\bbZ I\times\bbZ$-grading,
hence $\Psi$ is also homogeneous of degree 0.
We claim that it is surjective.
Let $v,w\in\bbN I$ and $l\in\bbN$.
We have
$$\Ker(\ev_1^\vee)\,,\,I_w\subset\Ker(\ev_w^\vee).$$
Hence the square \eqref{triangle2} factors through the commutative triangle
\begin{align}\label{triangle3}
\begin{split}
\xymatrix{
\overline\Y^\vee_{-v,\leqsl l,R_S}\,/\,(\overline\Y_{R_S}^\vee I_w\cap \overline\Y^\vee_{-v, \leqsl l,R_S})
\ar[r]^-\Psi\ar[rd]_-{\overline\ev_w^\vee}
&\bfY_{-v,\leqsl l,R_S}^\vee\,/\,(\Ker(\ev_1^\vee)+\bfY_{R_S}^\vee I_w)\cap 
\bfY_{-v, \leqsl l,R_S}^\vee\ar[d]^-{\overline\ev_w^\vee}
\\
&F_{v,w,\leqsl l,R_{w,S}}^\vee}
\end{split}
\end{align}
Assume that $w$ is large enough.
The left map is invertible by \eqref{YF}.
Further, Lemma  \ref{lem:39} yields the following commutative triangle
\begin{align}\label{triangle4}
\begin{split}
\xymatrix{
\bfY_{-v,\leqsl l,R_S}^\vee\,/\,(\Ker(\ev_1^\vee)+\bfY_{R_S}^\vee I_w)\cap \bfY_{-v,\leqsl l,R_S}^\vee
\ar[d]_-{\overline\ev_w^\vee}\ar@{=}[r]&\ar@{^{(}->}[ld]^-{\overline\ev_w^\vee}
\bfY_{-v,\leqsl l,R_S}^\vee/\Ker(\ev_w^\vee)\cap\bfY_{-v,\leqsl l,R_S}^\vee\\
F_{v,w,\leqsl l,R_{w,S}}^\vee&}
\end{split}
\end{align}
Hence \eqref{triangle3} and \eqref{triangle4} yield the following commutative triangle
\begin{align}\label{triangle5}
\begin{split}
\xymatrix{
\overline\Y^\vee_{-v,\leqsl l,R_S}\,/\,(\overline\Y_{R_S}^\vee I_w\cap \overline\Y^\vee_{-v,\leqsl l,R_S})
\ \ar[r]^-\Psi\ar@{=}[rd]_-{\ev_w^\vee}
&\bfY_{-v,\leqsl l,R_S}^\vee\,/\,\Ker(\ev_w^\vee)\cap \bfY_{-v,\leqsl l,R_S}^\vee\ar@{^{(}->}[d]^-{\overline\ev_w^\vee}\\
&F_{v,w,\leqsl l,R_{w,S}}^\vee}
\end{split}
\end{align}
We deduce that the map $\Psi$ in \eqref{triangle5} is invertible.
Summing over all weights $v$ and taking the cohomological degree $l$ to $\infty$, we deduce that 
the map $\Psi$ in \eqref{triangle2} is also invertible, proving Part (a) of the theorem.
Part (b) follows from Lemma \ref{lem:YU}, because $\Ker(\ev_1^\vee)=\Ker(\ev_2^\vee)$ by Lemma \ref{lem:KK}.
Proposition \ref{prop:Phi} yields a map $\Phi:\overline\Y_{R_S}^\vee\to \bfY_{K_S}^\vee$ such that
$$\rho_w^\vee\Phi(x)(y)=\rho_w^\vee(x)(y)
,\quad
x\in\overline\Y_{R_S}^\vee
,\, y\in F_{w,R_{\bfw,S}}^\vee
,\,
\bfw\in\bfN_1.$$
By Lemma \ref{lem:lastlast} we deduce that $\Phi(\overline\Y_{R_S}^\vee)\subset\bfY_{R_S}^\vee$,
proving Part (c), and therefore the theorem.
\end{proof}

\begin{remark}
The proof of the theorem implies that the $\bbZ I\times\bbZ$-graded $R_S$-modules
$\bfU_{R_S}^{-,\vee}$ and $\Y_{R_S}^\vee$ have the same graded dimension.
From Proposition \ref{prop:grading} we deduce that
$\bfU_{R_S}^{+}$ and $\Y_{R_S}^\vee$ also have the same graded dimension.
\end{remark}

\medskip

\subsection{The triangular decomposition of the Yangian}

We identify $\Y_{R_S}$ and $\Y_{R_S}^\vee$ with the subalgebras
$\Y_{R_S}\otimes 1$ and $\Y_{R_S}^\vee\otimes 1$ of $\overline\Y_{R_S}$ and $\overline\Y_{R_S}^\vee$
under the identification \eqref{oYY}.
Recall that the map $(-)^\T:\bfY_{R_S}\to\bfY_{R_S}^\vee$
is the transpose $R$-algebra anti-isomorphism in Lemma \ref{lem:FF}.
Let $(-)^\T$ denote also the inverse map $\bfY_{R_S}^\vee\to\bfY_{R_S}$.
We consider the $R$-subalgebras 
$\bfY_{R_S}^{\pm}\subset\bfY_{R_S}$ and  $\bfY_{R_S}^{\pm,\vee}\subset\bfY_{R_S}^\vee$ given by
\begin{align}\label{TDA}
\bfY_{R_S}^{-,\vee}=\Phi(\Y_{R_S}^\vee)
,\quad
\bfY_{R_S}^+=((\bfY_{R_S}^{-,\vee})^\T)^\op
\end{align}
and
\begin{align}\label{TDB}
\bfY_{R_S}^-=\bfY_{K_S}^{-,\vee}\cap\bfY_{R_S}
,\quad
\bfY_{R_S}^{+,\vee}=\bfY_{K_S}^+\cap\bfY_{R_S}^\vee,
\end{align}
where we abbreviate $\bfY_{K_S}^{\pm}=\bfY_{R_S}^{\pm}\otimes_{R_S}K_S$ and
$\bfY_{K_S}^{\pm,\vee}=\bfY_{R_S}^{\pm,\vee}\otimes_{R_S}K_S$.
We equip $\bfY_{R_S}^{\pm}$ and  $\bfY_{R_S}^{\pm,\vee}$ with the filtrations induced by the $u$-filtrations of 
$\bfY_{R_S}$ and $\bfY_{R_S}^\vee$. 
Let  $\gr(\bfY_{R_S}^{\pm})$ and  $\gr(\bfY_{R_S}^{\pm,\vee})$ be the associated graded
$R_S$-algebras.
By Lemmas \ref{lem:Delta}, \ref{lem:Deltavee} we have $\bbZ I\times\bbZ$-graded $R_S$-bialgebra isomorphisms
\begin{align}\label{lastisom}
\gr(\bfY_{R_S})=\bfU_{R_S}
,\quad
\gr(\bfY_{R_S}^\vee)=\bfU_{R_S}^\vee
\end{align}

\begin{proposition}\label{prop:TYMO}
The $R_S$-algebra isomorphisms in \eqref{lastisom} restrict to $R_S$-algebra isomorphisms 
$$\gr(\bfY_{R_S}^{\pm})=\bfU_{R_S}^{\pm}
,\quad
\gr(\bfY_{R_S}^{\pm,\vee})=\bfU_{R_S}^{\pm,\vee}.$$
\end{proposition}

To prove the proposition we 
consider the decomposition into weight subspaces 
$
\bfU_{R_S}^\vee=\bigoplus_{v\in\bbZ I}\bfU_{v,R_S}^\vee,$
and we define
$$\bfU_{\leqsl 0,R}^\vee=\bigoplus_{v\in-\bbN I}\bfU_{v,R_S}^\vee.$$
Similarly, for each tuple $w\in\bbN I$, let $A_{\leqsl 0,w,R_{w,S}}\subset A_{w,R_{w,S}}$ 
be the subspace spanned by all homogeneous elements of weight in $-\bbN I$
and set
$$A_{\leqsl 0,a,R_S}=\prod_{\bfw\in\bfN_a}A_{\leqsl 0,\bfw,R_{\bfw,S}}
,\quad
A_{\leqsl 0,\bfw,R_{\bfw,S}}=\bigotimes_{r=1}^sA_{\leqsl0,w_r,R_{w_r}}
,\quad
\bfw=(w_1,w_2,\dots,w_s)\in\bfN_a.$$
Let $\Delta$ be the trivial coproduct of  the enveloping algebra $\bfU_{R_S}^\vee$.
The total weight of tensor monomial in $(\bfU_{R_S}^\vee)^{\otimes s}$ of weight vectors of $\bfU_{R_S}^\vee$ 
is the sum of the weights of its factors.
We'll say that an element in $(\bfU_{R_S}^\vee)^{\otimes s}$ has negative total weight
if it is a sum of tensor monomials of weight vectors of $\bfU_{R_S}^\vee$ 
whose total weight belongs to $\bbZ I\setminus\bbN I$.
By \eqref{grrho1} there is an $R_S$-algebra homomorphism
$$\gr\rho_1^\vee:\bfU_{R_S}^\vee\to A_{1,R_S}.$$

\begin{lemma}\label{lem:L1}
\hfill
\begin{enumerate}[label=$\mathrm{(\alph*)}$,leftmargin=8mm,itemsep=1.2mm]
\item A non zero element $x\in\bfU_{R_S}^\vee$ belongs to the augmentation ideal of $\bfU_{R_S}^{-,\vee}$
if and only if, for each $s>0$, the element
$\Delta^{s-1}(x)$ belongs to $(\bfU_{\leqsl 0,R_S}^\vee)^{\otimes s}$ and has negative total weight.
\item
Under the obvious inclusion $\bfU_{\leqsl 0,R_S}^\vee\subset\bfU_{R_S}^\vee$ we have
$\bfU_{\leqsl 0,R_S}^\vee=(\gr\rho_1^\vee)^{-1}(A_{\leqsl 0,1,R_S}).$
\end{enumerate}
\end{lemma}

\begin{proof}
To prove Part (a), we'll prove the stronger result that
$x$ belongs to the augmentation ideal of $\bfU_{R_S}^{-,\vee}$ if and only if
$\Delta(x)\in(\bfU_{\leqsl 0,R_S}^\vee)^{\otimes 2}$ and has negative total weight.
To do this, it is enough to prove that if $\Delta(x)\in(\bfU_{\leqsl 0,R_S}^\vee)^{\otimes 2}$ then
$x\in\bfU_{R_S}^{-,\vee}\bfU_{R_S}^{0,\vee}$, because the other statements are immediate.
Fix an $R_S$-basis $(x_k)$ of $\frakg_{R_S}^\vee$ consisting of weight vectors such that
$$\wt(x_{k_1})>\wt(x_{k_2})\Rightarrow k_1>k_2.$$
A PBW monomial in $\bfU_{R_S}^\vee$ is a tensor of the form
$$x_\bfk=x_{k_1}\otimes x_{k_2}\otimes\cdots\otimes x_{k_s}
,\quad
\bfk=(k_1,k_2,\dots,k_s) .$$
Fix any element $x\in\bfU_{R_S}^\vee$ and write
$$x=\sum_\bfk\alpha_\bfk\,x_\bfk
,\quad
\alpha_\bfk\in R_S.$$
Well factorize each monomial 
$$x_\bfk=x_\bfk^- x_\bfk^0 x_\bfk^+
,\quad
x_\bfk^\flat\in\bfU_{R_S}^{\flat,\vee}
,\quad
\flat=\pm, 0.$$
The coproduct of $x$ in $(\bfU_{R_S}^\vee)^{\otimes 2}$
$$\Delta(x)=\sum_\bfk\alpha_\bfk\Delta(x_\bfk^-)\Delta(x_\bfk^0)\Delta(x_\bfk^+)$$
decomposes as a sum of homogeneous elements
for the $(\bbZ I)^2$-grading in the obvious way. Let 
$$x'=\sum_\bfk\alpha_\bfk x'_\bfk
,\quad
x'_\bfk=(x_\bfk^-\otimes 1)\Delta(x_\bfk^0)(1\otimes x_\bfk^+)$$
be the sum of all monomials whose second member has a maximal weight.
One can check that
\begin{itemize}[label=$\mathrm{(\alph*)}$,leftmargin=8mm,itemsep=1.2mm]
\item[-]
$\Delta(x)\in(\bfU_{\leqsl 0,R_S}^\vee)^{\otimes 2}\Rightarrow x'\in(\bfU_{\leqsl 0,R_S}^\vee)^{\otimes 2}$,
\item[-]
the elements $x'_\bfk$ are homogeneous and linearly independent,
\item[-]
$x'_\bfk\in(\bfU_{\leqsl 0,R_S}^\vee)^{\otimes 2}\Rightarrow x_\bfk^+\in R_S$.
\end{itemize}
Applying successively the 3 claims above, we deduce that
\begin{align*}
\Delta(x)\in(\bfU_{\leqsl 0,R_S}^\vee)^{\otimes 2}
&\Rightarrow x'\in(\bfU_{\leqsl 0,R_S}^\vee)^{\otimes 2},\\
&\Rightarrow x'_\bfk\in(\bfU_{\leqsl 0,R_S}^\vee)^{\otimes 2},\\
&\Rightarrow x_\bfk^+\in R_S,\\
&\Rightarrow x\in\bfU_{R_S}^{-,\vee}\bfU_{R_S}^{0,\vee}.
\end{align*}
Part (b) follows from the embedding $\bfU_{R_S}^\vee\subset A_{1,R_S}$, 
see Lemma \ref{lem:rhovee}.
\end{proof}

\begin{lemma}\label{lem:L2}
Fix $v\in\bbN I$ and $x\in\Y_{-v,R_S}^\vee$.
Fix a tuple $\bfw\in\bfN_1$ of length $s$ with $|\bfw|=w$.
\hfill
\begin{enumerate}[label=$\mathrm{(\alph*)}$,leftmargin=8mm,itemsep=1.2mm]
\item There is an operator $\rho_\bfw(x)\in A_{\leqsl0,\bfw,K_{\bfw,S}}$ making the following diagram to commute
\begin{align*}
\begin{split}
\xymatrix{
F_{w,R_{\bfw,S}}\ar[rr]^-{\res}&&F_{\bfw,K_{\bfw,S}}\\
F_{w,R_{\bfw,S}}\ar[rr]^-{\res}\ar[u]^-{\rho_w(x)}&&F_{\bfw,K_{\bfw,S}}\ar[u]_-{\rho_\bfw(x)}}
\end{split}
\end{align*}
where $\res$ is the Gysin restriction to $\frakM(\bfw)\subset\frakM(w)$.
\item 
There is an operator $y\in A_{\leqslant 0,\bfw,R_{\bfw,S}}$ making the following diagram to commute
\begin{align*}
\begin{split}
\xymatrix{
F_{\bfw,R_{\bfw,S}}\ar[rr]^-{\stab(1)}&&F_{w,R_{\bfw,S}}\\
F_{\bfw,R_{\bfw,S}}\ar[rr]^-{\stab(1)}\ar[u]^-{y}&&F_{w,R_{\bfw,S}}\ar[u]_-{\rho_w(x)}}
\end{split}
\end{align*}
modulo lower terms for the $u$-filtration of $F_{\bfw,R_{\bfw,S}}$.
\end{enumerate}
\end{lemma}

\begin{proof}
Composing the Gysin restriction with the multiplication by the inverse of some equivariant Euler class 
as in \eqref{res} yields
a $K_{\bfw,S}$-algebra homomorphism 
$$\bfr:A_{w,K_{\bfw,S}}\to A_{\bfw,K_{\bfw,S}}.$$
Setting $\rho_\bfw(x)=\bfr(x)$, we get an operator in $A_{\leqsl0,\bfw,K_{\bfw,S}}$ making the diagram in Part (a)
to commute. To prove Part (b), note first that, by definition of the map $\Phi$ in \eqref{Phi}, 
the actions of $x$ and $\Phi(x)$ on $F_{w,R_{\bfw,S}}$ coincide. Hence, by \eqref{stab}, setting
$y=\rho_\bfw(\Delta^{s-1}\Phi(x))$ we get an operator in $A_{\bfw,R_{\bfw,S}}$ with
a commutative diagram
\begin{align*}
\begin{split}
\xymatrix{
F_{\bfw,R_{\bfw,S}}\ar[rr]^-{\stab(1)}&&F_{w,R_{\bfw,S}}\\
F_{\bfw,R_{\bfw,S}}\ar[rr]^-{\stab(1)}\ar[u]^-{y}&&F_{w,R_{\bfw,S}}\ar[u]_-{\rho_w(x)}}
\end{split}
\end{align*}
Composing the commutative squares in Parts (a) and (b), we get the diagram
\begin{align*}
\begin{split}
\xymatrix{
F_{\bfw,R_{\bfw,S}}\ar[rr]^-{\res\,\circ\,\stab(1)}&&F_{\bfw,K_{\bfw,S}}\\
F_{\bfw,R_{\bfw,S}}\ar[rr]^-{\res\,\circ\,\stab(1)}\ar[u]^-{y}&&F_{\bfw,K_{\bfw,S}}\ar[u]_-{\rho_\bfw(x)}}
\end{split}
\end{align*}
The map $\res\,\circ\,\stab(1)$ is the diagonal modulo lower terms for the $u$-filtration, 
by the axiom (c) in the definition of the stable envelope in \S\ref{sec:stable}.
We deduce that $y\in A_{\leqslant 0,\bfw,R_{\bfw,S}}$ modulo lower terms for the $u$-filtration. 
This proves Part (b).
\end{proof}

\begin{proof}[Proof of Proposition $\ref{prop:TYMO}$]
By Theorem \ref{thm:2}(b) and \eqref{TDA}, there  are $\bbZ I\times\bbZ$-graded $R_S$-linear isomorphisms
\begin{align}\label{YYU}\gr(\bfY_{R_S}^{-,\vee})=\gr(\Y_{R_S}^\vee)=\bfU_{R_S}^{-,\vee}.\end{align}
By Lemmas \ref{lem:L1} and \ref{lem:L2}(b),
the isomorphism $\gr(\bfY_{R_S}^\vee)=\bfU_{R_S}^\vee$
in Lemma \ref{lem:Deltavee} restricts to an $R_S$-algebra embedding
$$\gr(\bfY_{R_S}^{-,\vee})\subset\bfU_{R_S}^{-,\vee}.$$
Thus, it yields indeed an $R_S$-algebra isomorphism
\begin{align}\label{isomA}
\gr(\bfY_{R_S}^{-,\vee})=\bfU_{R_S}^{-,\vee}.
\end{align}

Applying the transpose map $(-)^\T:\bfY_{R_S}^\vee\to\bfY_{R_S}$ to \eqref{isomA}, we deduce that
the isomorphism $\gr(\bfY_{R_S})=\bfU_{R_S}$
in Lemma \ref{lem:Delta} restricts to an $R_S$-algebra isomorphism
\begin{align}\label{isomB}
\gr(\bfY_{R_S}^+)=\bfU_{R_S}^+.
\end{align}

Since $\bfY_{K_S}^+=\bfY_{K_S}^{+,\vee}$ by \eqref{TDB}, and since $\bfU_{K_S}^+=\bfU_{K_S}^{+,\vee}$, from \eqref{isomB} 
and base change we deduce that the isomorphism $\gr(\bfY_{K_S})=\bfU_{K_S}$ restricts to a $K$-algebra isomorphism
\begin{align}\label{YU+}
\gr(\bfY_{K_S}^{+,\vee})=\bfU_{K_S}^{+,\vee}.
\end{align}
We must check that this isomorphism takes the $R_S$-submodule
$\gr(\bfY_{R_S}^{+,\vee})$ onto $\bfU_{R_S}^{+,\vee}.$
The isomorphism $\gr(\bfY_{R_S}^\vee)=\bfU_{R_S}^\vee$ in Lemma \ref{lem:Deltavee} and \eqref{YU+}
yield an inclusion
\begin{align}\label{embA}\gr(\bfY_{R_S}^{+,\vee})\subset\bfU_{K_S}^{+,\vee}\cap\bfU_{R_S}^\vee=\bfU_{R_S}^{+,\vee}.
\end{align}

Next, we claim that the algebra embedding $\Phi$ in \eqref{Phi}
restricts to an $R_S$-algebra embedding 
\begin{align}\label{YY-}
\Y_{R_S}\subset\bfY_{R_S}^-.
\end{align} 
By \eqref{TDA} and \eqref{TDB} we have
\begin{align}\label{YYT}
\bfY_{R_S}^-=(\bfY_{R_S}^{+,\vee})^\T.
\end{align}
Similarly, by definition of the transpose map in Lemma \ref{lem:FF}, we have
\begin{align}\label{UUT}\bfU_{R_S}^-=(\bfU_{R_S}^{+,\vee})^\T.\end{align}
Further, by \eqref{Okounkov} and Proposition~\ref{prop:grading},
the $\bbZ I\times\bbZ$-graded $R_S$-linear isomorphism 
$\gr(\Y_{R_S}^\vee)=\bfU_{R_S}^{-,\vee}$ in \eqref{YYU} yields a 
$\bbZ I\times\bbZ$-graded $R_S$-linear isomorphism
\begin{align}\label{YU-}\gr(\Y_{R_S})=\bfU_{R_S}^-.\end{align}
The relations \eqref{YY-} to \eqref{YU-} yield a 
$\bbZ I\times\bbZ$-graded $R_S$-module embedding
\begin{align}\label{embB}
\bfU_{R_S}^{+,\vee}\subset\gr(\bfY_{R_S}^{+,\vee}).
\end{align}
Comparing \eqref{embA} and \eqref{embB}, we deduce that the isomorphism $\gr(\bfY_{R_S}^\vee)=\bfU_{R_S}^\vee$ 
in Lemma \ref{lem:Deltavee} restricts to an $R_S$-algebra isomorphism
\begin{align}\label{isomC}
\gr(\bfY_{R_S}^{+,\vee})=\bfU_{R_S}^{+,\vee}.
\end{align}
Finally, applying the transpose map to \eqref{isomC} we deduce that
\begin{align}\label{isomD}
\gr(\bfY_{R_S}^-)=\bfU_{R_S}^-.
\end{align}

Now, let us prove the claim. To do that, let $x\in\Y_{R_S}$. 
By \eqref{TDA} we have $\Phi(x)\in\bfY_{K_S}^-$.
For any $w\in\bbN I$ the actions of $x$ and $\Phi(x)$ on the $R_{w,S}$-module $F_{w,R_{w,S}}\otimes_{R_S}K_S$
coincide by Proposition \ref{prop:Phi}. Hence $\Phi(x)$ preserves the $R_{w,S}$-submodule $F_{w,R_{w,S}}$.
Applying the transpose, we deduce that $\Phi(x)^\T\in\bfY_{K_S}^{+,\vee}$ and preserves
the $R_{w,S}$-module $F_{w,R_{w,S}}^\vee$ for each $w\in\bbN I$.
Hence $\Phi(x)^\T\in\bfY_{R_S}^\vee$ by Lemma \ref{lem:lastlast}.
We deduce that
$$\Phi(x)^\T\in\bfY_{K_S}^{+,\vee}\cap\bfY_{R_S}^\vee=\bfY_{R_S}^{+,\vee},$$
and, thus, that $\Phi(x)\in\bfY_{R_S}^-$ by \eqref{YYT}, proving the claim.
\end{proof}

\begin{theorem}\label{thm:last}
\hfill
\begin{enumerate}[label=$\mathrm{(\alph*)}$,leftmargin=8mm,itemsep=1.2mm]
\item
The multiplication yields the isomorphisms
\begin{align*}
\bfY_{R_S}^-\otimes_{R_S}\bfY_{R_S}^0\otimes_{R_S}\bfY_{R_S}^+&=\bfY_{R_S},\\
\bfY_{R_S}^{-,\vee}\otimes_{R_S}\bfY_{R_S}^{0,\vee}\otimes_{R_S}\bfY_{R_S}^{+,\vee}&=\bfY_{R_S}^\vee
\end{align*}
\item
The map $\Phi$ yields $R_S$-algebra isomorphisms
$\Y_{R_S}\to\bfY_{R_S}^-$ and
$\Y_{R_S}^\vee\to\bfY_{R_S}^{-,\vee}$.
\end{enumerate}
\end{theorem}

\begin{proof}
The theorem follows from Lemma \ref{lem:Y01} and Proposition \ref{prop:TYMO}.
\end{proof}

\bigskip

\bigskip

\centerline{\textbf{Acknowledgements}}

We would like to thank B. Davison, L. Hennecart and A. Negut for enlightening discussions and exchanges 
concerning COHAs. A. Negut told us that another proof of a localized version of 
Theorem~\ref{thm:mainintro}(a) is possible along the lines of \cite{N23a}. The research of O. S. was partially funded by the 
PNRR grant CF 44/14.11.2022 \textit{Cohomological Hall algebras of smooth surfaces and applications}.

\end{document}